\documentclass{amsart}

\usepackage[dvips]{graphics}
\usepackage{amscd}

\usepackage{color}
\usepackage{ifthen}
\usepackage{soul}
\usepackage{latexsym}
\usepackage{amsthm,amssymb}
\usepackage{amsbsy,amsfonts,amsmath}
\usepackage{txfonts}

\usepackage{ulem}

\numberwithin{equation}{section}

\newtheorem{theorem}{Theorem}[section]

\newtheorem{lemma}[theorem]{Lemma}

\theoremstyle{definition}
\newtheorem{definition}[theorem]{Definition}

\theoremstyle{remark}
\newtheorem{remark}[theorem]{Remark}

\newcommand{\rmSpan}
{\mathrm{Span}}
\newcommand{\rrmSpan}[2]{\rmSpan_{#1}(#2)}

\newcommand{\al}{\alpha}
\newcommand{\fkJ}{J}

\newcommand{\bN}{{\mathbb{N}}}
\newcommand{\bZ}{{\mathbb{Z}}}
\newcommand{\bZgeqo}{\bZ_{\geq 0}}
\newcommand{\bQ}{{\mathbb{Q}}}
\newcommand{\bR}{{\mathbb{R}}}
\newcommand{\bC}{{\mathbb{C}}}
\newcommand{\bCt}{\bC^\times}

\newcommand{\bA}{{\mathbb{A}}}

\newcommand{\lb}{\llbracket}
\newcommand{\rb}{\rrbracket}

\newcommand{\blb}{{\bar{\llbracket}}}
\newcommand{\brb}{{\bar{\rrbracket}}}

\newcommand{\ha}{a}
\newcommand{\bha}{{\bar a}}

\newcommand{\EQone}{({\rm{EQ}}1)}
\newcommand{\EQtwo}{({\rm{EQ}}2)}
\newcommand{\EQthree}{({\rm{EQ}}3)}
\newcommand{\EQfour}{({\rm{EQ}}4)}
\newcommand{\EQfive}{({\rm{EQ}}5)}
\newcommand{\EQsix}{({\rm{EQ}}6)}
\newcommand{\EQseven}{({\rm{EQ}}7)}
\newcommand{\EQeight}{({\rm{EQ}}8)}
\newcommand{\EQnine}{({\rm{EQ}}9)}
\newcommand{\EQten}{({\rm{EQ}}10)}

\newcommand{\FQone}{({\rm{FQ}}1)}
\newcommand{\FQtwo}{({\rm{FQ}}2)}
\newcommand{\FQthree}{({\rm{FQ}}3)}
\newcommand{\FQfour}{({\rm{FQ}}4)}
\newcommand{\FQfive}{({\rm{FQ}}5)}
\newcommand{\FQsix}{({\rm{FQ}}6)}
\newcommand{\FQseven}{({\rm{FQ}}7)}
\newcommand{\FQeight}{({\rm{FQ}}8)}
\newcommand{\FQnine}{({\rm{FQ}}9)}
\newcommand{\FQten}{({\rm{FQ}}10)}

\newcommand{\mcT}{{\mathcal{T}}}

\newcommand{\trU}{U}
\newcommand{\trUqQ}{\trU_{q,Q}}

\newcommand{\trUp}{\trU^+}

\newcommand{\trUm}{\trU^-}
\newcommand{\trUo}{\trU^0}

\newcommand{\trUpgeq}{\trU^{+,\geq}}
\newcommand{\trUmleq}{\trU^{-,\leq}}
\newcommand{\trUogeq}{\trU^{0,\geq}}
\newcommand{\trUoleq}{\trU^{0,\leq}}

\newcommand{\trK}{K}
\newcommand{\trL}{L}

\newcommand{\trE}{E}
\newcommand{\trF}{F}

\newcommand{\tE}{{\tilde{\trE}}}
\newcommand{\tF}{{\tilde{\trF}}}

\newcommand{\dF}{{\dot{\trF}}}

\newcommand{\dtF}{{\dot{\tF}}}

\newcommand{\mcM}{{\mathcal{M}}}
\newcommand{\mcMtrU}{{\mcM\trU}}

\newcommand{\rmid}{{\mathrm{id}}}

\newcommand{\udK}[1]{{\underline{\trK}}_{\,#1}}
\newcommand{\udL}[1]{{\underline{\trL}}_{\,#1}}

\newcommand{\dudK}[1]{{\underline{{\dot{\trK}}}}_{\,#1}}
\newcommand{\dudL}[1]{{\underline{{\dot{\trL}}}}_{\,#1}}

\newcommand{\udUo}{{\underline{\trU}}^0}
\newcommand{\udUogeq}{{\underline{\trU}}^{0,\geq}}
\newcommand{\udUoleq}{{\underline{\trU}}^{0,\leq}}

\newcommand{\udzeta}{{\underline{\zeta}}}

\newcommand{\uudK}[1]{\uuline{\trK}_{\,\raisebox{4pt}{\scriptsize #1}}}

\newcommand{\uudL}[1]{\uuline{\trL}_{\,\raisebox{4pt}{\scriptsize #1}}}

\newcommand{\dudUo}{{\underline{{\dot{\trU}}}}^0}
\newcommand{\dudUogeq}{{\underline{{\dot{\trU}}}}^{0,\geq}}
\newcommand{\dudUoleq}{{\underline{{\dot{\trU}}}}^{0,\leq}}

\newcommand{\udI}{{\underline{\mathcal{I}}}}

\newcommand{\duudK}[1]{\uuline{{\dot{\trK}}}_{\raisebox{4pt}{\,\scriptsize #1}}}
\newcommand{\duudL}[1]{\uuline{{\dot{\trL}}}_{\raisebox{4pt}{\,\scriptsize #1}}}

\newcommand{\dudmcRo}{{\underline{{\mathcal{R}}}}^0}

\newcommand{\trkappa}{N}
\newcommand{\tromega}{\omega}

\newcommand{\trRuvdotdelalzero}{{\mathcal{R}}_{\bullet\delta+\al_0}(u,v)}
\newcommand{\trRuvdotdelalone}{{\mathcal{R}}_{\bullet\delta+\al_1}(u,v)}
\newcommand{\trRuvdotdel}{{\mathcal{R}}_{\bullet\delta}(u,v)}

\newcommand{\trRuvdotdelali}{{\mathcal{R}}_{\bullet\delta+\al_i}(u,v)}

\newcommand{\trmcU}{{\mathcal{U}}}
\newcommand{\trmcUo}{\trmcU^0}
\newcommand{\trmcUp}{\trmcU^+}
\newcommand{\trmcUm}{\trmcU^-}
\newcommand{\trmcUpgeq}{\trmcU^{+,\geq}}
\newcommand{\trmcUmleq}{\trmcU^{-,\leq}}
\newcommand{\trmcUogeq}{\trmcU^{0,\geq}}
\newcommand{\trmcUoleq}{\trmcU^{0,\leq}}

\newcommand{\trmcJ}{{\mathcal{J}}}
\newcommand{\trmcJo}{\trmcJ^0}
\newcommand{\trmcJp}{\trmcJ^+}
\newcommand{\trmcJm}{\trmcJ^-}
\newcommand{\trmcJpgeq}{\trmcJ^{+,\geq}}
\newcommand{\trmcJmleq}{\trmcJ^{-,\leq}}
\newcommand{\trmcJogeq}{\trmcJ^{0,\geq}}
\newcommand{\trmcJoleq}{\trmcJ^{0,\leq}}

\newcommand{\trmfkU}{{\mathfrak{U}}}
\newcommand{\trmfkUo}{\trmfkU^0}
\newcommand{\trmfkUp}{\trmfkU^+}
\newcommand{\trmfkUm}{\trmfkU^-}
\newcommand{\trmfkUpgeq}{\trmfkU^{+,\geq}}
\newcommand{\trmfkUmleq}{\trmfkU^{-,\leq}}
\newcommand{\trmfkUogeq}{\trmfkU^{0,\geq}}
\newcommand{\trmfkUoleq}{\trmfkU^{0,\leq}}

\newcommand{\trmfkC}{{\mathfrak{C}}}
\newcommand{\trmfkH}{{\mathfrak{H}}}

\newcommand{\trmfkR}{{\mathfrak{R}}}

\newcommand{\hattrmfkU}{{\hat{\trmfkU}}}
\newcommand{\hattrmfkC}{{\hat{\trmfkC}}}

\newcommand{\trp}{{\mathfrak{p}}}

\title{Universal $R$-matrix of double parameter quantum affine algebra $\trUqQ({\hat {sl_2}})$}
\author{Fengchang Li}
\address{School of Mathematical Sciences, University of Science and Technology of China, Jinzhai Road 96, Hefei 
\mbox{\,\,\,\,\,\,\,}  230026, Anhui Province, P.R.China.}
\email{fengchangli@ustc.edu.cn}
\author{Masatake Maruyama}
\address{Graduate School of Science and Engineering,
University of Toyama,
3190 Gofuku, Toyama 930-8555, Japan.}
\email{d26c2002@ems.u-toyama.ac.jp}
\author{Hiroyuki Yamane}
\address{Faculty of Science, Academic Assembly,
University of Toyama,
3190 Gofuku, Toyama 930-8555, Japan.}
\email{hiroyuki@sci.u-toyama.ac.jp}

\begin{document}

\begin{abstract}
We give the explicit formula of the 
universal $R$-matrix of a double parameter 
(or two-parameter, or multi-parameter) quantum affine algebra
of type ${\mathrm{A}}_1^{(1)}$.
For $N$ with $q_{00}q_{01}$ being a primitive $N$-th root of unity,
we introduce its $2N$-dimensional representation
and explicitly calculate the $R$-matrix associated with it via 
the universal $R$-matrix.
\end{abstract}

\maketitle

\section{Introduction}
Solutions of the Yang-Baxter equations have played an essential role in physics and mathematics, such as statistical mechanics, knot theory, conformal field theory, etc (see \cite{YANGGE1994}, \cite{RT91}, \cite{K87} etc).  
A solution of a Yang-Baxter equation is called an {\it{$R$-matrix}}.

Since the introduction of quantum groups by Drinfeld~\cite{D86} and Jimbo~\cite{J86} in the 1980s, the universal $R$-matrices~\cite{D86} have continued to play
a central role in numerous areas of mathematics and physics.
For the finite dimensional representation $\sigma$ of a quantum group $U_q({\mathfrak{g}})$, the associated 
$R$-matrix is given by 
$(\sigma\otimes\sigma)({\mathcal{R}})$, the image of the universal $R$-matrix ${\mathcal{R}}$. However, in general, it is difficult to obtain an explicit closed expression for 
$(\sigma\otimes\sigma)({\mathcal{R}})$.
The well-known alternative method is the Jimbo's intertwiner method \cite{J86}. 

Variations of quantum groups,
such as multi-parameter quantum groups \cite{JNY18, OY91, R90}, quantum superalgebras \cite{ HY10, LYZ23, LSS93, Y94, Y99},  and the Drinfeld doubles of Nichols algebras of diagonal type \cite{AY15, AYY15, BY18, Hec09}, have also been studied. 

We give the explicit formula (see Theorem~\ref{theorem:maintwo}) of the universal $R$-matrix $\trmfkR$ of the double parameter affine quantum group $\trUqQ({\hat {sl_2}})$,
where the $2\times 2$-matrix 
$Q=
\left(\begin{array}{cc} 
q_{00} & q_{01} \\ q_{10} & q_{11} 
\end{array}\right)=
\left(\begin{array}{cc} 
q^2  & a \\ q^{-4}a^{-1} & q^2 
\end{array}\right)$ 
with $a\in\bCt$
(see \eqref{eqn:defQ}). For $a\cdot q^2$ being the primitive $\trkappa$-th root of unity, we introduce the  representation $\rho:\trUqQ({\hat {sl_2}})\to {\mathrm{M}}_{2\trkappa}(\bC)$ and clearly decide
${\mathcal{R}}(z)=(\rho\otimes\rho)(\trmfkR)$ (see Theorem~\ref{theorem:MainVrepRMat}),
where if $N\geq 2$, the Jimbo's intertwiner method cannot be applied.
The restriction of 
$\rho$ to $U_q(sl_2)$ is isomorphic to the direct sum of $N$-copies of the irreducible two-dimensional $U_q(sl_2)$-representation
(or the vector representation of $U_q(sl_2)$).
If $N=1$, $\trUqQ({\hat {sl_2}})$ is the usual $\trU_q({\hat {sl_2}})$,
$\rho$ is the  vector representation ($2$-dimensional), and
${\mathcal{R}}(z)$ is the $R$-matrix ($4\times 4$-matrix)
of the six‑vertex model. 
More precisely, as shown by \cite{LSS93},
this ${\mathcal{R}}(z)$ (for $N=1$) is obtained from 
the standard six-vertex model $R$-matrix by  multiplication with a 
complex regular map ${\mathcal{A}}(z)=\exp\left(-\sum\limits_{k=1}^\infty
{\frac {z^k(q^k-q^{-k})} {k(q^k+q^{-k})}}\right)$.
If $N=2$, the matrix ${\mathcal{R}}(z)$ 
($16\times 16$-matrix)
depends on two complex regular maps
${\mathcal{B}}(z)$
and ${\mathcal{C}}(z)$ in a rather intricate way. Nevertheless, explicit computation (performed using Mathematica~14.3~\cite{Wolf14-3}) shows that these functions can be treated as completely independent constants (see Theorem~\ref{theorem:YBERtwoz}).

In \cite{LSS93}, also in \cite{D98}, the explicit formula of the universal $R$-matrix of $U_q(\hat{sl_2})$ is obtained strictly applying the Drinfeld quantum double method.
In addition, in \cite{LSS93}, the image of the vector representation of $U_q(\hat{sl_2})$ is obtained (it is the same as ${\mathcal{R}}(z)$ for $N=1$). In this paper, as mentioned in the previous paragraph, we establish for $\trUqQ({\hat {sl_2}})$ the counterparts of these two results.  
Except for Section~\ref{section:hadic},
we avoid a Drinfeld-type $h$-adic topological setting for the description of (the Cartan part of) the universal $R$-matrix. 
Instead, we use an argument where $q$ is a root of unity.

This paper is organized as follows.
In Section~2, we introduce the definition of $\trUqQ({\hat {sl_2}})$.
In Section~3, we analyze the
positive part $\trUp$ of $\trUqQ({\hat {sl_2}})$.
We prove the commutativity of 
the imaginary root vectors $\trE_{n\delta}$ in a direct way. This commutativity was originally proved in \cite{D93} for $U_q(\hat{sl_2})$. Remark~\ref{remark:CndEQseven} says that  $\trE_{m\delta}\trE_{n\delta}=\trE_{n\delta}\trE_{m\delta}$
is mainly due to $(q^{2(m+n)}-1)(q^{2(m+n-2)}+1)\ne 0$. On the other hand, \cite{D93} used a theological way (due to the non-degeneracy of the Drinfeld pairing) to show the same equations. We obtain the modification $\tE_{n\delta}$ of $\trE_{n\delta}$ in a well-known Drinfeld argument involving $\exp$ and $\log$ maps, see \cite{D87,B94}. 
In Section~4, we introduce the Lusztig isomorphism of $\trUqQ({\hat {sl_2}})$,
and obtain $q$-commutator relations
between root vectors of the positive and negative parts.
In Section~5, we give an explicit formula for the universal $R$-matrix. 
In Section~6,
we introduce the vector representation of $\trUqQ({\hat {sl_2}})$,
and obtain the explicit formula of the image of its universal $R$-matrix.
In Section~\ref{section:DiReshTwist},
we explain the Drinfeld-Reshetikhin twisting.
In Section~\ref{section:hadic},
inspired by \cite{GG25}, we introduce 
the topological
$\bC[[h]]$
double parameter 
affine quantum group 
$U^{{\mathfrak{h}}}_{\hbar,P}({\hat {sl_2}})$
and decide its universal $R$-matrix. Regarding $U^{{\mathfrak{h}}}_{\hbar,P}({\hat {sl_2}})$,
it is somewhat surprising that the double-parameter version $P$ (see \eqref{eqn:Pmatrix}) of the Cartan matrix of type $A_1^{(1)}$ is not degenerate.

In the future, we might obtain results similar to those of this paper for the other quantum affine (super) algebras, including the double parameter version of $U_q(\text{A}(0,2)^{(4)})$ whose prominent study was recently achieved by \cite{Li26}. In this case, the approach to the universal $R$-matrix will be very different, although $U_q(\text{A}(0,2)^{(4)})$ shares many quite similar properties with $U_q(\hat{sl_2})$, e.g. their Cartan matrices and root systems are the same. The main reason is,  the estimation of the coproduct of imaginary root vectors, see Section 3 of \cite{Li26}, is quite different, because of the existence of odd imaginary roots and the non-commutativity of certain imaginary root vectors. 

\section{Definition of $\trUqQ({\hat {sl_2}})$}

We call an element similar to the one of \cite[Proposition~1.7.1]{Y94}
the {\it{universal $R$-matirx}}.

Let $\fkJ_{x,y}:=\{z\in\bZ\,|\,x\leq z\leq y\}$ for $x$, $y\in\bR\cup\{\pm\infty\}$.
For $n\in\bZgeqo$ and $x\in\bC$, let 
$(n)_x:=\sum_{m=0}^{n-1}x^m$
and $[n]_x:=x^{-(n-1)}(n)_{x^2} $.
Let $\bCt=\bC\setminus\{0\}$ . We use the following expressions and terminology in  
{\rm{(Df1)}}-{\rm{(Df3)}} below.
\newline\newline
{\rm{(Df1)}}
Let $A$ be a commutative ring.
Let $B$ be an associative $A$-algebra (with $1$).
Let $X_1,\ldots,X_n$ 
be $A$-subalgebras (with $1$)
of $B$.
We say that
{\it{$X_1\otimes_A\cdots\otimes_AX_n$ can be naturally identified with $B$ as an $A$-module}}, if
the $A$-module homomorphism
$f:X_1\otimes_A\cdots\otimes_AX_n
\to B$ defined by 
$f(x_1\otimes_A\cdots\otimes_Ax_n)
=x_1\cdots x_n$
is bijective.
Let $C$ be an $A$-subalgebra 
(with $1$)
of $B$. Let $Y$ and $Z$
be two-sided ideal of
$B$ and $C$, respectively.
Assume $Z\subset Y$.
Let ${\tilde C}
=\{c+Y|c\in C\}$
$(\subset B/Y)$.
We say that 
{\it{${\tilde C}$
can be naturally identified with $C/Z$ as an $A$-algebra}} if
the $A$-module homomorphism
$g:C/Z\to B/Y$ defined by 
$g(c+Z)=c+Y$ $(c\in C)$
is injective.
\newline
{\rm{(Df2)}}
Let $V^+$ and $V^-$ be a finite-dimensional $\bC$-linear space
with $\dim V^+=\dim V^-$.
Let $l=\dim V^+$
$(\in\bN)$.
Let $(\,,\,):V^+
\times V^-
\to\bC$ be 
a non-degenerate bilinear map.
There exists a
unique $C\in V^+\otimes V^-$
such that 
there exist
$e^\pm_t\in V^\pm$
$(t\in\fkJ_{1,l})$
with $(e^+_t,e^-_z)=\delta_{tz}$
and $C=\sum_{t=1}^le^+_t\otimes e^-_t$.
We call $C$ the
{\it{canonical element of $V^+\otimes V^-$
with respect to $(\,,\,)$.}}
\newline
{\rm{(Df3)}}
Let $V$ be a $\bC$-linear space.
For $M=\sum_{t=1}^rm_t\otimes m^\prime_t\in{\mathrm{End}}_\bC(V\otimes V)$,
let $M_{12}=\sum_{t=1}^rm_t\otimes m^\prime_t\otimes\rmid_V$,
$M_{13}=\sum_{t=1}^rm_t\otimes\rmid_V\otimes m^\prime_t$,
$M_{23}=\sum_{t=1}^r\rmid_V\otimes m_t\otimes m^\prime_t$
$\in{\mathrm{End}}_\bC(V\otimes V\otimes V)$,
and let $M_{21}=\sum_{t=1}^rm^\prime_t\otimes m_t\in{\mathrm{End}}_\bC(V\otimes V)$.
\newline\par
Unless otherwise stated, the terms {\it{algebra}}, {\it{Hopf algebra}} etc.~below
mean those over $\bC$.

Let $q\in\bCt$.
Let $\ha\in\bCt$.
Let $q_{11}=q_{00}=q^2$, $q_{01}=\ha$ and $q_{10}=q^{-4}\ha^{-1}$.
Let
\begin{equation}\label{eqn:defQ}
Q=
\left(\begin{array}{cc} 
q_{00} & q_{01} \\ q_{10} & q_{11} 
\end{array}\right)
=
\left(\begin{array}{cc} 
q^2 & \ha \\ q^{-4}\ha^{-1} & q^2 
\end{array}\right) \quad
(\in{\mathrm{M}}_2(\bC)).
\end{equation}

\begin{equation*}
\mbox{From now on until Theorem~\ref{theorem:trUthPBW}, we assume that $q(q^n-1)\ne 0$ for all $n\in\bN$.}
\end{equation*}

Let $I=\fkJ_{0,1}$ $(=\{0,1\}$.
Let $U=\trUqQ({\hat {sl_2}})$ be the associative $\bC$-algebra with identity
defined by the generators
\begin{equation*}
K_i^{\pm 1}, L_i^{\pm 1}, E_i,  F_i\quad (i\in I), 
\end{equation*} and the relations:
\begin{align}
&XY=YX\,\, (X,Y\in\{K_i^{\pm 1}, L_i^{\pm 1}\}),\quad
K_iK_i^{-1}=K_i^{-1}K_i=L_iL_i^{-1}=L_i^{-1}L_i=1,
\\
&K_iE_jK_i^{-1}=q_{ij}E_j,\,\,K_iF_jK_i^{-1}=q_{ij}^{-1}F_j,\,\,
L_iE_jL_i^{-1}=q_{ji}^{-1}E_j,\,\,L_iF_jL_i^{-1}=q_{ji}F_j, \\
&[E_i,F_j]=\delta_{ij}(-K_i+L_i), \label{eqn:defEiEjKiLi} \\
&E_i^3E_j-q_{ij}(3)_{q_{ii}}E_i^2E_jE_i+q_{ii}q_{ij}^2(3)_{q_{ii}}E_iE_jE_i^2-q_{ii}^3q_{ij}^3E_jE_i^3=0
\quad (i\ne j), \label{eqn:defEiEj} \\
&F_i^3F_j-q_{ji}(3)_{q_{ii}}F_i^2F_jF_i+q_{ii}q_{ji}^2(3)_{q_{ii}}F_iF_jF_i^2-q_{ii}^3q_{ji}^3F_jF_i^3=0
\quad (i\ne j). \label{eqn:defFiFj} 
\end{align}

Define the linear homomorphisms
$\Delta:\trU\to\trU\otimes\trU$ and 
$\varepsilon:\trU\to\trU\otimes\trU$, 
and the linear anti-isomorphisms by
\begin{equation*}
\begin{array}{l}
\Delta(\trK_i)=\trK_i\otimes\trK_i,\,
\Delta(\trL_i)=\trL_i\otimes\trL_i,\\
\Delta(\trE_i)=\trE_i\otimes 1+\trK_i\otimes \trE_i\quad
\Delta(\trF_i)=\trF_i\otimes \trL_i+1\otimes \trF_i, \\
\varepsilon(\trK_i)=\varepsilon(\trL_i)=1,\,
\varepsilon(\trE_i)=\varepsilon(\trF_i)=0, \\
S(\trK_i)=\trK_i^{-1},\,
S(\trL_i)=\trL_i^{-1},\,
S(\trE_i)=-\trK_i^{-1}\trE_i,\,
S(\trF_i)=-\trF_i\trL_i^{-1}
\\
(i\in I).
\end{array}
\end{equation*}
Then $(\trU,\Delta,\varepsilon,S)$
is the Hopf algebra.

\noindent
{\rm{(Ba1)}} Let $\trUo$, $\trUp$, $\trUm$,
$\trUpgeq$, $\trUmleq$,
$\trUogeq$, $\trUoleq$
be the $\bC$-subalgebras 
(with $1$)
defined as
$\langle \trK_i^{\pm 1},
\trL_i^{\pm 1}\,|\,
i\in I\rangle$, 
$\langle \trE_i\,|\,
i\in I\rangle$, 
$\langle \trF_i\,|\,
i\in I\rangle$,
$\langle \trE_i, \trK_i^{\pm 1}\,|\,
i\in I\rangle$, 
$\langle \trF_i, \trL_i^{\pm 1}\,|\,
i\in I\rangle$,
$\langle \trK_i^{\pm 1}\,|\,
i\in I\rangle$, 
$\langle \trL_i^{\pm 1}\,|\,
i\in I\rangle$,
respectively.
\newline\newline
In a standard argument (see \cite{Y89}), we see the following {\rm{(Ba2)}}-{\rm{(Ba4)}}.
\newline\newline
{\rm{(Ba2)}} 
$\trUm\otimes \trUo\otimes\trUp$,
$\trUmleq\otimes\trUpgeq$,
$\trUoleq\otimes\trUogeq$,
$\trUp\otimes\trUogeq$,
$\trUm\otimes\trUoleq$
can be naturally identified 
with $\trU$, $\trU$, $\trUo$, 
$\trUpgeq$, $\trUmleq$, 
respectively,  as a $\bC$-linear space.
\newline
{\rm{(Ba3)}}  
$\trUo$ has the
 $\bC$-basis 
$\{K_0^{x_0}K_1^{x_1}L_0^{y_0}L_1^{y_1}|
x_0, x_1, y_0, y_1\in\bZ\}$ (i.e., 
$\trUo$ is the Laurent polynomial $\bC$-algebra in 
$K_0^{\pm 1}$, $K_1^{\pm 1}$, $L_0^{\pm 1}$, $L_1^{\pm 1}$).
\newline
{\rm{(Ba4)}}  $\trUp$ (resp. $\trUm$) is the $\bC$-algebra presented by
the generators $E_0$, $E_1$ (resp. $F_0$, $F_1$) and the relations
\eqref{eqn:defEiEj} (resp. \eqref{eqn:defFiFj}).
\newline\par
Let $\Pi$ be a set composed of two elements $\al_0$ and $\al_1$, i.e.,
$\Pi=\{\al_0,\al_1\}$.
Let $\bZ\Pi$ be the rank-two free $\bZ$-module
with the basis $\Pi$, i.e., $\bZ\Pi=\bZ\al_0\oplus\bZ\al_1$.
Let $\bZgeqo\Pi=\bZgeqo\al_0\oplus\bZgeqo\al_1$
$(\subset\bZ\Pi)$.
For $x$, $y\in\bZ$,
let $K_{x\al_0+y\al_1}=K_0^xK_1^y$
and $L_{x\al_0+y\al_1}=L_0^xL_1^y$.
For $\beta=x\al_0+y\al_1\in\bZ\Pi$ $(x, y\in\bZ)$,
let $\trK_\beta=K_0^xK_1^y$ and $\trL_\beta=L_0^xL_1^y$.

Let $\chi=\chi_Q:\bZ\Pi\times\bZ\Pi\to\bCt$
be the $\bZ$-bimodule homomorphism defined by
$\chi(\al_i,\al_j):=q_{ij}$ $(i,j\in I)$. 
\newline\newline
{\rm{(Ba5)}} We define the $\bC$-linear subspaces $U_\lambda$ $(\lambda\in\bZ\Pi)$
of $U$ such that 
\begin{equation} \label{eqn:DsumtrU}
\begin{array}{l}
U=\oplus_{\lambda\in\bZgeqo\Pi}U_\lambda,\,\,
1, K_i^{\pm 1}, L_i^{\pm 1}\in U_0,\,
E_i\in U_{\al_i},\,F_i\in U_{-\al_i}\,\,(i\in I) \\
\mbox{and}\,\,
U_\lambda U_\mu\subset U_{\lambda+\mu}\,
(\lambda,\mu\in\bZgeqo\Pi).
\end{array}
\end{equation}
Let $\trU^\pm_\lambda:=\trU^\pm\cap \trU_\lambda$
$(\lambda,\mu\in\bZgeqo\Pi)$. Note that 
\begin{equation} \label{eqn:DsumUpm}
\begin{array}{l}
\trU^\pm=\oplus_{\lambda\in\bZgeqo\Pi}\trU^\pm_{\pm\lambda},\,\,
1\in\trU^\pm_0,\,\, \dim\trU^\pm_0=1,\\
E_i\in\trUp_{\al_i},\,F_i\in\trUm_{-\al_i},\,
\dim\trU^\pm_{\pm\al_i}=1\,(i\in I), \\
\mbox{and}\,\,
\trU^\pm_{\pm\lambda}\trU^\pm_{\pm\mu}\subset\trU^\pm_{\pm(\lambda+\mu)}\,
(\lambda,\mu\in\bZgeqo\Pi).
\end{array}
\end{equation}
\begin{equation}
\label{eqn:defNormbZgeqo}
\mbox{For $\lambda=x\al_0+y\al_1\in\bZgeqo\Pi$ $(x,y\in\bZgeqo)$,
let $|\lambda|=x+y$ $(\in\bZgeqo)$. }    
\end{equation}

\section{Positive part $\trUp$ of $\trUqQ({\hat {sl_2}})$}

\subsection{Root vectors $\trE_{k\delta+\al_0}$, $\trE_{k\delta+\al_1}$ and $\trE_{(k+1)\delta}$ with $k\geq 0$}

\begin{lemma}
There exists a unique $\bC$-algebra 
anti-isomorphism $\psi:\trUp\to\trUp$
such that $\psi(E_0):=E_1$ and $\psi(E_1):=E_0$.
\end{lemma}

Let
\begin{equation*}
\lb X,Y\rb:=XY-\chi(\lambda,\mu)YX \quad
(X\in\trUp_\lambda, Y\in\trUp_\mu).
\end{equation*}
\newline
Then we have
\begin{equation*}
\left\{\begin{array}{l}
\lb \lb X,Y\rb,Z\rb
=\lb X, \lb Y,Z\rb \rb
+\chi(\mu,\nu)\lb X,Z\rb Y-\chi(\lambda,\mu) Y \lb X,Z\rb, \\
\lb \lb X,Y\rb,Z\rb
=\lb X, \lb Y,Z\rb \rb
+\chi(\mu,\nu)
\left(\lb X,Z\rb Y- 
{\frac {\chi(\lambda,\mu)} {\chi(\mu,\nu)}} Y \lb X,Z\rb
\right), \\
\lb X, \lb Y,Z\rb \rb  
=\lb \lb X,Y\rb,Z\rb
+\chi(\lambda,\mu)
\left(Y\lb X,Z\rb - 
{\frac {\chi(\mu,\nu)} {\chi(\lambda,\mu)}} \lb X,Z\rb Y
\right) \\
\mbox{for $X\in\trUp_\lambda$, $Y\in\trUp_\mu$ and $Z\in\trUp_\nu$.}
\end{array}\right.
\end{equation*}

Let $\delta:=\al_0+\al_1$.
For $n\in\bZgeqo$, $X\in\trUp_{n\delta}$, $\lambda\in\bZgeqo\Pi$ and $Y\in \trUp_\lambda$,
we have
\begin{equation*}
\lb X,Y\rb
=-\chi(n\delta,\lambda)\lb Y,X\rb
\quad\rm{and}\quad
\lb Y,X\rb=-\chi(\lambda,n\delta)\lb X,Y\rb.
\end{equation*} 

For $n\in\bZgeqo$, $Z\in \trUp_{n\delta}$, 
$\lambda\in\bZgeqo\Pi$, $X\in \trUp_\lambda$,
$\mu\in\bZgeqo\Pi$ and $Y\in \trUp_\mu$,
we have
\begin{equation*}
\begin{array}{l}
\lb \lb X,Y \rb, Z\rb=
\lb X,\lb Y , Z\rb\rb
+\chi(\mu,n\delta)\lb\lb X,Z\rb, Y\rb, \\
\lb Z, \lb X,Y \rb\rb=\lb\lb Z,X\rb, Y \rb
+\chi(n\delta,\lambda)
\lb X,\lb Z, Y \rb\rb.
\end{array}
\end{equation*} 

Define
\begin{equation*}
\left\{\begin{array}{l}
\trE_{\al_1}=\trE_1, \trE_{\al_0}=\trE_0, \\
\trE_{n\delta+\al_1}={\frac 1 {[2]_q}}
\lb \trE_\delta,\trE_{(n-1)\delta+\al_1}\rb \quad (n\in\bN), \\
\trE_{n\delta}=\lb \trE_{\al_0},\trE_{(n-1)\delta+\al_1}\rb \quad (n\in\bN), \\
\trE_{n\delta+\al_0}={\frac 1 {[2]_q}}
\lb \trE_{(n-1)\delta+\al_0}, \trE_\delta\rb\quad (n\in\bN). 
\end{array}\right.
\end{equation*}

By \eqref{eqn:defEiEj},
we have
\begin{equation}
\label{eqn:Edelaliali}
\lb \trE_{\delta+\al_1},\trE_{\al_1}\rb=0,
\quad \lb \trE_{\al_0},\trE_{\delta+\al_0}\rb=0.
\end{equation}

\begin{lemma} We have
$\psi(\trE_{n\delta})=\lb \trE_{(n-1)\delta+\al_0}, \trE_{\al_1}\rb$
$(n\in\bN)$, 
$\psi(\trE_{n\delta+\al_1})=\trE_{n\delta+\al_0}$ $(n\in\bZgeqo)$, 
and $\psi(\trE_{n\delta+\al_0})=\trE_{n\delta+\al_1}$ $(n\in\bZgeqo)$.
In particular, $\psi(\trE_\delta)=\trE_\delta$.
\end{lemma}

We have
\begin{equation}
\label{eqn:eqnlemma1aa}
\trE_{2\delta}=\lb \trE_{\al_0},\trE_{\delta+\al_1}\rb=\lb \trE_{\delta+\al_0}, \trE_{\al_1}\rb,
\end{equation}
since
\begin{equation*}
\begin{array}{l}
\trE_{2\delta}=\lb \trE_{\al_0},\trE_{\delta+\al_1}\rb
={\frac 1 {[2]_q}}\lb \trE_{\al_0}, \lb \trE_\delta,\trE_{\al_1}\rb\rb \\
\quad =\lb \trE_{\delta+\al_0}, \trE_{\al_1} \rb
+{\frac {\chi(\al_0,\delta)} {[2]_q}}
(\trE_\delta^2-{\frac {\chi(\delta,\al_1)} {\chi(\al_0,\delta)}}\trE_\delta^2) \\
\quad =\lb \trE_{\delta+\al_0}, \trE_{\al_1}\rb.
\end{array}
\end{equation*}
By \eqref{eqn:eqnlemma1aa}, 
we have
$\psi(\trE_{2\delta})=\trE_{2\delta}$.

We have
\begin{equation*}
\begin{array}{l}
\lb \trE_{2\delta},\trE_{\al_1}\rb
= \lb \lb \trE_{\al_0}, \trE_{\delta+\al_1}\rb, \trE_{\al_1}\rb \\
\quad =\chi(\delta+\al_1,\al_1)(\trE_\delta \trE_{\delta+\al_1}
-{\frac {\chi(\al_0,\delta+\al_1)} {\chi(\delta+\al_1,\al_1)}}\trE_{\delta+\al_1}\trE_\delta) \\
\quad = \ha q^4(\trE_\delta \trE_{\delta+\al_1}-{\frac {\ha^2 q^2} {\ha q^4}}\trE_{\delta+\al_1}\trE_\delta) \\
\quad = \ha q^4(\trE_\delta \trE_{\delta+\al_1}-{\frac {\ha} {q^2}}
\cdot{\frac 1 {\ha q^2}}(\trE_\delta \trE_{\delta+\al_1}-[2]_qE_{2\delta+\al_1})) \\
\quad = \ha q^4(1-q^{-4})\trE_\delta \trE_{\delta+\al_1}+\ha [2]_qE_{2\delta+\al_1} \\
\quad = \ha (q^4-1)\trE_\delta \trE_{\delta+\al_1}+\ha [2]_qE_{2\delta+\al_1}
\end{array}
\end{equation*}
Since $\chi(2\delta,\al_0)=\chi(\al_1,2\delta)$, we have
\begin{equation*}
\lb \trE_{\al_0},\trE_{2\delta}\rb=\psi(\lb \trE_{2\delta},\trE_{\al_1}\rb)
=\ha (q^4-1)\trE_{\delta+\al_0}\trE_\delta
+\ha [2]_qE_{2\delta+\al_0}.
\end{equation*}
We have
\begin{equation*}
\begin{array}{l}
\lb \trE_{\delta+\al_0},\trE_{\delta+\al_1}\rb
={\frac 1 {[2]_q}}\lb\lb \trE_{\al_0},\trE_\delta\rb,\trE_{\delta+\al_1}\rb \\
\quad = \lb \trE_{\al_0}, \trE_{2\delta+\al_1}\rb
+{\frac {\chi(\delta,\delta+\al_1)} {[2]_q}}
(\trE_{2\delta}\trE_\delta-{\frac {\chi(\al_0,\delta)} {\chi(\delta,\delta+\al_1)}}\trE_\delta \trE_{2\delta}) \\
\quad = \trE_{3\delta}
+{\frac {\ha q^2} {[2]_q}}
(\trE_{2\delta}\trE_\delta-{\frac {\ha q^2} {\ha q^2}}\trE_\delta \trE_{2\delta}) \\
\quad = \trE_{3\delta}
+{\frac {\ha q^2} {[2]_q}}
\lb \trE_{2\delta},\trE_\delta\rb,
\end{array}
\end{equation*}
and
\begin{equation*}
\begin{array}{l}
\lb \trE_{2\delta+\al_0}, \trE_{\al_1}\rb ={\frac 1 {[2]_q}}\lb\lb \trE_{\delta+\al_0},\trE_\delta\rb, \trE_{\al_1}\rb \\
\quad = \lb \trE_{\delta+\al_0}, \trE_{\delta+\al_1}\rb
+{\frac {\chi(\delta,\delta+\al_1)} {[2]_q}}(\trE_{2\delta} \trE_\delta
-{\frac {\chi(\al_0,\delta)} {\chi(\delta,\delta+\al_1)}}\trE_\delta \trE_{2\delta})  \\
\quad = \lb \trE_{\delta+\al_0}, \trE_{\delta+\al_1}\rb
+{\frac {\chi(\delta,\delta+\al_1)} {[2]_q}}\lb \trE_{2\delta}, \trE_\delta\rb \\
\quad = \trE_{3\delta}+{\frac {2\ha q^2} {[2]_q}}\lb \trE_{2\delta}, \trE_\delta\rb,
\end{array}
\end{equation*}
which implies
\begin{equation*}
\begin{array}{l}
\lb \trE_{2\delta},\trE_\delta\rb =\lb \trE_{2\delta}, \lb \trE_{\al_0},\trE_{\al_1}\rb\rb \\
\quad =\lb\lb \trE_{2\delta}, \trE_{\al_0}\rb, \trE_{\al_1}\rb
+\chi(2\delta,\al_0)\lb \trE_{\al_0},\lb \trE_{2\delta}, \trE_{\al_1}\rb\rb \\
\quad =\chi(2\delta,\al_0)
(-\lb\lb \trE_{\al_0},\trE_{2\delta}\rb, \trE_{\al_1}\rb+\lb \trE_{\al_0},\lb \trE_{2\delta}, \trE_{\al_1}\rb\rb) \\
\quad = {\frac 1 {\chi(\al_0,2\delta)}}(
-\lb \ha (q^4-1)\trE_{\delta+\al_0}\trE_\delta
+\ha [2]_qE_{2\delta+\al_0}, \trE_{\al_1}\rb \\
\quad\quad +\lb \trE_{\al_0},  \ha (q^4-1)\trE_\delta \trE_{\delta+\al_1}+\ha [2]_qE_{2\delta+\al_1}\rb ) \\
\quad ={\frac 1 {\ha^2 q^4}}(-\ha (q^4-1)({\frac 1 {[2]_q}}\trE_{\delta+\al_0}\trE_{\delta+\al_1}+\chi(\delta,\al_1)\trE_{2\delta} \trE_\delta) \\
\quad\quad -\ha [2]_q(\trE_{3\delta}+{\frac {2\ha q^2} {[2]_q}}\lb \trE_{2\delta}, \trE_\delta\rb) \\
\quad\quad + \ha (q^4-1)({\frac 1 {[2]_q}}\trE_{\delta+\al_0}\trE_{\delta+\al_1}+\chi(\al_0,\delta)\trE_\delta \trE_{2\delta}) \\
\quad\quad + \ha [2]_q \trE_{3\delta}) \\
\quad ={\frac 1 {\ha^2 q^4}}(-\ha^2 q^2(q^4-1)\lb \trE_{2\delta},\trE_\delta\rb-2\ha^2 q^2\lb \trE_{2\delta},\trE_\delta\rb) \\
\quad ={\frac 1 {\ha^2 q^4}}(-\ha^2 q^2(q^4+1))\lb \trE_{2\delta},\trE_\delta\rb \\
\quad = -(q^2+q^{-2})\lb \trE_{2\delta},\trE_\delta\rb.
\end{array}
\end{equation*}
Hence we have
\begin{equation*}
0=(q^2+1+q^{-2})\lb \trE_{2\delta},\trE_\delta\rb
=q^{-2}(q^2+q+1)(q^2-q+1)\lb \trE_{2\delta},\trE_\delta\rb,
\end{equation*} which implies
\begin{equation}
\label{eqn:EtwodelEdel}
\lb \trE_{2\delta},\trE_\delta\rb=0.
\end{equation}

We have
\begin{equation*}
\begin{array}{l}
\lb \trE_\delta^2,\trE_{\al_1}\rb =
[2]_q( \trE_\delta \trE_{\delta+\al_1}+\chi(\delta,\al_1)\trE_{\delta+\al_1}\trE_\delta)\\
\quad =[2]_q( \trE_\delta \trE_{\delta+\al_1}+
{\frac {\chi(\delta,\al_1)} {\chi(\delta,\delta+\al_1)}}
(\trE_\delta \trE_{\delta+\al_1}-[2]_q \trE_{2\delta+\al_1})) \\
\quad = -[2]_q^2 \trE_{2\delta+\al_1} +  2[2]_q \trE_\delta \trE_{\delta+\al_1}.
\end{array}
\end{equation*}
Using $\psi$, we have
\begin{equation*}
\lb \trE_{\al_0},\trE_\delta^2\rb = -[2]_q^2 \trE_{2\delta+\al_0} +  2[2]_q \trE_{\delta+\al_0}\trE_\delta. 
\end{equation*}

Let ${\tilde{E}}_{2\delta}=-{\frac {q-q^{-1}} 2}\trE_\delta^2
+{\frac 1 {\ha q^2}}\trE_{2\delta}$. Then we have
\begin{equation*}
\begin{array}{l}
\lb {\tilde{E}}_{2\delta}, \trE_{\al_1} \rb \\
\quad = -{\frac {q-q^{-1}} 2}( -[2]_q^2 \trE_{2\delta+\al_1} +  2[2]_q \trE_\delta \trE_{\delta+\al_1})  \\
\quad\quad + {\frac 1 {\ha q^2}}(\ha (q^4-1)\trE_\delta \trE_{\delta+\al_1}+\ha [2]_qE_{2\delta+\al_1}) \\
\quad = [2]_q([2]_q\cdot{\frac {q-q^{-1}} 2}+q^{-2})\trE_{2\delta+\al_1} \\
\quad = [2]_q\cdot{\frac {q^2+q^{-2}} 2}\trE_{2\delta+\al_1} \\
\quad = {\frac {[4]_q}  2}\trE_{2\delta+\al_1}.
\end{array}
\end{equation*} For $n\in\bZgeqo$, by \eqref{eqn:EtwodelEdel}, we have
\begin{equation}
\label{eqn:tilEtwodelEndelalone}
\lb {\tilde{E}}_{2\delta}, \trE_{n\delta+\al_1}\rb ={\frac {[4]_q}  2}\trE_{(n+2)\delta+\al_1} ,
\end{equation} and using $\psi$, we also have
\begin{equation}
\label{eqn:tilEtwodelEndelalzero}
\lb \trE_{n\delta+\al_0}, {\tilde{E}}_{2\delta} \rb ={\frac {[4]_q}  2}\trE_{(n+2)\delta+\al_0}.
\end{equation}

For $m,n\in\bZgeqo$, 
we have
\begin{equation}
\label{eqn:EdelEmdelalEdelnal}
\begin{array}{l}
\lb \trE_\delta, \lb \trE_{m\delta+\al_1}, \trE_{n\delta+\al_1} \rb\rb \\
\quad = [2]_q(\lb \trE_{(m+1)\delta+\al_1},\trE_{n\delta+\al_1}\rb +
\ha q^2\lb \trE_{m\delta+\al_1}, \trE_{(n+1)\delta+\al_1}\rb),
\end{array}
\end{equation} 
\begin{equation}
\label{eqn:EdelEdelEmdelalEdelnal}
\begin{array}{l}
\lb \trE_\delta, \lb \trE_\delta, \lb \trE_{m\delta+\al_1}, \trE_{n\delta+\al_1} \rb\rb\rb \\
\quad = [2]_q^2(\lb \trE_{(m+2)\delta+\al_1}, \trE_{n\delta+\al_1} \rb
+2\ha q^2\lb \trE_{(m+1)\delta+\al_1}, \trE_{(n+1)\delta+\al_1} \rb \\
\quad\quad +\ha^2 q^4\lb \trE_{m\delta+\al_1}, \trE_{(n+2)\delta+\al_1} \rb)
\end{array}
\end{equation} and, by 
\eqref{eqn:tilEtwodelEndelalone}, we also have
\begin{equation}
\label{eqn:EtwodelEmdelaloneEndelalone}
\begin{array}{l}
\lb {\tilde{E}}_{2\delta},\lb \trE_{m\delta+\al_1}, \trE_{n\delta+\al_1} \rb\rb \\
\quad = {\frac {[4]_q}  2}
(\lb \trE_{(m+2)\delta+\al_1}, \trE_{n\delta+\al_1} \rb
+\ha^2 q^4\lb \trE_{m\delta+\al_1}, \trE_{(n+2)\delta+\al_1} \rb).
\end{array}
\end{equation}
By \eqref{eqn:Edelaliali},
\eqref{eqn:EdelEmdelalEdelnal}
and \eqref{eqn:EtwodelEmdelaloneEndelalone},
we have 
\newline\newline
\mbox{$\EQone$}\quad\quad\quad\quad\quad\quad
$\lb \trE_{(n+1)\delta+\al_1}, \trE_{n\delta+\al_1}\rb =0 
\quad(n\in\bZgeqo)$,
\newline\newline
and
\newline\newline
\mbox{$\EQtwo$}\quad
$\lb \trE_{(n+r)\delta+\al_1}, \trE_{n\delta+\al_1}\rb
+\ha^{r-1} q^{2(r-1)}\lb \trE_{(n+1)\delta+\al_1}, \trE_{(n+r-1)\delta+\al_1}\rb
=0 \quad (n\in\bZgeqo, r\in\bN)$.
\newline\newline
Using $\psi$, we also have
\newline\newline
\mbox{$\EQthree$}\quad\quad\quad\quad\quad\quad
$\lb \trE_{n\delta+\al_0}, \trE_{(n+1)\delta+\al_0}\rb =0 
\quad(n\in\bZgeqo)$,
\newline\newline
and
\newline\newline
\mbox{$\EQfour$}\,\,
$\lb  \trE_{n\delta+\al_0}, \trE_{(n+r)\delta+\al_0}\rb
+\ha^{r-1} q^{2(r-1)}\lb \trE_{(n+r-1)\delta+\al_0}, \trE_{(n+1)\delta+\al_0}\rb
=0 \quad (n\in\bZgeqo, r\in\bN).$
\newline\newline
Note that
\begin{equation*}
\begin{array}{l}
\chi(\delta+\al_1,(r-1)\delta+\al_1)
=\chi((r-1)\delta+\al_1+(-r+2)\delta,(r-1)\delta+\al_1) \\
\quad = q^2\cdot \ha^{-r+2}q^{2(-r+2)}=\ha^{-r+2}q^{2(-r+3)}.
\end{array}
\end{equation*}
Then we have 
\begin{equation*}
\begin{array}{l}
\lb \trE_{\al_0}, \lb \trE_{\delta+\al_1}, \trE_{(r-2)\delta+\al_1}\rb\rb \\
\quad = \lb \trE_{2\delta}, \trE_{(r-2)\delta+\al_1}\rb
+\chi(\al_0,\delta+\al_1)
(\trE_{\delta+\al_1}\trE_{(r-1)\delta}
-{\frac {\chi(\delta+\al_1,(r-2)\delta+\al_1)} {\chi(\al_0,\delta+\al_1)}}
\trE_{(r-1)\delta}\trE_{\delta+\al_1}) \\
\quad = \lb \trE_{2\delta}, \trE_{(r-2)\delta+\al_1}\rb
+\ha^2q^2
(\trE_{\delta+\al_1}\trE_{(r-1)\delta}
-{\frac {\ha^{-r+3}q^{2(-r+4)}} {\ha^2q^2}}
\trE_{(r-1)\delta}\trE_{\delta+\al_1})  \\
\quad = \lb \trE_{2\delta}, \trE_{(r-2)\delta+\al_1}\rb
+\ha^2q^2
({\frac 1 {\ha^{r-1}q^{2(r-1)}}}(
\trE_{(r-1)\delta} \trE_{\delta+\al_1}-\lb \trE_{(r-1)\delta}, \trE_{\delta+\al_1}\rb)
-\ha^{-r+1}q^{2(-r+3)}
\trE_{(r-1)\delta}\trE_{\delta+\al_1}) \\
\quad = \lb \trE_{2\delta}, \trE_{(r-2)\delta+\al_1}\rb
-\ha^{-r+3}q^{-2r+4}(q^4-1)\trE_{(r-1)\delta}\trE_{\delta+\al_1}
-\ha^{-r+3} q^{-2(r-2)}\lb \trE_{(r-1)\delta}, \trE_{\delta+\al_1}\rb \\
\quad = (\ha (q^4-1) \trE_\delta \trE_{(r-1)\delta+\al_1}+\ha [2]_q \trE_{r\delta+\al_1})  \\
\quad\quad -\ha^{-r+3}q^{-2r+4}(q^4-1)\trE_{(r-1)\delta}\trE_{\delta+\al_1}
-\ha^{-r+3} q^{-2(r-2)}\lb \trE_{(r-1)\delta}, \trE_{\delta+\al_1}\rb \\
\quad = \ha (q^4-1) \trE_\delta \trE_{(r-1)\delta+\al_1}
-\ha^{-r+3}q^{-2r+4}(q^4-1)\trE_{(r-1)\delta}\trE_{\delta+\al_1} \\
\quad\quad +\ha [2]_q \trE_{r\delta+\al_1}
-\ha^{-r+3} q^{-2(r-2)}\lb \trE_{(r-1)\delta}, \trE_{\delta+\al_1}\rb.
\end{array}
\end{equation*}
\begin{equation}
\label{eqn:aninduction}
\mbox{To prove 
$\EQfive$-$\EQeight$ below, we use an induction on $|\cdot|$
of \eqref{eqn:defNormbZgeqo}.}
\end{equation}
We have
\begin{equation*}
\begin{array}{l}
\lb \trE_{r\delta}, \trE_{\al_1}\rb \\
\quad = \lb\lb \trE_{\al_0}, \trE_{(r-1)\delta+\al_1}\rb, \trE_{\al_1}\rb \\
\quad =\lb \trE_{\al_0}, \lb \trE_{(r-1)\delta+\al_1},\trE_{\al_1}\rb\rb
+\chi((r-1)\delta+\al_1,\al_1)
(\trE_\delta \trE_{(r-1)\delta+\al_1}
-{\frac {\chi(\al_0,(r-1)\delta+\al_1)} {\chi((r-1)\delta+\al_1,\al_1)}}
\trE_{(r-1)\delta+\al_1}\trE_\delta) \\
\quad =-\ha^{r-2} q^{2(r-2)}\lb \trE_{\al_0}, \lb \trE_{\delta+\al_1}, \trE_{(r-2)\delta+\al_1}\rb\rb
+\ha^{r-1}q^{2r}
(\trE_\delta \trE_{(r-1)\delta+\al_1}
-{\frac {\ha^r q^{2(r-1)}} {\ha^{r-1}q^{2r}}}
\trE_{(r-1)\delta+\al_1}\trE_\delta) \\
\quad =-\ha^{r-2} q^{2(r-2)}
\Bigl(
\ha (q^4-1) \trE_\delta \trE_{(r-1)\delta+\al_1}
-\ha^{-r+3}q^{-2r+4}(q^4-1)\trE_{(r-1)\delta}\trE_{\delta+\al_1} \\
\quad\quad +\ha [2]_q \trE_{r\delta+\al_1}
-\ha^{-r+3} q^{-2(r-2)}\lb \trE_{(r-1)\delta}, \trE_{\delta+\al_1}\rb
\Bigr)+\ha^{r-1}q^{2r}
\trE_\delta \trE_{(r-1)\delta+\al_1} \\
\quad\quad +\ha^r q^{2(r-1)}\cdot{\frac 1 {\ha q^2}}([2]_q \trE_{r\delta+\al_1}-\trE_\delta
\trE_{(r-1)\delta+\al_1}) \\
\quad =\ha (q^4-1) \trE_{(r-1)\delta} \trE_{\delta+\al_1}
+\ha \lb \trE_{(r-1)\delta}, \trE_{\delta+\al_1}\rb 
\\
\quad =\ha (q^4-1) \trE_{(r-1)\delta} \trE_{\delta+\al_1}
+
{\frac \ha {[2]_q}}
\lb \trE_\delta, \lb \trE_{(r-1)\delta}, \trE_{\al_1}\rb\rb
\quad(\mbox{by $\EQseven$ with \eqref {eqn:aninduction}})
\\
\quad =\ha (q^4-1) \trE_{(r-1)\delta} \trE_{\delta+\al_1}
+{\frac \ha {[2]_q}}
\lb \trE_\delta,\Bigl((q^4-1)\sum_{k=1}^{r-2}\ha^k 
\trE_{((r-1)-k)\delta}\trE_{k\delta+\al_1}+\ha^{(r-1)-1}[2]_qE_{(r-1)\delta+\al_1}
\Bigr) \rb \\
\hspace{3cm}
(\mbox{by $\EQfive$ with \eqref {eqn:aninduction}})
\\
\quad =\ha (q^4-1) \trE_{(r-1)\delta} \trE_{\delta+\al_1}
+\ha\Bigl((q^4-1)\sum_{k=1}^{r-2}\ha^k 
\trE_{((r-1)-k)\delta}\trE_{(k+1)\delta+\al_1}+\ha^{(r-1)-1}[2]_qE_{r\delta+\al_1}
\Bigr) \\
\quad =\ha (q^4-1) \trE_{(r-1)\delta} \trE_{\delta+\al_1}
+(q^4-1)\sum_{k=2}^{r-1}\ha^k 
\trE_{(r-k)\delta}\trE_{k\delta+\al_1}+\ha^{r-1}[2]_qE_{r\delta+\al_1}
\\
\quad =
(q^4-1)\sum_{k=1}^{r-1}\ha^k 
\trE_{(r-k)\delta}\trE_{k\delta+\al_1}+\ha^{r-1}[2]_qE_{r\delta+\al_1}.
\end{array}
\end{equation*}
\newline\newline
We conclude
\newline\newline
$\EQfive$ \quad
$\lb \trE_{r\delta}, \trE_{n\delta+\al_1}\rb  =
(q^4-1)\sum_{k=1}^{r-1}\ha^k 
\trE_{(r-k)\delta}\trE_{(n+k)\delta+\al_1}+\ha^{r-1}[2]_qE_{(n+r)\delta+\al_1}$,
\newline\newline
and
\newline\newline
$\EQsix$ \quad
$\lb \trE_{n\delta+\al_0}, \trE_{r\delta}\rb  =
(q^4-1)\sum_{k=1}^{r-1}\ha^k 
\trE_{(n+k)\delta+\al_0}\trE_{(r-k)\delta}+\ha^{r-1}[2]_qE_{(n+r)\delta+\al_0}$.
\newline\par
We have
\begin{equation*}
\begin{array}{l}
\lb \trE_{y\delta+\al_0},\trE_{x\delta+\al_1} \rb \\
\quad = 
{\frac 1 {[2]_q}}
\lb\lb \trE_{(y-1)\delta+\al_0}, \trE_\delta\rb, \trE_{x\delta+\al_1} \rb \\
\quad = \lb \trE_{(y-1)\delta+\al_0}, \trE_{(x+1)\delta+\al_1} \rb \\
\quad\quad +{\frac {\chi(\delta,x\delta+\al_1)} {[2]_q}}
(\trE_{(x+y)\delta}\trE_\delta-
{\frac {\chi((y-1)\delta+\al_0,\delta)} {\chi(\delta,x\delta+\al_1)}}
\trE_\delta \trE_{(x+y)\delta}) \\
\quad = \lb \trE_{(y-1)\delta+\al_0}, \trE_{(x+1)\delta+\al_1} \rb
-{\frac {\ha q^2} {[2]_q}}\lb \trE_\delta, \trE_{(x+y)\delta} \rb \\
\quad =  \trE_{(x+y+1)\delta}
-{\frac {y \ha q^2} {[2]_q}}\lb \trE_\delta, \trE_{(x+y)\delta} \rb.
\end{array}
\end{equation*}
Then we have 
\begin{equation*}
\begin{array}{l}
\lb \trE_{y\delta},\trE_{x\delta} \rb \\
\quad = \lb \trE_{y\delta}, \lb \trE_{\al_0}, \trE_{(x-1)\delta+\al_1} \rb\rb \\
\quad = \lb\lb \trE_{y\delta},  \trE_{\al_0} \rb, \trE_{(x-1)\delta+\al_1} \rb
+\chi(y\delta,\al_0)
\lb \trE_{\al_0}, \lb \trE_{y\delta}, \trE_{(x-1)\delta+\al_1} \rb\rb \\
\quad = -\chi(y\delta,\al_0)\lb\lb \trE_{\al_0},\trE_{y\delta}\rb, \trE_{(x-1)\delta+\al_1} \rb
+\chi(y\delta,\al_0)
\lb \trE_{\al_0}, \lb \trE_{y\delta}, \trE_{(x-1)\delta+\al_1} \rb\rb \\
\quad = \ha^{-y}q^{-2y}
\Bigl(-\lb(q^4-1)\sum_{k=1}^{y-1}\ha^k 
\trE_{k\delta+\al_0}\trE_{(y-k)\delta}+\ha^{y-1}[2]_qE_{y\delta+\al_0}, \trE_{(x-1)\delta+\al_1} \rb \\
\quad\quad + \lb \trE_{\al_0},
(q^4-1)\sum_{k=1}^{y-1}\ha^k 
\trE_{(y-k)\delta}\trE_{(x-1+k)\delta+\al_1}+\ha^{y-1}[2]_qE_{(x-1+y)\delta+\al_1} \rb \Bigr) \\
\quad = -\ha^{-y}q^{-2y}(q^4-1)
\sum_{k=1}^{y-1}\ha^kE_{k\delta+\al_0}
\bigl((q^4-1)\sum_{e=1}^{y-k-1}
\ha^eE_{(y-k-e)\delta}\trE_{(x-1+e)\delta+\al_1}
+\ha^{y-k-1}[2]_qE_{(y-k+x-1)\delta+\al_1}
\bigr) \\
\quad -\ha^{-y}q^{-2y}(q^4-1)
\sum_{k=1}^{y-1}\ha^k\ha^{y-k}q^{2(y-k)}
\trE_{(x+k)\delta}\trE_{(y-k)\delta} 
-\ha^{-1}q^{-2y}[2]_q\lb \trE_{y\delta+\al_0}, \trE_{(x-1)\delta+\al_1}\rb
\\
\quad +\ha^{-y}q^{-2y}(q^4-1)
\sum_{k=1}^{y-1}\ha^k\bigl((q^4-1)\sum_{e=1}^{y-k-1}
\ha^eE_{e\delta+\al_0}\trE_{(y-k-e)\delta}
+\ha^{y-k-1}[2]_qE_{(y-k)\delta+\al_0}\bigr)\trE_{(x-1+k)\delta+\al_1}
\\
\quad +\ha^{-y}q^{-2y}(q^4-1)\sum_{k=1}^{y-1}
\ha^k\ha^{y-k}q^{2(y-k)}\trE_{(y-k)\delta}\trE_{(x+k)\delta}
+\ha^{-1}q^{-2y}[2]_qE_{(x+y)\delta}  \\
\quad = -\ha^{-y}q^{-2y}(q^4-1)^2
\sum_{k=1}^{y-2}\sum_{e=1}^{y-k-1}\ha^{k+e}\trE_{k\delta+\al_0}
\trE_{(y-k-e)\delta}\trE_{(x-1+e)\delta+\al_1} \\
\quad -\ha^{-y}q^{-2y}(q^4-1)\sum_{k=1}^{y-1}\ha^{y-1}[2]_qE_{k\delta+\al_0}\trE_{(y-k+x-1)\delta+\al_1}
\\
\quad -\ha^{-y}q^{-2y}(q^4-1)
\sum_{k=1}^{y-1}\ha^k\ha^{y-k}q^{2(y-k)}
\trE_{(x+k)\delta}\trE_{(y-k)\delta} 
-\ha^{-1}q^{-2y}[2]_q\lb \trE_{y\delta+\al_0}, \trE_{(x-1)\delta+\al_1}\rb
\\
\quad +\ha^{-y}q^{-2y}(q^4-1)^2
\sum_{k=1}^{y-2}\sum_{e=1}^{y-k-1}
\ha^{k+e}\trE_{e\delta+\al_0}\trE_{(y-k-e)\delta}\trE_{(x-1+k)\delta+\al_1} \\
\quad +\ha^{-y}q^{-2y}(q^4-1)
\sum_{k=1}^{y-1}\ha^{y-1}[2]_qE_{(y-k)\delta+\al_0}\trE_{(x-1+k)\delta+\al_1}
\\
\quad +\ha^{-y}q^{-2y}(q^4-1)\sum_{k=1}^{y-1}
\ha^k\ha^{y-k}q^{2(y-k)}\trE_{(y-k)\delta}\trE_{(x+k)\delta}
+\ha^{-1}q^{-2y}[2]_qE_{(x+y)\delta} \\
\quad = 
-(q^4-1)
\sum_{k=1}^{y-1}q^{-2k}
\lb \trE_{(x+k)\delta}, \trE_{(y-k)\delta} \rb \\
\quad 
-\ha^{-1}q^{-2y}[2]_q\lb \trE_{y\delta+\al_0}, \trE_{(x-1)\delta+\al_1}\rb
+\ha^{-1}q^{-2y}[2]_qE_{(x+y)\delta} \\
\quad = 
-(q^4-1)
\sum_{k=1}^{y-1}q^{-2k}
\lb \trE_{(x+k)\delta}, \trE_{(y-k)\delta} \rb
+ yq^{-2(y-1)}\lb \trE_\delta, \trE_{(x+y-1)\delta} \rb
 \\
\quad = 
(q^4-1)
\sum_{k=1}^{y-1}q^{-2k}
\lb \trE_{(y-k)\delta}, \trE_{(x+k)\delta} \rb
+ yq^{-2(y-1)}\lb \trE_\delta, \trE_{(x+y-1)\delta} \rb \\
\quad = 
(q^4-1)q^{-2}
\lb \trE_{(y-1)\delta}, \trE_{(x+1)\delta} \rb
+(q^4-1)
\sum_{k=2}^{y-1}q^{-2k}
\lb \trE_{(y-k)\delta}, \trE_{(x+k)\delta} \rb
+ yq^{-2(y-1)}\lb \trE_\delta, \trE_{(x+y-1)\delta} \rb  \\
\quad = 
(q^4-1)q^{-2}
\lb \trE_{(y-1)\delta}, \trE_{(x+1)\delta} \rb
+q^{-2}(q^4-1)
\sum_{k=1}^{y-1}q^{-2k}
\lb \trE_{(y-1-k)\delta}, \trE_{(x+1+k)\delta} \rb
+ yq^{-2(y-1)}\lb \trE_\delta, \trE_{(x+y-1)\delta} \rb \\
\quad = 
(q^4-1)q^{-2}
\lb \trE_{(y-1)\delta}, \trE_{(x+1)\delta} \rb
+q^{-2}(\lb \trE_{(y-1)\delta}, \trE_{(x+1)\delta} \rb-(y-1)q^{-2(y-2)}\lb \trE_\delta, \trE_{(x+y-1)\delta} \rb) \\
\quad + yq^{-2(y-1)}\lb \trE_\delta, \trE_{(x+y-1)\delta} \rb  \\
\quad = 
q^2
\lb \trE_{(y-1)\delta}, \trE_{(x+1)\delta} \rb
 + q^{-2(y-1)}\lb \trE_\delta, \trE_{(x+y-1)\delta} \rb
\\
\quad = q^{2(y-1)}\lb \trE_\delta,\trE_{(x+y-1)\delta}\rb 
+(\sum_{k=1}^{y-1}q^{2(k-1)}q^{-2(y-k)})\lb \trE_\delta, \trE_{(x+y-1)\delta}\rb \\
\quad =(\sum_{k=1}^y q^{4k-2-2y})\lb \trE_\delta, \trE_{(x+y-1)\delta}\rb \\
\quad =q^{2-2y}{\frac {q^{4y}-1} {q^4-1}}\lb \trE_\delta, \trE_{(x+y-1)\delta}\rb \\
\quad ={\frac {q^{2y}-q^{-2y}} {q^2-q^{-2}}}\lb \trE_\delta, \trE_{(x+y-1)\delta}\rb. 
\end{array}
\end{equation*} Then we have
\begin{equation} \label{eqn:commuEndelta}
\begin{array}{l}
0=\lb \trE_{y\delta},\trE_{\delta} \rb-
{\frac {q^{2y}-q^{-2y}} {q^2-q^{-2}}}\lb \trE_\delta, \trE_{y\delta}\rb \\
\quad = (1+{\frac {q^{2y}-q^{-2y}} {q^2-q^{-2}}})\lb \trE_{y\delta},\trE_{\delta}\rb \\
\quad = {\frac {q^2(q^{2(y-1)}+1)(1-q^{-2(y+1)})} {q^2-q^{-2}}}\lb \trE_{y\delta},\trE_{\delta}\rb.
\end{array}
\end{equation}
Hence, we have
\newline\newline
$\EQseven$ \quad
$\lb \trE_{x\delta},\trE_{y\delta} \rb =0$,
\newline\newline
and 
\newline\newline
$\EQeight$ \quad
$\lb \trE_{x\delta+\al_0},\trE_{y\delta+\al_1} \rb =\trE_{(x+y+1)\delta}.$

\begin{remark}
\label{remark:CndEQseven}
Equation~\eqref{eqn:commuEndelta} tells
that
$\EQseven$ follows mainly from $(q^{2(m+n)}-1)(q^{2(m+n-2)}+1)\ne 0$.
\end{remark}

\subsection{Calculation of $\Delta(\trE_{n\delta})$ and $\Delta(\trE_{n\delta+\al_i})$}

For $x$, $y\in\bZgeqo$, 
define the $\bC$-linear subspaces 
${\mathcal{V}}_{x,y}$, 
of $\trU\otimes\trU$
by
\begin{equation*}
{\mathcal{V}}_{x,y}
=\oplus_{m,n,r,l\in\bZgeqo}\trUp_{m\delta+(x+r)\al_0}\trK_{n\delta+(y+l)\al_1}\otimes \trUp_{n\delta+(y+l)\al_1}.
\end{equation*}

\begin{lemma} \label{lemma:PosCopOne} 
Let $n\in\bN$. Then we have the following{\rm{:}}
\newline
{\rm{(1)}}
\begin{equation} \label{eqn:DelEndl}
\begin{array}{l}
\Delta(\trE_{n\delta})= \\
\quad \trE_{n\delta}\otimes 1
+(q-q^{-1})q^2 a\left(\sum\limits_{k=1}^{n-1}
\trE_{k\delta}\trK_{(n-k)\delta}\otimes \trE_{(n-k)\delta}
\right)
+\trK_{n\delta}\otimes \trE_{n\delta}
+X
\end{array}
\end{equation} for some $X\in{\mathcal{V}}_{1,1}$.
\newline\newline
{\rm{(2)}} 
\begin{equation} \label{eqn:DelEndlalzero}
\begin{array}{l}
\Delta(\trE_{n\delta+\al_0})= \\
\quad \trE_{n\delta+\al_0}\otimes 1
+(q-q^{-1})\left(\sum\limits_{k=0}^{n-1}
(q^2 a)^{-(n-k-1)}
\trE_{k\delta+\al_0}\trK_{(n-k)\delta}\otimes \trE_{(n-k)\delta}
\right)
+\trK_{n\delta+\al_0}\otimes \trE_{n\delta+\al_0}
+Y_0
\end{array}
\end{equation} for some $Y_0\in{\mathcal{V}}_{2,1}$.
\newline\newline
{\rm{(3)}}
\begin{equation} \label{eqn:DelEndlalone}
\begin{array}{l}
\Delta(\trE_{n\delta+\al_1})= \\
\quad \trE_{n\delta+\al_1}\otimes 1
+(q-q^{-1})\left(\sum\limits_{k=0}^{n-1}
(q^2 a)^{-(n-k-1)}
\trE_{(n-k)\delta}\trK_{k\delta+\al_1}\otimes 
\trE_{k\delta+\al_1}\right)
+\trK_{n\delta+\al_1}\otimes \trE_{n\delta+\al_1}
+Y_1
\end{array}
\end{equation} for some $Y_1\in{\mathcal{V}}_{1,2}$.
\newline\newline
{\rm{(4)}} We have
\begin{equation} \label{eqn:NonZeroPosRt}
\trE_{n\delta}\ne 0,\, \trE_{n\delta+\al_0}\ne 0,\,
\trE_{n\delta+\al_1}\ne 0.
\end{equation}
\end{lemma}
\begin{proof}
{\rm{(1)}}-{\rm{(3)}}
We can prove them in induction on $n$
using the following equations.
\begin{equation}\label{eqn;DeltaEdelta}
\begin{array}{l}
\Delta(\trE_\delta)=\Delta(\lb E_0,E_1\rb) \\
\quad = \lb E_0\otimes 1+\trK_{\al_0}\otimes E_0,
E_1\otimes 1+\trK_{\al_1}\otimes E_1 \rb \\
\quad = \trE_\delta\otimes 1+ 
(1-q^{-4})E_0K_{\al_1}\otimes E_1 +\trK_\delta\otimes \trE_\delta \\
\quad = \trE_\delta\otimes 1+ 
q^{-2}(q^2-q^{-2})\trE_{\al_0}\trK_{\al_1}\otimes \trE_{\al_1} +\trK_\delta\otimes \trE_\delta
\end{array}
\end{equation}

\begin{equation*}
\begin{array}{l}
\Delta(\trE_{\delta+\al_0})=
{\frac 1 {[2]}} \Delta(\lb \trE_{\al_0},\trE_\delta \rb) \\
\quad = {\frac 1 {[2]}} \lb E_0\otimes 1+\trK_{\al_0}\otimes E_0,
\Delta(\trE_\delta) \rb \\
\quad = \Bigl( \trE_{\delta+\al_0}\otimes 1 
+
{\frac {q^{-2}(q^2-q^{-2})} {[2]}}(1-q^2a\cdot a^{-1}q^{-4})E_0^2K_{\al_1}\otimes E_1
+{\frac 1 {[2]}} (1-q^2 a\cdot a^{-1}q^{-2}) E_0K_\delta \otimes \trE_\delta \Bigr) \\
\quad \quad +\Bigl( 0+
{\frac {q^{-2}(q^2-q^{-2})} {[2]}}E_0K_\delta\otimes (q^2 E_0E_1-q^2a E_1E_0)
+\trK_{\delta+\al_0}\otimes \trE_{\delta+\al_0}\Bigr) \\
\quad = \trE_{\delta+\al_0}\otimes 1
+q^{-3}(q-q^{-1})^2E_{\al_0}^2K_{\al_1}\otimes \trE_{\al_1}
+(q-q^{-1})\trE_{\al_0}\trK_\delta\otimes \trE_\delta
+\trK_{\delta+\al_0}\otimes \trE_{\delta+\al_0}
\end{array}
\end{equation*}

\begin{equation*}
\begin{array}{l}
\Delta(\trE_{\delta+\al_1})={\frac 1 {[2]}} \Delta(\lb \trE_\delta , \trE_{\al_1} \rb) \\
\quad ={\frac 1 {[2]}} \lb \Delta(\trE_\delta), E_1\otimes 1+\trK_{\al_1}\otimes E_1 \rb \\
\quad = \Bigl(
\trE_{\delta+\al_1}\otimes 1+{\frac {q^2-q^{-2}} {[2]}}\trE_\delta \trK_{\al_1}\otimes \trE_{\al_1}+0
\Bigl)\\
\quad \quad +\Bigl( 0+
{\frac {1-q^{-4}} {[2]}}(1-aq^2\cdot a^{-1}q^{-4})
E_0K_{\al_1}^2\otimes E_1^2
+\trK_{\delta+\al_1}\otimes \trE_{\delta+\al_1} \Bigr) \\
\quad = 
\trE_{\delta+\al_1}\otimes 1
+(q-q^{-1})\trE_\delta \trK_{\al_1}\otimes \trE_{\al_1}
+q^{-3}(q-q^{-1})^2E_{\al_0}\trK_{2\al_1}\otimes \trE_{\al_1}^2
+\trK_{\delta+\al_1}\otimes \trE_{\delta+\al_1} 
\end{array}
\end{equation*}
\newline
{\rm{(4)}} The claim follows from
\eqref{eqn:DsumUpm}
and \eqref{eqn:DelEndl},
\eqref{eqn:DelEndlalzero},
\eqref{eqn:DelEndlalone},
\eqref{eqn;DeltaEdelta}.
\end{proof}

\subsection{Modifications $\tE_{k\delta}$ of $\trE_{k\delta}$}

\begin{lemma} \label{lemma:mcUmcUpr}
Let ${\mathcal{G}}$ be 
the $\bC$-subalgebra (with $1$) of $\trUp$
generated by $\trE_{p\delta}$ $(p\in\bN)$.
Let ${\mathcal{G}}^\prime:=\sum_{u=0}
^\infty{\mathcal{G}}\trE_{u\delta+\al_1}$
$(\subset \trUp)$.
Then ${\mathcal{G}}^\prime=
\oplus_{k=0}^\infty{\mathcal{G}}\trE_{k\delta+\al_1}$
as a $\bC$-linear subspace of $\trUp$.
Moreover for every $k\in\bZgeqo$,
the map 
$f_k:{\mathcal{G}}\to {\mathcal{G}}\trE_{k\delta+\al_1}$
defined by 
$f_k(X):=XE_{k\delta+\al_1}$
$(X\in{\mathcal{G}})$ is 
a $\bC$-linear isomorphism.
\end{lemma}
\begin{proof}
By \eqref{eqn:NonZeroPosRt},
we have $\trE_{k\delta+\al_1}\ne 0$
$(k\in\bZgeqo)$.

Let $r\in\bN$ and
$k_t\in\bZgeqo$,
$X_t\in{\mathcal{G}}$
$(t\in\fkJ_{1,r})$,
Assume that $k_t<k_{t+1}$
$(t\in\fkJ_{1,r-1})$
and $\sum_{t=1}^r X_tE_{k_t\delta+\al_1}=0$.
For $t\in\fkJ_{1,r}$,
let $l_t\in\bN$, $m_{t,u}\in\bZgeqo$, 
$X_{t,u}\in{\mathcal{G}}\cap \trUp_{m_{t,u}\delta}$ 
$(u\in\fkJ_{1,l_t})$
be such that $m_{t,u}<m_{t,u+1}$ $(u\in\fkJ_{1,l_t-1})$
and 
$X_t=\sum_{u=1}^{l_t}X_{t,u}$.

By \eqref{eqn:DelEndl} and \eqref{eqn:DelEndlalone},
we have
\begin{equation*}
\begin{array}{l}
0=\Delta(0)=
\Delta(\sum_{t=1}^r X_tE_{k_t\delta+\al_1}) \\
\quad
= \sum_{t=1}^r\sum_{u=1}^{l_t}
\trE_{k_t\delta+\al_1}\trK_{m_{t,u}\delta}\otimes X_{t,u}
+Y
\end{array}
\end{equation*}
for some $Y\in{\mathcal{V}}_{0,1}$.
Hence $X_{t,u}=0$
$(t\in\fkJ_{1,r},\,u\in\fkJ_{1,l_t})$,
which implies $X_t=0$ $(t\in\fkJ_{1,r})$.
Hence the statement holds.
\end{proof}

Let $z$ be an indeterminate. 
In $\bC[[z]]$, we have:
\begin{align}
& \log(1+z)=\sum_{n=1}^\infty (-1)^{n+1}{\frac {z^n} n},\quad
\log(1-z)=-\sum_{n=1}^\infty {\frac {z^n} n}, \\
& \exp(\sum_{n=1}^\infty {\frac {z^n} n})={\frac 1 {1-z}}
=\sum_{k=0}^\infty z^k,\quad
\exp(-\sum_{n=1}^\infty {\frac {z^n} n})=1-z.
\end{align}

Let
\begin{equation*}
A=1+(q-q^{-1})\sum_{k=1}^\infty \trE_{k\delta}
(q^2a)^{-(k-1)}z^k,
\end{equation*} where
$z$ is an indeterminate.
Define $\tE_k\in \trUp_{k\delta}$ $(k\in\bN)$ by
\begin{equation*}
A=\exp((q-q^{-1})\sum_{k=1}^\infty\tE_{k\delta}z^k).
\end{equation*}
Taking derivative with respect to $z$, we have
 \begin{equation*}
 (q-q^{-1})\sum_{k=1}^\infty \trE_{k\delta}
(q^2a)^{-(k-1)}kz^{k-1}
=((q-q^{-1})\sum_{k=1}^\infty\tE_{k\delta}kz^{k-1})\cdot A
\end{equation*} which implies 
\begin{equation}
\label{eqn;DifEkdeltilEkdel}
(q^2a)^{-(k-1)}kE_{k\delta}
=k\tE_{k\delta}+(q-q^{-1})
\sum_{r=1}^{k-1}r(q^2a)^{-(k-r-1)}
\tE_{r\delta}\trE_{(k-r)\delta}
\end{equation} 
\begin{lemma} \label{lemma:expanA}
We have 
\newline\newline
$\EQnine$ \quad\quad\quad\quad\quad\quad\quad\quad  $\lb \tE_{k\delta}, \trE_{n\delta+\al_1}\rb
= {\frac {[2k]_q} k} \trE_{(n+k)\delta+\al_1}$, 
\newline\newline
$\EQten$ \quad\quad\quad\quad\quad\quad\quad\quad  $\lb \trE_{n\delta+\al_0}, \tE_{k\delta}\rb
= {\frac {[2k]_q} k} \trE_{(n+k)\delta+\al_0}$,
\newline\newline for $k\in\bN$ and $n\in\bZgeqo$.
\end{lemma}
\begin{proof}
We prove $\EQnine$.
Let ${\mathcal{G}}$,
${\mathcal{G}}^\prime$
be those of Lemma~\ref{lemma:mcUmcUpr}.
We prove $[\tE_{k\delta},\trE_{\al_1}\trL_{\al_1}^{-1}]
= {\frac {[2k]_q} k} \trE_{k\delta+\al_1}\trL_{\al_1}^{-1}$.
Define the $\bC$-linear homomorphism 
$T:{\mathcal{G}}^\prime\to{\mathcal{G}}^\prime$
by
\begin{equation*}
T(XE_{\al_1+u\delta}):=XE_{\al_1+(u+1)\delta}
\quad (u\in\bZgeqo,\,
X
\in{\mathcal{G}}).
\end{equation*} 
Define the $\bC$-linear homomorphism 
$T^\prime:{\mathcal{G}}^\prime \trL_{\al_1}^{-1}
\to{\mathcal{G}}^\prime \trL_{\al_1}^{-1}$
by $T^\prime(YL_{\al_1}^{-1}):=T(Y)\trL_{\al_1}^{-1}$
$(Y\in{\mathcal{G}}^\prime)$.
Notice that
$T^\prime({\mathrm{ad}}\tE_{k\delta})_{|{\mathcal{G}}^\prime}
=({\mathrm{ad}}\tE_{k\delta})_{|{\mathcal{G}}^\prime}T^\prime$.
Fix $n\in\bN$, and
assume we have proved 
$\lb \tE_{k\delta}, \trE_{u\delta+\al_1}\rb={\frac {[2k]_q} k}T^kE_{\al_1+u\delta}$
for $k\in\fkJ_{0,n-1}$
and $u\in\bZgeqo$.
Let $\equiv_n$ mean the equality modulo $z^{n+1}$.
By an induction on $n$, we have:
\begin{align*}
& A^{-1}\trE_{\al_1+u\delta}\trL_{\al_1}^{-1}A \\
 & \equiv_n 
(\exp\left(-(q-q^{-1})\sum_{k=1}^\infty z^k
{\mathrm{ad}}\tE_{k\delta}\right)(\trE_{\al_1+u\delta}\trL_{\al_1}^{-1}) \\
 & \equiv_n 
(\exp\left(-\sum_{k=1}^\infty z^k{\frac {q^{2k}-q^{-2k}} k} T^k\right)\trE_{\al_1+u\delta})\trL_{\al_1}^{-1} \\
& \quad\quad
+z^n{\frac {q^{2n}-q^{-2n}} n} T^nE_{\al_1+u\delta}\trL_{\al_1}^{-1}
-z^n(q-q^{-1}){\mathrm{ad}}\tE_{n\delta}(\trE_{\al_1+u\delta}\trL_{\al_1}^{-1}) \\
& \equiv_n
\left((1-q^2 z T)(\sum_{r=0}^\infty z^rq^{-2r}T^r)\trE_{\al_1+u\delta}-z^n(q-q^{-1})(\lb\tE_{n\delta},\trE_{\al_1+u\delta}\rb-
{\frac {[2n]_q} n} T^nE_{\al_1+u\delta})\right)\trL_{\al_1}^{-1} \\
& \equiv_n
\left(\trE_{\al_1+u\delta}+\sum_{r=1}^\infty z^r(q^{-2r}-q^{2-2(r-1)})T^rE_{\al_1+u\delta}
-z^n(q-q^{-1})(\lb\tE_{n\delta},\trE_{\al_1+u\delta}\rb-
{\frac {[2n]_q} n} T^nE_{\al_1+u\delta})\right)\trL_{\al_1}^{-1}  \\
& \equiv_n
\left(\trE_{\al_1+u\delta}+\sum_{r=1}^\infty z^rq^{-2r}(1-q^4)T^rE_{\al_1+u\delta}-z^n(q-q^{-1})(\lb\tE_{n\delta},\trE_{\al_1+u\delta}\rb-
{\frac {[2n]_q} n} T^nE_{\al_1+u\delta})\right)\trL_{\al_1}^{-1}.
\end{align*}
Hence 
\begin{align*}
& \trE_{\al_1+u\delta}\left(1+(q-q^{-1})\sum_{k=1}^\infty \trE_{k\delta}
q^2\ha z^k\right)\trL_{\al_1}^{-1} \\
& \quad \equiv_n 
\trE_{\al_1+u\delta}\trL_{\al_1}^{-1}\left(1+(q-q^{-1})\sum_{k=1}^\infty \trE_{k\delta}
(q^2\ha)^{-(k-1)}z^k\right) \\
& \quad \equiv_n \trE_{\al_1+u\delta}\trL_{\al_1}^{-1} A
\\
& \quad \equiv_n 
A\left(\trE_{\al_1+u\delta}+\sum_{r=1}^\infty z^rq^{-2r}(1-q^4)T^rE_{\al_1+u\delta}-z^n(q-q^{-1})(\lb\tE_{n\delta},\trE_{\al_1+u\delta}\rb-
{\frac {[2n]_q} n} T^nE_{\al_1+u\delta})\right)\trL_{\al_1}^{-1}
\\
& \quad \equiv_n 
\left(1+(q-q^{-1})\sum_{k=1}^\infty \trE_{k\delta}
(q^2\ha)^{-(k-1)}z^k\right) \\
& \quad\quad\quad\quad
\cdot\left(\trE_{\al_1+u\delta}+\sum_{r=1}^\infty z^rq^{-2r}(1-q^4)T^rE_{\al_1+u\delta}-z^n(q-q^{-1})(\lb\tE_{n\delta},\trE_{\al_1+u\delta}\rb-
{\frac {[2n]_q} n} T^nE_{\al_1+u\delta})\right)\trL_{\al_1}^{-1}.
\end{align*}
Hence 
\begin{align*}
&-(q-q^{-1})(q^2\ha)^{-(n-1)}\lb \trE_{n\delta}, \trE_{\al_1+u\delta} \rb \\
& \quad =
q^{-2n}(1-q^4)T^nE_{\al_1+u\delta}
+(q-q^{-1})\sum_{k=1}^{n-1}
(q^2\ha)^{-(k-1)}q^{-2(n-k)}(1-q^4)\trE_{k\delta}T^{n-k}\trE_{\al_1+u\delta} \\
& \quad\quad
-(q-q^{-1})(\lb\tE_{n\delta},\trE_{\al_1+u\delta}\rb-
{\frac {[2n]_q} n} T^nE_{\al_1+u\delta}).
\end{align*}
By $\EQfive$, we have $\EQnine$.
\end{proof}

\subsection{Calculation of $\Delta(\tE_{n\delta})$}

\begin{lemma} \label{lemma:PosCopTwo}
We have
\begin{equation} \label{eqn:DeltEndel}
\Delta(\tE_{n\delta})
=\tE_{n\delta}\otimes 1
+\trK_{n\delta}\otimes \tE_{n\delta}+X_n\quad (n\in\bN)
\end{equation} for some $X_n\in{\mathcal{V}}_{1,1}$.
\end{lemma}
\begin{proof}
Fix $n\in\bN$, and
assume we have proved 
\eqref{eqn:DeltEndel}
with $k$ in place of $n$
for $k\in\fkJ_{1,n-1}$.
Let $\equiv_n$ mean the equality modulo 
$z^{n+1}\bC[[z]]{\mathcal{V}}_{0,0}
+\bC[[z]]{\mathcal{V}}_{1,1}$.
By an induction on $n$, we have:
\begin{align*}
&1+(q-q^{-1})\sum_{k=1}^\infty 
(q^2a)^{-(k-1)}z^k \Delta(\trE_{k\delta}) \\
& \quad \equiv_n 
\exp((q-q^{-1})\sum_{k=1}^\infty z^k\Delta(\tE_{k\delta}))
\\ 
& \quad \equiv_n 
\exp((q-q^{-1})\sum_{k=1}^\infty z^k
(\tE_{k\delta}\otimes 1
+\trK_{k\delta}\otimes \tE_{k\delta})) \\
& \quad\quad\quad\quad  
+(q-q^{-1})z^n(\Delta(\tE_{n\delta})
-(\tE_{n\delta}\otimes 1
+\trK_{n\delta}\otimes \tE_{n\delta}))) \\
& \quad \equiv_n 
\exp((q-q^{-1})\sum_{k=1}^\infty z^k
(\tE_{k\delta}\otimes 1))\cdot
\exp((q-q^{-1})\sum_{k=1}^\infty z^k(\trK_{\delta}^k\otimes \tE_{k\delta})) \\
& \quad\quad\quad\quad  
+(q-q^{-1})z^n(\Delta(\tE_{n\delta})
-(\tE_{n\delta}\otimes 1
+\trK_{n\delta}\otimes \tE_{n\delta}))) \\
& \quad \equiv_n 
\left(1+(q-q^{-1})\sum_{k=1}^\infty 
(q^2a)^{-(k-1)}z^k (\trE_{k\delta}\otimes 1)\right) \\
& \quad\quad\quad\quad \cdot
\left(1+(q-q^{-1})\sum_{k=1}^\infty 
(q^2a)^{-(k-1)}z^k (\trK_{k\delta}\otimes \trE_{k\delta})\right) \\
& \quad\quad\quad\quad\quad\quad  
+(q-q^{-1})z^n(\Delta(\tE_{k\delta})
-(\tE_{n\delta}\otimes 1
+\trK_{n\delta}\otimes \tE_{n\delta}))).
\end{align*}
From the coefficients of $z^n$ of the first and last sides
the above equations,
we obtain
\begin{align*}
& (q-q^{-1})(q^2a)^{-{(n-1)}}\Delta(\trE_{n\delta}) \\
& \quad =
(q-q^{-1})(q^2a)^{-{(n-1)}}(\trE_{n\delta}\otimes 1)
+(q-q^{-1})^2(q^2a)^{-{(n-2)}}(\sum\limits_{k=1}^{n-1}
\trE_{k\delta}\trK_{(n-k)\delta}\otimes \trE_{(n-k)\delta}) \\
& \quad\quad
+(q-q^{-1})(q^2a)^{-{(n-1)}}(\trK_{n\delta}\otimes \trE_{n\delta})
\\
& \quad\quad\quad\quad  
+(q-q^{-1})(\Delta(\tE_{n\delta})
-(\tE_{n\delta}\otimes 1
+\trK_{n\delta}\otimes \tE_{n\delta}))) + Z.
\end{align*}
for some $Z\in{\mathcal{V}}_{1,1}$.
Hence by \eqref{eqn:DelEndl},
we have \eqref{eqn:DeltEndel}
for this $n$.
\end{proof}

\begin{remark}
The above formulas on commutation relations and comultiplication of root vectors can also be obtained by Remark 3.10 in \cite{AH02} and Corollary 4.1.12 in \cite{HS20} quickly, from known formulas in the single parameter case, if we define multiparameter quantum groups from Nichols algebras.
\end{remark}

\section{Negative part $\trUm$}

\subsection{Root vectors $\trF_{k\delta+\al_0}$, $\trF_{k\delta+\al_1}$ and $\trF_{(k+1)\delta}$ with $k\geq 0$}

Let
\begin{equation*}
\blb X,Y\brb:=XY-\chi(\mu,\lambda)YX \quad
(X\in \trU_\lambda, Y\in \trU_\mu).
\end{equation*}

Define:
\begin{equation*}
\left\{\begin{array}{l}
\trF_{\al_1}:=\trF_1, \trF_{\al_0}:=\trF_0, \\
\trF_{n\delta+\al_1}:={\frac 1 {[2]_q}}
\blb \trF_\delta,\trF_{(n-1)\delta+\al_1}\brb \quad (n\in\bN), \\
\trF_{n\delta}:=\blb \trF_{\al_0},\trF_{(n-1)\delta+\al_1}\brb \quad (n\in\bN), \\
\trF_{n\delta+\al_0}:={\frac 1 {[2]_q}}
\blb \trF_{(n-1)\delta+\al_0}, \trF_\delta\brb\quad (n\in\bN). 
\end{array}\right.
\end{equation*}
Let $\bha:=\ha^{-1}q^{-4}$.
By the same argument as that for $\EQone-\EQseven$, we have 
\newline\newline
\mbox{$\FQone$}\quad\quad\quad\quad\quad\quad
$\blb \trF_{(n+1)\delta+\al_1}, \trF_{n\delta+\al_1}\brb =0 
\quad(n\in\bZgeqo)$,
\newline\newline
\mbox{$\FQtwo$}\quad
$\blb \trF_{(n+r)\delta+\al_1}, \trF_{n\delta+\al_1}\brb
+\bha^{r-1} q^{2(r-1)}\blb \trF_{(n+1)\delta+\al_1}, \trF_{(n+r-1)\delta+\al_1}\brb
=0 \quad (n\in\bZgeqo, r\in\bN)$,
\newline\newline
\mbox{$\FQthree$}\quad\quad\quad\quad\quad\quad
$\blb \trF_{n\delta+\al_0}, \trF_{(n+1)\delta+\al_0}\brb =0 
\quad(n\in\bZgeqo)$,
\newline\newline
\mbox{$\FQfour$}\,\,
$\blb \trF_{n\delta+\al_0}, \trF_{(n+r)\delta+\al_0}\brb
+\bha^{r-1} q^{2(r-1)}\blb \trF_{(n+r-1)\delta+\al_0}, \trF_{(n+1)\delta+\al_0}\brb
=0 \quad (n\in\bZgeqo, r\in\bN)$,
\newline\newline
$\FQfive$ \quad
$\blb \trF_{r\delta}, \trF_{n\delta+\al_1}\brb  =
(q^4-1)\sum_{k=1}^{r-1}\bha^k 
\trF_{(r-k)\delta}\trF_{(n+k)\delta+\al_1}+\bha^{r-1}[2]_qF_{(n+r)\delta+\al_1}$,
\newline\newline
$\FQsix$ \quad
$\blb \trF_{n\delta+\al_0}, \trF_{r\delta}\brb  =
(q^4-1)\sum_{k=1}^{r-1}\bha^k 
\trF_{(n+k)\delta+\al_0}\trF_{(r-k)\delta}+\bha^{r-1}[2]_qF_{(n+r)\delta+\al_0}$,
\newline\newline
$\FQseven$ \quad
$\blb \trF_{x\delta},\trF_{y\delta} \brb =0$,
\newline\newline
and 
\newline\newline
$\FQeight$ \quad
$\blb \trF_{x\delta+\al_0},\trF_{y\delta+\al_1} \brb = \trF_{(x+y+1)\delta}$.

\vspace{1cm}

Let
\begin{equation*}
{\bar A}:=1+(q-q^{-1})\sum_{k=1}^\infty \trF_{k\delta}
(q^2\bha)^{-(k-1)}z^k,
\end{equation*} where
$z$ is an indeterminate.
Define $\tF_{k\delta}\in \trUm_{-k\delta}$ $(k\in\bN)$ by
\begin{equation*}
{\bar A}=\exp((q-q^{-1})\sum_{k=1}^\infty\tF_{k\delta}z^k).
\end{equation*}

We have 
\newline\newline
$\FQnine$ \quad\quad\quad\quad\quad\quad\quad\quad  $\blb \tF_{k\delta}, \trF_{n\delta+\al_1}\brb
= {\frac {[2k]_q} k} \trF_{(n+k)\delta+\al_1}$, 
\newline\newline
$\FQten$ \quad\quad\quad\quad\quad\quad\quad\quad  $\blb \trF_{n\delta+\al_0}, \tF_{k\delta}\brb
= {\frac {[2k]_q} k} \trF_{(n+k)\delta+\al_0}$,
\newline\newline for $k\in\bN$ and $n\in\bZgeqo$.

\begin{definition}
Define the $\bC$-algebra automorphisms $\mcT_k:\trUqQ\to\trUqQ$
$(k\in I)$
as follows. 
Let $i,j\in I$, and suppose $i\ne j$.
Then let
\begin{align}
&\mcT_i(K_i)=K_j^{-1}(=\trK_{\al_i-\delta}),\,\,\mcT_i(K_j)=K_iK_j^2(=\trK_{\al_j+\delta}), \\
&\mcT_i(L_i)=L_j^{-1}(=\trL_{\al_i-\delta}),\,\,\mcT_i(L_j)=L_iL_j^2(=\trL_{\al_j+\delta}), \\
&\mcT_i(E_i)={\frac 1 {q-q^{-1}}}F_jL_j^{-1},\,\,
\mcT_i(E_j)=
{\frac {q^2q_{ij}} {q+q^{-1}}}
\lb E_j, \lb E_j, E_i \rb\rb, \\
&\mcT_i(F_i)=(q-q^{-1})K_j^{-1}E_j,\,\,
\mcT_i(F_j)=
{\frac {q^3q_{ij}^{-1}} {(q^4-1)(q^2-1)}}
\blb F_j, \blb F_j, F_i \brb\brb.
\end{align} 
\end{definition}
We can see the following, where $i\ne j$.
\begin{align}
& \mcT_i(\lb E_j,E_i\rb)=\lb E_j,E_i\rb,\,\, \mcT_i(\blb F_j,F_i\brb)=\blb F_j,F_i\brb, \\
&\mcT_i^{-1}(K_i)=K_i^2K_j(=\trK_{\al_i+\delta}),\,\,\mcT_i^{-1}(K_j)=K_i^{-1}(=\trK_{\al_j-\delta}), \\
&\mcT_i^{-1}(L_i)=L_i^2L_j(=\trL_{\al_i+\delta}),\,\,\mcT_i^{-1}(L_j)=L_i^{-1}(=\trL_{\al_j-\delta}), \\
&\mcT_i^{-1}(E_i)={\frac {q^2q_{ij}} {q+q^{-1}}}\lb\lb E_j,E_i\rb, E_i\rb,\,\,
\mcT_i^{-1}(E_j)={\frac 1 {q-q^{-1}}}K_i^{-1}F_i,\\
&\mcT_i^{-1}(F_i)={\frac {q^3q_{ij}^{-1}} {(q^4-1)(q^2-1)}}\blb\blb F_j,F_i\brb, F_i\brb,\,\,
\mcT_i^{-1}(F_j)=(q-q^{-1})E_iL_i^{-1}.
\end{align} 

Define the $\bC$-algebra anti-automorphisms $\Omega:\trUqQ\to \trUqQ$ 
by $\Omega(K_i)=L_j$, $\Omega(L_i)=K_j$, $\Omega(E_i)=E_j$
and $\Omega(F_i)=F_j$, where $i\ne j$. Then $\Omega^{-1}=\Omega$
and $\Omega\mcT_k\Omega=\mcT_k^{-1}$ $(k\in I)$.

Recall $a=q_{01}$.
We have
\begin{align}
& \mcT_1^{-k}(E_1)=q^{-2k}a^{-k}\trE_{k\delta+\al_1}\quad (k\in\bZgeqo), \label{eqn:TonemkEone}
\\
& \mcT_1^{-k}(F_1)=
{\frac {q^{6k}a^k} {(q^2-1)^{2k}}}\trF_{k\delta+\al_1}\quad (k\in\bZgeqo), \label{eqn:TonemkFone} \\
& \mcT_1^k(E_1)={\frac {q^{6k-5}a^{k-1}} {(q^2-1)^{2k-1}}}
\trF_{(k-1)\delta+\al_0}\trL_{(k-1)\delta+\al_0}^{-1}\quad (k\in\bN), \label{eqn:TonemkmoFzero} \\
& \mcT_1^k(F_1)=(q^2-1)q^{-(2k-1)}a^{-(k-1)}\trK_{(k-1)\delta+\al_0}^{-1}\trE_{(k-1)\delta+\al_0}
\quad (k\in\bN). \label{eqn:TonemkmoEzero}
\end{align} 
By \eqref{eqn:TonemkEone}-\eqref{eqn:TonemkmoEzero},
we have
\begin{equation} \label{eqn:commEFkdelali}
[\trE_{k\delta+\al_i},\trF_{k\delta+\al_i}]
=q^{-4k}(q^2-1)^{2k}(-\trK_{k\delta+\al_i}+\trL_{k\delta+\al_i})
\quad (i\in I,\,k\in\bZgeqo).
\end{equation}

We have
\begin{align}
& \mcT_1(\trE_{k\delta})= \trE_{k\delta}\,\,\mbox{and}\,\,
\mcT_1(\trF_{k\delta})= \trF_{k\delta}\,(k\in\bN).
\end{align} 

We also have
\begin{align*}
& \lb \trE_{r\delta}, \mcT_1^{-n}(E_1)\rb \\
& \quad =(q^2-q^{-2})\sum_{k=1}^{r-1}q^{-2(k-1)}(q^2a)^{2k}
\trE_{(r-k)\delta}\mcT_1^{-(n+k)}(E_1)+
q^{-2(r-1)}(q^2a)^{2r-1}[2]_q\mcT_1^{-(n+r)}(E_1),
\end{align*} and
\begin{align*}
& [\mcT_1^{n+1}(F_1), \trE_{r\delta}] \\
& \quad =(q^2-q^{-2})\sum_{k=1}^{r-1}q^{-2(k-1)}(q^2a)^{2k}
\trK_{k\delta}\mcT_1^{n+k+1}(F_1)\trE_{(r-k)\delta}
+q^{-2(r-1)}(q^2a)^{2r-1}[2]_qK_{k\delta}\mcT_1^{n+k+1}(F_1).
\end{align*} 

We have
\begin{align*}
\lb \tE_{k\delta}, \mcT_1^n(E_1)\rb
& = (q^2a)^k{\frac {[2k]_q} k}\mcT_1^{n-k}(E_1), \\
\lb \tE_{k\delta}, \mcT_1^n(E_1L_{\al_1}^{-1})\rb
& = (q^2a)^k{\frac {[2k]_q} k}\mcT_1^{n-k}(E_1)\trL_{(n-k-1)\delta+\al_0}, 
\\
[\mcT_1^n(F_1),\tE_{k\delta}]&=
 (q^2a)^k{\frac {[2k]_q} k}\trK_{k\delta}\mcT_1^{n+k}(E_1).
\end{align*}

We have
\begin{align*}
&\mcT_1^r(\trE_{n\delta})=\trE_{n\delta} \\
&\quad = \lb \trE_{\al_0}, \trE_{(n-1)\delta+\al_1} \rb \\
&\quad = \lb {\frac 1 {q-q^{-1}}}\trK_{\al_0}\mcT_1(F_1),
(q^2a)^{n-1}\mcT_1^{-(n-1)}(E_1) \rb \\
&\quad = {\frac {(q^2a)^{n-1}} {q-q^{-1}}}
\lb \mcT_1(\trK_{\al_1}^{-1}F_1), \mcT_1^{-(n-1)}(E_1) \rb  \\
&\quad = {\frac {(q^2a)^{n-1}} {q-q^{-1}}}
\lb \mcT_1^k(\trK_{\al_1}^{-1}F_1), \mcT_1^{k-n}(E_1) \rb 
\end{align*} and
\begin{align*}
&\mcT_1^r(\trF_{n\delta})=\trF_{n\delta} \\
&\quad = \blb \trF_{\al_0}, \trF_{(n-1)\delta+\al_1} \brb \\
&\quad = \blb (q-q^{-1})\mcT_1(E_1L_{\al_1}^{-1}),
{\frac {(q^2-1)^{2(n-1)}} {q^{6(n-1)}a^{n-1}}}
\mcT_1^{-n+1}(F_1) \brb \\
&\quad = {\frac {(q^2-1)^{2n-1}} {q^{6n-5}a^{n-1}}}
\blb  \mcT_1(E_1L_{\al_1}^{-1}), \mcT_1^{-n+1}(F_1) \brb  \\
&\quad = {\frac {(q^2-1)^{2n-1}} {q^{6n-5}a^{n-1}}}
\blb  \mcT_1^k(E_1L_{\al_1}^{-1}), \mcT_1^{k-n}(F_1) \brb .
\end{align*}

\subsection{Modifications $\tF_{k\delta}$ of $\trF_{k\delta}$}

For $n\in\bN$, 
let 
let $\dF_{n\delta}:=q^2\blb  \mcT_1(E_1L_{\al_1}^{-1}), \mcT_1^{-n+1}(F_1) \brb$. We have
\begin{equation*}
\dF_{n\delta}={\frac {q^{6n-3}a^{n-1}} {(q^2-1)^{2n-1}}}\trF_{n\delta}
\quad (n\in\bN).
\end{equation*}

\begin{lemma} \label{lemma:tEfFndel}
For $k,n\in\bN$ and $m\in\bZ$, we have that 
\begin{align*}
& [\tE_{k\delta}, 
\blb  \mcT_1^m(E_1L_{\al_1}^{-1}), \mcT_1^{m-n}(F_1) \brb] \\
& \quad = {\frac {[2k]_q} k}\Bigl(
-\trK_{k\delta}\blb\mcT_1^m(E_1L_{\al_1}^{-1}),\mcT_1^{m-n+k}(F_1)\brb 
+\trL_{k\delta}\blb\mcT_1^{m-k}(E_1L_{\al_1}^{-1}),\mcT_1^{m-n}(F_1)\brb \Bigr) \\
& \quad =
\left\{\begin{array}{l}
{\frac {[2k]_q} k}(-\trK_{k\delta}+\trL_{k\delta})\blb\mcT_1^m(E_1L_{\al_1}^{-1}),\mcT_1^{m-n+k}(F_1)\brb
\quad \mbox{if $n>k$,} \\
{\frac {[2k]_q} k}q^{-2}
(-\trK_{k\delta}+\trL_{k\delta}) \quad \mbox{if $n=k$,}  \\
0 \quad \mbox{if $k>n$.}
\end{array}\right.
\end{align*} 
In particular, for $k,n\in\bN$, we have
\begin{align}
& [\tE_{k\delta}, 
\dF_{n\delta}] \nonumber
\\
& \quad =
\left\{\begin{array}{l}
{\frac {[2k]_q} k}(-\trK_{k\delta}+\trL_{k\delta})\dF_{(n-k)\delta}
\quad \mbox{if $n>k$,} \\
{\frac {[2k]_q} k}
(-\trK_{k\delta}+\trL_{k\delta}) \quad \mbox{if $n=k$,}  \\
0 \quad \mbox{if $k>n$.}
\end{array}\right.
\label{eqn:tEfFndelone}
\end{align} 
\end{lemma}
\begin{proof}
We have
\begin{align*}
& [\tE_{k\delta}, 
\blb  \mcT_1^m(E_1L_{\al_1}^{-1}), \mcT_1^{m-n}(F_1) \brb] \\
& \quad = [\tE_{k\delta}, 
[  \mcT_1^m(E_1L_{\al_1}^{-1}), \mcT_1^{m-n}(F_1) ]_{(q^2a)^{-m}q^{-2}}]  \\
& \quad = 
(q^2a)^k{\frac {[2k]} k}\Bigl(\bigl(
\mcT_1^{m-k}(E_1L_{\al_1}^{-1})\trL_{k\delta}
\mcT_1^{m-n}(F_1)
+(-1)\mcT_1^m(E_1L_{\al_1}^{-1})\trK_{k\delta}
\mcT_1^{m-n+k}(F_1)\bigr)  \\
& \quad\,\, -(q^2a)^{-n}q^{-2}
\bigl((-1)\trK_{k\delta}
\mcT_1^{m-n+k}(F_1)\mcT_1^m(E_1L_{\al_1}^{-1})
+\mcT_1^{m-n}(F_1)\mcT_1^{m-k}(E_1L_{\al_1}^{-1})\trL_{k\delta}
\bigl)\Bigl) \\
& \quad =(q^2a)^k{\frac {[2k]} k}\Bigl(
(-1)(q^2a)^{-k}\trK_{k\delta}
[\mcT_1^m(E_1L_{\al_1}^{-1}),\mcT_1^{m-n+k}(F_1)]_{(q^2a)^{-n+k}q^{-2}} \\
& \quad\,\,
+(q^2a)^{-k}
[\mcT_1^{m-k}(E_1L_{\al_1}^{-1}),\mcT_1^{m-n}(F_1)]_{(q^2a)^{-n+k}q^{-2}}
\trL_{k\delta}\Bigr) \\
& \quad = {\frac {[2k]} k}\Bigl(
(-1)\trK_{k\delta}\blb\mcT_1^m(E_1L_{\al_1}^{-1}),\mcT_1^{m-n+k}(F_1)\brb \\
& \quad\,\,
+\trL_{k\delta}\blb\mcT_1^{m-k}(E_1L_{\al_1}^{-1}),\mcT_1^{m-n}(F_1)\brb \Bigr).
\end{align*}

Since $\mcT_1^r(\trF_{n\delta})=\trF_{n\delta}$ for $r\in\bZ$, for $n>k\geq 1$, we have
\begin{align*}
& [\tE_{k\delta}, 
\blb  \mcT_1^m(E_1L_{\al_1}^{-1}), \mcT_1^{m-n}(F_1) \brb] \\
& \quad ={\frac {[2k]} k}(-\trK_{k\delta}+\trL_{k\delta})\blb\mcT_1^m(E_1L_{\al_1}^{-1}),\mcT_1^{m-n+k}(F_1)\brb.
\end{align*}

We have 
\begin{align*}
& \blb  \mcT_1^m(E_1L_{\al_1}^{-1}), \mcT_1^m(F_1) \brb \\
& \quad =[\mcT_1^m(E_1L_{\al_1}^{-1}), \mcT_1^m(F_1) ]_{q^{-2}} \\
& \quad =[\mcT_1^m(E_1)\trL_{\al_1-m\delta}^{-1}, \mcT_1^m(F_1) ]_{q^{-2}} \\
& \quad =q^{-2}[\mcT_1^m(E_1), \mcT_1^m(F_1) ]\trL_{\al_1-m\delta}^{-1} \\
& \quad =q^{-2}\mcT_1^m([E_1, F_1])\trL_{\al_1-m\delta}^{-1} \\
& \quad =q^{-2}\mcT_1^m(-\trK_{\al_1}+\trL_{\al_1})\trL_{\al_1-m\delta}^{-1} \\
& \quad =q^{-2}(-\trK_{\al_1-m\delta}+\trL_{\al_1-m\delta})\trL_{\al_1-m\delta}^{-1} \\
& \quad =q^{-2}(-\trK_{\al_1-m\delta}\trL_{\al_1-m\delta}^{-1}+1).
\end{align*} Hence we have
\begin{align*}
& [\tE_{k\delta}, 
\blb  \mcT_1^m(E_1L_{\al_1}^{-1}), \mcT_1^{m-k}(F_1) \brb] \\
& \quad ={\frac {[2k]} k}q^{-2}\Bigl(
(-1)\trK_{k\delta}(-\trK_{\al_1-m\delta}\trL_{\al_1-m\delta}^{-1}+1)
+\trL_{k\delta}(-\trK_{\al_1-(m-k)\delta}\trL_{\al_1-(m-k)\delta}^{-1}+1) \Bigr)
\\
& \quad ={\frac {[2k]} k}q^{-2}
(-\trK_{k\delta}+\trL_{k\delta}).
\end{align*} 

We have
\begin{align*}
& \blb \mcT_1^m(E_1L_{\al_1}^{-1}), \mcT_1^{m+k}(F_1) \brb \\
& \quad = [\mcT_1^m(E_1L_{\al_1}^{-1}), \mcT_1^{m+k}(F_1)]_{(q^2a)^kq^{-2}} \\
& \quad = [\mcT_1^m(E_1)\trL_{\al_1-m\delta}^{-1}, \mcT_1^{m+k}(F_1)]_{(q^2a)^kq^{-2}} \\
& \quad = (q^2a)^kq^{-2}[\mcT_1^m(E_1), \mcT_1^{m+k}(F_1)]\trL_{\al_1-m\delta}^{-1} \\
& \quad = q^{-2}\trL_{\al_1-m\delta}^{-1} [\mcT_1^m(E_1), \mcT_1^{m+k}(F_1)]
\end{align*} and 
\begin{align*}
& \lb \mcT_1^{m+k}(\trK_{\al_1}^{-1}F_1), \mcT_1^m(E_1) \rb \\
& \quad = [ \mcT_1^{m+k}(\trK_{\al_1}^{-1}F_1), \mcT_1^m(E_1) ]_{(q^2a)^kq^{-2}} \\
& \quad = [ \trK_{\al_1-(m+k)\delta}^{-1}\mcT_1^{m+k}(F_1), \mcT_1^m(E_1) ]_{(q^2a)^kq^{-2}} \\
& \quad = \trK_{\al_1-(m+k)\delta}^{-1}[\mcT_1^{m+k}(F_1), \mcT_1^m(E_1) ].
\end{align*} Hence we have
\begin{align*}
& \blb \mcT_1^m(E_1L_{\al_1}^{-1}), \mcT_1^{m+k}(F_1) \brb \\
& \quad = -q^{-2}\trK_{\al_1-(m+k)\delta}\trL_{\al_1-m\delta}^{-1}
\lb \mcT_1^{m+k}(\trK_{\al_1}^{-1}F_1), \mcT_1^m(E_1) \rb.
\end{align*}

Since $\mcT_1^r(\trE_{n\delta})=\trE_{n\delta}$ for $r\in\bZ$, for $k>n\geq 1$, we have
\begin{align*}
& [\tE_{k\delta}, 
\blb  \mcT_1^m(E_1L_{\al_1}^{-1}), \mcT_1^{m-n}(F_1) \brb] \\
& \quad = {\frac {[2k]} k}\Bigl(
(-1)\trK_{k\delta}\blb\mcT_1^m(E_1L_{\al_1}^{-1}),\mcT_1^{m-n+k}(F_1)\brb \\
& \quad\,\,
+\trL_{k\delta}\blb\mcT_1^{m-k}(E_1L_{\al_1}^{-1}),\mcT_1^{m-n}(F_1)\brb \Bigr) \\
& \quad = {\frac {[2k]} k}\Bigl(
q^{-2}\trK_{k\delta}\trK_{\al_1-(m-n+k)\delta}\trL_{\al_1-m\delta}^{-1}
\lb \mcT_1^{m-n+k}(\trK_{\al_1}^{-1}F_1), \mcT_1^m(E_1) \rb \\
& \quad\,\,
-q^2L_{k\delta}\trK_{\al_1-(m-n)\delta}\trL_{\al_1-(m-k)\delta}^{-1}
\lb \mcT_1^{m-n}(\trK_{\al_1}^{-1}F_1), \mcT_1^{m-k}(E_1) \rb \Bigr) \\
& \quad = 0.
\end{align*}
\end{proof}

\begin{lemma}\label{lemma:nrmalFhd}
Define $\dtF_{k\delta}\in \trUm_{-k\delta}$ $(k\in\bN)$ by
\begin{equation}\label{eqn:nrmalFhdOne}
\exp (\sum_{k=1}^\infty z^k\dtF_{k\delta})
=1+\sum_{r=1}^\infty z^r\dF_{r\delta},
\end{equation}
where $z$ in an indeterminate. Then we have
\begin{equation}\label{eqn:nrmalFhdTwo}
[\tE_{k\delta},\dtF_{r\delta}]
=\delta_{k,r}{\frac {[2k]_q} k}
(-\trK_{k\delta}+\trL_{k\delta})\quad (k,r\in\bN).
\end{equation}
Moreover, we have
\begin{equation}\label{eqn:nrmalFhdThree}
\dtF_{r\delta}
={\frac {q^{2r}} {(q-q^{-1})^{2r-1}}}\tF_{r\delta}
\quad (r\in\bN).
\end{equation}
\end{lemma} 
\begin{proof}

Let $B$ be the both sides of  $\eqref{eqn:nrmalFhdOne}$.
We prove \eqref{eqn:nrmalFhdTwo}.
We use a similar argument as in 
Proof of Lemma\ref{lemma:expanA}.
Let $\tE_0:=\dF_0:=1$
$(\in\trU)$.
Let ${\mathcal{G}}^\prime$ be
the $\bC$-linear
subspace
of $\trU$
with the $\bC$-basis
$\{\trK_{a\delta}\trL_{b\delta}\dF_{r\delta}
\,|\,r,a,b\in\bZgeqo\}$,
and let ${\mathcal{G}}:=
\oplus_{k=0}^\infty\tE_{k\delta}{\mathcal{G}}^\prime$
the $\bC$-linear
subspace
of $\trU$
with the $\bC$-basis
$\{\tE_{k\delta}\trK_{a\delta}\trL_{b\delta}\dF_{r\delta}
\,|\,k,r,a,b\in\bZgeqo\}$.
For $l\in\bN$,
define the
$\bC$-linear map 
${\tilde{T}}_l:{\mathcal{G}}\to{\mathcal{G}}$
by 
\begin{equation*}
{\tilde{T}}_l(\tE_{k\delta}Y)
:=\delta_{k,l}{\frac {[2l]_q} l}
(-\trK_{l\delta}+\trL_{l\delta})Y\quad (k\in\bZgeqo,\,
Y\in{\mathcal{G}}^\prime).
\end{equation*}
Then ${\tilde{T}}_l({\mathrm{ad}}\dtF_{r\delta})_{|{\mathcal{G}}}
=({\mathrm{ad}}\dtF_{r\delta})_{|{\mathcal{G}}}{\tilde{T}}_l$
$(l,r\in\bN)$.
Fix $n\in\bN$, and
assume that we have proved 
$[X,\dtF_{l\delta}]
={\frac {[2l]_q} l}{\tilde{T}}_l(X)$
for $X\in{\mathcal{G}}$.
Let $\equiv_n$ mean the equality modulo $z^{n+1}$.
By an induction on $n$, we have:
\begin{align*}
& B\tE_{k\delta}B^{-1} \\
 & \equiv_n 
(\exp\left(\sum_{r=1}^\infty z^r
{\mathrm{ad}}\dtF_{r\delta}\right))(\tE_{k\delta}) \\
 & \equiv_n 
(\exp\left(-\sum_{l=1}^\infty z^l{\tilde{T}}_l\right)(\tE_{k\delta}) 
+z^n{\tilde{T}}_n(\tE_{k\delta})
+z^n({\mathrm{ad}}\dtF_{n\delta})(\tE_{k\delta}) \\
& \equiv_n
\tE_{k\delta}-z^k{\frac {[2k]_q} k}(-\trK_{k\delta}+\trL_{k\delta})
+z^n({\tilde{T}}_n(\tE_{k\delta})
-[\tE_{k\delta},\dtF_{n\delta}]).
\end{align*}
Hence 
\begin{align*}
& \left(1+\sum_{r=1}^\infty z^r\dF_{r\delta}
\right) \tE_{k\delta} \\
& \quad \equiv_n \left(
\tE_{k\delta}-z^k{\frac {[2k]_q} k}(-\trK_{k\delta}+\trL_{k\delta})
+z^n({\tilde{T}}_n(\tE_{k\delta})
-[\tE_{k\delta},\dtF_{n\delta}])
\right) \\
& \quad\quad\quad\quad \cdot \left(1+\sum_{r=1}^\infty z^r\dF_{r\delta}
\right). 
\end{align*}
By \eqref{eqn:tEfFndelone}, we have
$[\tE_{k\delta},\dF_{n\delta}]
-({\tilde{T}}_n(\tE_{k\delta})
-[\tE_{k\delta},\dtF_{n\delta}])
=[\tE_{k\delta},\dF_{n\delta}]$.
Hence 
we have \eqref{eqn:nrmalFhdTwo}.

We have the following:

\begin{align*}
&
\exp (\sum_{k=1}^\infty z^k\dtF_{k\delta}) \\
& \quad =1+\sum_{r=1}^\infty z^r\dF_{r\delta} \\
& \quad =1+\sum_{r=1}^\infty z^r
{\frac {q^{6r-3}a^{r-1}} {(q^2-1)^{2r-1}}}\trF_{r\delta} \\
& \quad =1+\sum_{r=1}^\infty z^r
{\frac {q^{6r-3}a^{r-1}} {(q^2-1)^{2r-1}}}
(q^2a)^{-(r-1)}
(q^2{\bar a})^{-(r-1)}\trF_{r\delta} \\
& \quad =1+(q-q^{-1})\sum_{r=1}^\infty z^r
{\frac {q^{6r-2}a^{r-1}} {(q^2-1)^{2r}}}
(q^2a)^{-(r-1)}
(q^2{\bar a})^{-(r-1)}\trF_{r\delta} \\
& \quad =1+(q-q^{-1})\sum_{r=1}^\infty z^r
{\frac {q^{4r}} {(q^2-1)^{2r}}}
(q^2{\bar a})^{-(r-1)}\trF_{r\delta} \\
& \quad =
\exp ((q-q^{-1})\sum_{k=1}^\infty z^k
{\frac {q^{4k}} {(q^2-1)^{2k}}}\tF_{k\delta}) \\
& \quad =
\exp (\sum_{k=1}^\infty z^k
{\frac {q^{2k}} {(q-q^{-1})^{2k-1}}}\tF_{k\delta}).
\end{align*}
Hence we have \eqref{eqn:nrmalFhdThree}.
\end{proof}

\subsection{Coproduct for negative part}
For $x$, $y\in\bZgeqo$, 
define the $\bC$-linear subspaces 
${\mathcal{V}}^\prime_{x,y}$, 
of $\trU\otimes\trU$
by
\begin{equation*}
{\mathcal{V}}^\prime_{x,y}
=\oplus_{m,n,r,l\in\bZgeqo}
\trUm_{-(n\delta+(y+l)\al_1)}\otimes
\trUm_{-(m\delta+(x+r)\al_0)}\trL_{n\delta+(y+l)\al_1}.
\end{equation*}

By the same arguments for 
Lemmas~\ref{lemma:PosCopOne} 
and \ref{lemma:PosCopTwo}, 
we have:

\begin{lemma} \label{lemma:NegCopOne} 
Let $n\in\bN$. Then we have the following{\rm{:}}
\newline
{\rm{(1)}}
\begin{equation} \label{eqn:NegDelEndl}
\begin{array}{l}
\Delta(\trF_{n\delta})= \\
\quad 1\otimes \trF_{n\delta}
+(q-q^{-1})q^2 {\bar{a}}\left(\sum\limits_{k=1}^{n-1}
\trF_{(n-k)\delta}\otimes \trF_{k\delta}\trL_{(n-k)\delta}
\right)
+\trF_{n\delta}\otimes \trL_{n\delta}
+X
\end{array}
\end{equation} for some $X\in{\mathcal{V}}^\prime_{1,1}$.
\newline\newline
{\rm{(2)}} 
\begin{equation} \label{eqn:NegDelEndlalzero}
\begin{array}{l}
\Delta(\trF_{n\delta+\al_0})= \\
\quad 1\otimes\trF_{n\delta+\al_0}
+(q-q^{-1})\left(\sum\limits_{k=0}^{n-1}
(q^2{\bar{a}})^{-(n-k-1)}
\trF_{(n-k)\delta}\otimes\trF_{k\delta+\al_0}\trL_{(n-k)\delta}
\right)
+\trF_{n\delta+\al_0}\otimes\trL_{n\delta+\al_0} 
+Y_0
\end{array}
\end{equation} for some $Y_0\in{\mathcal{V}}^\prime_{1,2}$.
\newline\newline
{\rm{(3)}}
\begin{equation} \label{eqn:NegDelEndlalone}
\begin{array}{l}
\Delta(\trF_{n\delta+\al_1})= \\
\quad 1 \otimes\trF_{n\delta+\al_1}
+(q-q^{-1})\left(\sum\limits_{k=0}^{n-1}
(q^2{\bar{a}})^{-(n-k-1)} 
\trF_{k\delta+\al_1}\otimes\trF_{(n-k)\delta}\trL_{k\delta+\al_1}\right)
+\trF_{n\delta+\al_1}\otimes \trL_{n\delta+\al_1}
+Y_1
\end{array}
\end{equation} for some $Y_1\in{\mathcal{V}}^\prime_{2,1}$.
\newline\newline
{\rm{(4)}}
We have
\begin{equation} \label{eqn:NegDeltEndel}
\Delta(\tF_{n\delta})
=1\otimes \tE_{n\delta}
+\tF_{n\delta}\otimes\trL_{n\delta} +X_n\quad 
\end{equation} for some $X_n\in{\mathcal{V}}^\prime_{1,1}$.
\end{lemma}

\section{Universal ${\mathrm{R}}$-matrix}
\subsection{Drinfeld paring}
By a well-known argument called the Drinfeld
quantum double construction,
we have the following.
We have the 
bilinear homomorphism
\begin{equation} \label{eqp:DEFpair}
\langle\,,\,\rangle:\trUpgeq\times\trUmleq\to\bC
\end{equation}
satisfying the following $(DP1)$-$(DP3)$.
We also denote by $\langle\,,\,\rangle$
the bilinear map from $(\trUpgeq\otimes\trUpgeq)
\times(\trUmleq\otimes\trUmleq)$ to $\bC$ defined by
\begin{equation*}
\langle X_1\otimes X_2,Y_1\otimes Y_2\rangle
=\langle X_1,Y_1\rangle\cdot\langle X_2,Y_2\rangle
\quad (X_r\in\trUpgeq,\,Y_r\in\trUmleq\,\,(r\in\fkJ_{1,2})).
\end{equation*}
Let $\tau:\trU\otimes\trU\to\trU\otimes\trU$
be the $\bC$-algebra automorphism defined by
$\tau(Z_1\otimes Z_2)=Z_2\otimes Z_1$
$(Z_1$, $Z_2\in\trU)$.
\newline\newline
$(DP1)$ \quad $\langle X_1X_2,Y\rangle
=\langle X_1\otimes X_2,(\tau\circ\Delta)(Y)\rangle$, \quad 
$\langle X, Y_1Y_2\rangle
=\langle \Delta(X),Y_1\otimes Y_2\rangle$. \newline
$(DP2)$ \quad $\langle S(X),Y\rangle=\langle X,S^{-1}(Y)\rangle$,
$\langle 1,Y\rangle=\varepsilon(Y)$,
$\langle X,1\rangle=\varepsilon(X)$. \newline
$(DP3)$ \quad $\langle \trK_\lambda, \trL_\mu\rangle
=\chi(\lambda,\mu)$, $\langle \trE_i, \trF_j\rangle=\delta_{ij}$, 
 $\langle \trK_\lambda, \trF_j\rangle=0$, 
 $\langle \trE_i, \trL_\mu\rangle=0$. 
\newline\par It is easy to see:
\begin{equation}
\langle
X\trK_\lambda,
Y\trL_\mu
\rangle
=\chi(\lambda,\mu)\cdot 
\delta_{\lambda^\prime,\mu^\prime}\langle X, Y\rangle
\quad (\lambda,\mu\in\bZ\Pi,\,
\lambda^\prime,\mu^\prime\in\bZgeqo\Pi,\,X\in\trUp_{\lambda^\prime},Y\in\trUm_{-\mu^\prime}).
\end{equation}
Then we have the following lemma by \cite{D86}.
 \begin{lemma} \label{lemma:DrinfeldQDCone}
 Let $\Delta^{(2)}=(\rmid\otimes\Delta)\cdot\Delta$.
 Then for $X\in\trUpgeq$ and $Y\in\trUmleq$,
 letting $\Delta^{(2)}(X)=\sum_{u=1}^kX_{1,u}\otimes X_{2,u} \otimes X_{3,u}$
 $(X_{x,u}\in\trUpgeq)$
 and $\Delta^{(2)}(Y)=\sum_{v=1}^rY_{1,v}\otimes Y_{2,v} \otimes Y_{3,v}$
  $(Y_{y,v}\in\trUmleq)$, 
 we have
 \begin{equation*}
 \begin{array}{lcl}
 YX &= & \sum\limits_{u=1}^k\sum\limits_{v=1}^r\langle S^{-1}(X_{1,u}), Y_{1,v}\rangle\langle X_{3,u}, Y_{3,v}\rangle
 X_{2,u}Y_{2,v}, \\
XY &= & \sum\limits_{u=1}^k\sum\limits_{v=1}^r\langle X_{1,u}, Y_{1,v}\rangle\langle S^{-1}(X_{3,u}), Y_{3,v}\rangle
 Y_{2,v}X_{2,u}.
 \end{array}
 \end{equation*}
\end{lemma}
By Lemma~\ref{lemma:DrinfeldQDCone}, we have the following.

 \begin{lemma} \label{lemma:DrinfeldQDCtwo}
Let $\lambda\in\bZ\Pi$.
 Define the $\bC$-linear homomorphism
 $f_\lambda:\trU\to\bC$ by the following 
 {\rm{(i)}}-{\rm{(ii)}}{\rm{:}} \newline\newline
 {\rm{(i)}} $f_\lambda(\trF_iZ)=f_\lambda(Z\trE_i)=0$ for  $i\in I$
 and $Z\in\trU$. \newline
 {\rm{(ii)}} $f_\lambda(\trK_\mu\trL_\nu)=    
 \delta_{0,\mu}\delta_{\lambda,\nu}$
 for $\mu$, $\nu\in\bZ\Pi$. \newline\newline
 Then for $X\in\trUp_\lambda$ and
 $Y\in\trUm_{-\lambda}$,
 we have 
\begin{equation} \label{eqn:DrinfeldQDCtwoEqn}
 \langle X, Y\rangle=f_\lambda(XY)
\quad \mbox{for}\quad 
X\in\trUp_\lambda
\quad \mbox{and}\quad 
 Y\in\trUm_{-\lambda}.
 \end{equation}
 \end{lemma}

Using \eqref{eqn:DrinfeldQDCtwoEqn},
by \eqref{eqn:commEFkdelali}
(resp. \eqref{eqn:nrmalFhdTwo}), we have
the following \eqref{eqn:pairndelta} (resp. \eqref{eqn:pairndelta}).
\begin{equation} \label{eqn:pairEFkdelali}
\langle \trE_{k\delta+\al_i},\trF_{k\delta+\al_i}\rangle
=q^{-4k}(q^2-1)^{2k}
\quad (i\in I,\,k\in\bZgeqo).
\end{equation}
\begin{equation} \label{eqn:pairndelta}
\langle\tE_{k\delta},\dtF_{r\delta}\rangle
=
\delta_{k,r}{\frac {[2k]_q} k}
\quad (k,r\in\bN).
\end{equation}
By \eqref{eqn:nrmalFhdThree} and \eqref{eqn:pairndelta}, we have
\begin{equation} \label{eqn:pairnNODOTdelta}
\langle\tE_{k\delta},\tF_{r\delta}\rangle
=
\delta_{k,r}{\frac {[2k]_q} k}\cdot {\frac {(q^2-1)^{2k-1}} {q^{4k-1}}}
\quad (k,r\in\bN).
\end{equation}
\newcommand{\oprod}{\mathop{\overset{\rightarrow}{\prod}}\limits}
\newcommand{\loprod}{\mathop{\overset{\leftarrow}{\prod}}\limits}
\newcommand{\mfP}{{\mathfrak{P}}^{{\mathrm{fin}}}}
\newcommand{\mmfP}[1]{\mfP(#1)}

\subsection{PBW-Theorem}
\label{section:PBW-TH}

For a subset $M$ of $\bZ$, let $\mmfP{M}$
be the subset of the power set ${\mathfrak{P}}(\bZ)$
of $\bZ$ composed of all the finite subsets of $\bZ$,
where note $\emptyset\in\mmfP{M}$.

For ${\mathcal{Y}}\in\mmfP{\bZ}$,
let the symbols $\oprod_{y\in{\mathcal{Y}}}$
and $\loprod_{y\in{\mathcal{Y}}}$ mean as in the 
following $(Prod1)$-$(Prod2)$.
If ${\mathcal{Y}}\ne\emptyset$, let $X_y\in\trU$ $(y\in{\mathcal{Y}})$.
\newline
$(Prod1)$ \quad If ${\mathcal{Y}}=\emptyset$,
define $\oprod_{y\in{\mathcal{Y}}} X_y= 
\loprod_{y\in{\mathcal{Y}}} X_y=1_\trU$. \newline 
$(Prod2)$  \quad If ${\mathcal{Y}}\ne\emptyset$,
then letting $z\in{\mathcal{Y}}$ be such that 
$z\geq y$ for all $y\in{\mathcal{Y}}$,  
\begin{equation*}
\mbox{define \,\, $\oprod_{y\in{\mathcal{Y}}} X_y
=(\oprod_{t\in{\mathcal{Y}}\setminus\{z\}} X_t)X_z$
 \,\, and \,\, $\loprod_{y\in{\mathcal{Y}}} X_y
=X_z(\loprod_{t\in{\mathcal{Y}}\setminus\{z\}} X_t)$  \,\, in an induction with $|{\mathcal{Y}}|$.}
\end{equation*}
Examples are as follows. Let ${\mathcal{Y}}=\{2,5,8,9,12\}$. Then 
we mean $\oprod_{y\in{\mathcal{Y}}} X_y
=X_2X_5X_8X_9X_{12}$ and
$\loprod_{y\in{\mathcal{Y}}} X_y
=X_{12}X_9X_8X_5X_2$.
Moreover, letting $k_2=3$, $k_5=0$, $k_8=2$, 
$k_9=1$, $k_{12}=4$ $(\in\bZgeqo)$,
we mean
 $\oprod_{y\in{\mathcal{Y}}} X_y^{k_y}
=X_2^3X_8^2X_9X_{12}^4$ and
$\loprod_{y\in{\mathcal{Y}}} X_y^{k_y}
=X_{12}^4X_9X_8^2X_2^3$. 
\newcommand{\gvmc}[1]{{\grave{{\mathcal{#1}}}}}
\newcommand{\acmc}[1]{{\acute{{\mathcal{#1}}}}}
\newcommand{\gvone}[1]{{\grave{#1}}}
\newcommand{\acone}[1]{{\acute{#1}}}
\newcommand{\gvtwo}[2]{{\gvone{#1}}_{\gvone{#2}}}
\newcommand{\actwo}[2]{{\acone{#1}}_{\acone{#2}}}

\begin{theorem} \label{theorem:valueDrinfeldpair}
Let ${\gvmc{X}}$, ${\gvmc{Y}}$, 
${\acmc{X}}$, ${\acmc{Y}}$ ${\in\mmfP{\bZgeqo}}$
and ${\gvmc{Z}}$, ${\acmc{Z}}$ ${\in\mmfP{\bN}}$.
Let 
${\gvtwo{m}{x}}$, ${\gvtwo{r}{z}}$, ${\gvtwo{n}{y}}$, 
${\actwo{m}{x}}$, ${\actwo{r}{z}}$, ${\actwo{n}{y}}$
be an element of $\bZgeqo$
for  
${\gvone{x}}\in{\gvmc{X}}$, 
${\gvone{z}}\in{\gvmc{Z}}$, 
${\gvone{y}}\in{\gvmc{Y}}$, 
${\acone{x}}\in{\acmc{X}}$, 
${\acone{z}}\in{\acmc{Z}}$, 
${\acone{y}}\in{\acmc{Y}}$
respectively.
Then
we have
\begin{equation} \label{eqn:valueDrinfeldpairEqn}
\begin{array}{l}
\langle
\left(\oprod_{{\gvone{x}}\in{\gvmc{X}}} \trE_{{\gvone{x}}\delta+\al_1}^{{\gvtwo{m}{x}}}\right)
\cdot
\left(\oprod_{{\gvone{z}}\in{\gvmc{Z}}} \tE_{{\gvone{z}}\delta}^{{\gvtwo{r}{z}}}\right)
\cdot
\left(\loprod_{{\gvone{y}}\in{\gvmc{Y}}} \trE_{{\gvone{y}}\delta+\al_0}^{{\gvtwo{n}{y}}}\right),
\left(\oprod_{{\acone{x}}\in{\acmc{X}}} \trF_{{\acone{x}}\delta+\al_1}^{{\actwo{m}{x}}}\right)
\cdot
\left(\oprod_{{\acone{z}}\in{\acmc{Z}}} \tF_{{\acone{z}}\delta}^{{\actwo{r}{z}}}\right)
\cdot
\left(\loprod_{{\acone{y}}\in{\acmc{Y}}} \trF_{{\acone{y}}\delta+\al_0}^{{\actwo{n}{y}}}\right)\rangle \\
=\left(\delta_{{\gvmc{X}},{\acmc{X}}}\cdot
\prod\limits_{{\gvone{x}}\in{\gvmc{X}}}\left(
\delta_{{\gvtwo{m}{x}},{\acone{m}}_{\gvone{x}}}
({\gvtwo{m}{x}})_{q^2}!
\langle\trE_{{\gvone{x}}\delta+\al_1},\trF_{{\gvone{x}}\delta+\al_1}\rangle^{\gvtwo{m}{x}}
\right)\right) \\
\quad\quad\cdot
\left(\delta_{{\gvmc{Z}},{\acmc{Z}}}\cdot
\prod\limits_{{\gvone{z}}\in{\gvmc{Z}}}\left(
\delta_{{\gvtwo{r}{z}},{\acone{r}}_{\gvone{z}}}
{\gvtwo{r}{z}}!
\langle\tE_{{\gvone{z}}\delta},
\tF_{{\gvone{z}}\delta}\rangle^{\gvtwo{r}{z}}
\right)\right) \\
\quad\quad\cdot
\left(\delta_{{\gvmc{Y}},{\acmc{Y}}}\cdot
\prod\limits_{{\gvone{y}}\in{\gvmc{Y}}}\left(
\delta_{{\gvtwo{n}{y}},{\acone{n}}_{\gvone{y}}}
({\gvtwo{n}{y}})_{q^2}!
\langle\trE_{{\gvone{y}}\delta+\al_0},\trF_{{\gvone{y}}\delta+\al_0}\rangle^{\gvtwo{n}{y}}
\right)\right) \\
=\left(\delta_{{\gvmc{X}},{\acmc{X}}}\cdot
\prod\limits_{{\gvone{x}}\in{\gvmc{X}}}\left(
\delta_{{\gvtwo{m}{x}},{\acone{m}}_{\gvone{x}}}
({\gvtwo{m}{x}})_{q^2}!
q^{-4{\gvone{x}}{\gvtwo{m}{x}}}
(q^2-1)^{2{\gvone{x}}{\gvtwo{m}{x}}}
\right)\right)
\\
\quad\quad\cdot\left(
\delta_{{\gvmc{Z}},{\acmc{Z}}}\cdot
\prod\limits_{{\gvone{z}}\in{\gvmc{Z}}}\left(
\delta_{{\gvtwo{r}{z}},{\acone{r}}_{\gvone{z}}}
{\gvtwo{r}{z}}!
([2{\gvone{z}}]_q)^{{\gvtwo{r}{z}}}
{\gvone{z}}^{-{\gvtwo{r}{z}}}
q^{-{\gvtwo{r}{z}}(4{\gvone{z}}-1)}
(q^2-1)^{{\gvtwo{r}{z}}(2{\gvone{z}}-1)}
\right)\right) \\
\quad\quad\cdot
\left(\delta_{{\gvmc{Y}},{\acmc{Y}}}\cdot
\prod\limits_{{\gvone{y}}\in{\gvmc{Y}}}\left(
\delta_{{\gvtwo{n}{y}},{\acone{n}}_{\gvone{y}}}
({\gvtwo{n}{y}})_{q^2}!
q^{-4{\gvone{y}}{\gvtwo{n}{y}}}
(q^2-1)^{2{\gvone{y}}{\gvtwo{n}{y}}}
\right)\right).
\end{array}
\end{equation} 
\end{theorem}
\begin{proof}
For convenience of notation, we denote 
$\left(\oprod_{{\gvone{x}}\in{\gvmc{X}}} \trE_{{\gvone{x}}\delta+\al_1}^{{\gvtwo{m}{x}}}\right)
,
\left(\oprod_{{\gvone{z}}\in{\gvmc{Z}}} \tE_{{\gvone{z}}\delta}^{{\gvtwo{r}{z}}}\right)
,
\left(\loprod_{{\gvone{y}}\in{\gvmc{Y}}} \trE_{{\gvone{y}}\delta+\al_0}^{{\gvtwo{n}{y}}}\right)$ by
$E({\gvmc{X}},M_{\gvmc{X}})$, $E({\gvmc{Z}},R_{\gvmc{Z}})$, $E(\gvmc{Y},N_{\gvmc{Y}})$, where $M_{\gvmc{X}}$, $R_{\gvmc{Z}}$, $N_{\gvmc{Y}}$ denote the sets $\{\gvtwo{m}{x}|\gvone{x}\in \gvmc{X}\}$, $\{\gvtwo{r}{z}|\gvone{z}\in \gvmc{Z}\}$, $\{\gvtwo{n}{y}|\gvone{y}\in \gvmc{Y}\}$ respectively. We denote 
$\left(\oprod_{{\acone{x}}\in{\acmc{X}}} \trF_{{\acone{x}}\delta+\al_1}^{{\actwo{m}{x}}}\right)
,
\left(\oprod_{{\acone{z}}\in{\acmc{Z}}} \tF_{{\acone{z}}\delta}^{{\actwo{r}{z}}}\right)
,
\left(\loprod_{{\acone{y}}\in{\acmc{Y}}} \trF_{{\acone{y}}\delta+\al_0}^{{\actwo{n}{y}}}\right)$ by $F({\acmc{X}},M_{\acmc{X}})$, $F({\acmc{Z}},R_{\acmc{Z}})$, $F(\acmc{Y},N_{\acmc{Y}})$, where $M_{\acmc{X}}$, $R_{\acmc{Z}}$, $N_{\acmc{Y}}$ denote the sets $\{\actwo{m}{x}|\acone{x}\in \acmc{X}\}$, $\{\actwo{r}{z}|\acone{z}\in \acmc{Z}\}$, $\{\actwo{n}{y}|\acone{y}\in \acmc{Y}\}$ respectively. 

At first we claim that 
\begin{equation}\label{equ:EXFFF}\langle E(\gvmc{X},M_{\gvmc{X}}),F(\acmc{X},M_{\acmc{X}})F(\acmc{Z},M_{\acmc{Z}})F(\acmc{Y},M_{\acmc{Y}})\rangle=\delta_{\emptyset,\acmc{Z}}\delta_{\emptyset,\acmc{Y}}\delta_{\gvmc{X},\acmc{X}}\prod\limits_{\gvone{x}\in \gvmc{X}}\delta_{\gvtwo{m}{x},\acone{m}_{\gvone{x}}}\langle E_{\gvone{x}\delta+\alpha_1}^{\gvtwo{m}{x}},  F_{\gvone{x}\delta+\alpha_1}^{\gvtwo{m}{x}}\rangle.\end{equation} 
By \eqref{eqn:DelEndlalone} and $(DP1)$, it is easy to see $\gvmc{Y}$, $\gvmc{Z}$ must be $\emptyset$ if we want LHS of \eqref{equ:EXFFF} to be non-zero. Then we focus on $\langle E(\gvmc{X},M_{\gvmc{X}}),F(\acmc{X},M_{\acmc{X}})\rangle$.
We prove by induction on $|\gvmc{X}|$.  Suppose $\langle E(\gvmc{X},M_{\gvmc{X}}),F(\acmc{X},M_{\acmc{X}})\rangle\neq 0$. If $|\gvmc{X}|=1$, assume that $\gvmc{X}=\{m\}$. If $|\acmc{X}|\ge 2$, denote $\max \gvmc{X}$ by $n$, then 
$F(\acmc{X},M_{\acmc{X}})=F(\acmc{X}\backslash \{n\},M_{\acmc{X}}\backslash \{\acone{m}_{n}\})\cdot F_{n\delta+\alpha_1}^{\acone{m}_{n}}$ and $n>m$, but by $(DP1)$ and \eqref{eqn:DelEndlalone}, this is impossible, hence $\gvmc{X}=\acmc{X}$ and $M_{\gvmc{X}}=M_{\acmc{X}}$, for degree reason. Then we suppose that the claim holds for $|\gvmc{X}|\leq k$. Take $\gvmc{X}$, $M_{\gvmc{X}}$ s.t. $|\gvmc{X}|=k+1$ and $\acmc{X}$, $M_{\acmc{X}}$. Denote $\max\gvmc{X}$ by $s$, $\max \acmc{X}$ by t, then 
$E(\gvmc{X},M_{\gvmc{X}})=E(\gvmc{X}\backslash \{s\},M_{\gvmc{X}}\backslash \{\gvone{m}_s\})\cdot E_{s\delta+\alpha_1}^{\gvone{m}_s}$,
$F(\acmc{X},M_{\acmc{X}})=F(\acmc{X}\backslash \{s\},M_{\acmc{X}}\backslash \{\acone{m}_s\})\cdot E_{s\delta+\alpha_1}^{\acone{m}_s}$. 
Since $\langle E(\gvmc{X},M_{\gvmc{X}}), F(\acmc{X},M_{\acmc{X}})\rangle\neq 0$, by $(DP1)$, \eqref{eqn:DelEndlalone} and degree reason, we have $t\le s$. Similarly, by $(DP2)$ and \eqref{eqn:NegDelEndlalone}, we have $t\ge s$, then we have $s=t$. By the same argument, we can get $\gvone{m}_s=\acone{m}_s$. Then
 by induction hypothesis, we get that 
\eqref{equ:EXFFF} holds.

Similarly, we can get
\begin{equation}\label{equ:EYFFF}\langle E(\gvmc{Y},N_{\gvmc{Y}}),F(\acmc{X},M_{\acmc{X}})F(\acmc{Z},R_{\acmc{Z}})F(\acmc{Y},M_{\acmc{Y}})\rangle=\delta_{\emptyset,\acmc{X}}\delta_{\emptyset,\acmc{Z}}\delta_{\gvmc{Y},\acmc{Y}}\prod_{\gvone{y}\in \gvmc{Y}}\delta_{\gvtwo{n}{y},\acone{n}_{\gvone{y}}}\langle E_{\gvone{y}\delta+\alpha_1}^{\gvtwo{n}{y}},  F_{\gvone{y}\delta+\alpha_1}^{\gvtwo{n}{y}}\rangle.\end{equation} 

Then we claim that 
\begin{equation}\label{equ:EZFFF}\langle E(\gvmc{Z},R_{\gvmc{Z}}),F(\acmc{X},M_{\acmc{X}})F(\acmc{Z},R_{\acmc{Z}})F(\acmc{Y},M_{\acmc{Y}})\rangle=\delta_{\emptyset,\acmc{X}}\delta_{\emptyset,\acmc{Y}}\delta_{\gvmc{Z},\acmc{Z}}\prod_{\gvone{z}\in \gvmc{Z}}\delta_{\gvtwo{r}{z},\acone{r}_{\gvone{z}}}\langle E_{\gvone{z}\delta}^{\gvtwo{n}{y}},  F_{\gvone{z}\delta}^{\gvtwo{n}{y}}\rangle.\end{equation} 
By $(DP1)$ and Lemma~\ref{eqn:DeltEndel}, it is easy to see that if LHS of \eqref{equ:EZFFF} is non-zero, then $\acmc{X}$, $\acmc{Y}$ must be $\emptyset$.
Then we focus on $\langle E(\gvmc{Z},N_{\gvmc{Z}}),F(\acmc{Z},N_{\acmc{Z}})\rangle$. We prove by induction on $|\gvmc{Z}|$. Suppose $\langle E(\gvmc{Z},N_{\gvmc{Z}}),F(\acmc{Z},N_{\acmc{Z}})\rangle\neq 0$, if $|\gvmc{Z}|=1$, by similar methods as above using Lemma~\ref{eqn:DeltEndel} and $(DP1)$, we get that 
\begin{equation}
    \langle E(\gvmc{Z},R_{\gvmc{Z}}),F(\acmc{Z},R_{\acmc{Z}})\rangle=\delta_{\gvmc{Z},\acmc{Z}}\delta_{\gvtwo{r}{z},\actwo{r}{z}}\langle \tilde{E}_{\gvone{z}\delta}^{\gvtwo{r}{z}}, \tilde{F}_{\gvone{z}\delta}^{\gvtwo{r}{z}}\rangle.
\end{equation}
Similarly, if $|\acmc{Z}|=1$, then
\begin{equation}
    \langle E(\gvmc{Z},R_{\gvmc{Z}}),F(\acmc{Z},R_{\acmc{Z}})\rangle=\delta_{\gvmc{Z},\acmc{Z}}\delta_{\gvtwo{r}{z},\actwo{r}{z}}\langle \tilde{E}_{\acone{z}\delta}^{\actwo{r}{z}}, \tilde{F}_{\acone{z}\delta}^{\actwo{r}{z}}\rangle.
\end{equation}
Suppose the claim holds for $|\gvmc{Z}|=k$, we consider the case that $|\gvmc{Z}|=k+1$.
Denote $\max \gvmc{Z}$ by $s$, $\max \acmc{Z}$ by $t$, then $\langle E(\gvmc{Z},R_{\gvmc{Z}}),F(\acmc{Z},R_{\acmc{Z}})\rangle=\langle E(\gvmc{Z}\backslash\{s\},R_{\gvmc{Z}}\backslash \{\gvone{r}_s\})\cdot \tilde{E}_{s\delta}^{\gvone{r}_s},F(\acmc{Z}\backslash \{t\},R_{\acmc{Z}}\{\acone{r}_{t}\})\cdot \tilde{F}_{t\delta}^{\acone{r}_t}\rangle$. By the same method as the proof of \eqref{equ:EXFFF} using \eqref{eqn:DelEndlalzero}, \eqref{eqn:DelEndlalone}, Lemma~\ref{eqn:DeltEndel} and $(DP1)$, we get that $s=t$ and $\gvone{r}_s=\acone{r}_t$. Again by \eqref{eqn:DelEndlalzero}, \eqref{eqn:DelEndlalone}, Lemma~\ref{eqn:DeltEndel} and $(DP1)$ and induction hypothesis, we get that \eqref{equ:EZFFF} holds.

By \eqref{equ:EXFFF}, \eqref{equ:EYFFF}, \eqref{equ:EZFFF}, it is easy to get 
\begin{equation} 
\begin{array}{l}
\langle
\left(\oprod_{{\gvone{x}}\in{\gvmc{X}}} \trE_{{\gvone{x}}\delta+\al_1}^{{\gvtwo{m}{x}}}\right)
\cdot
\left(\oprod_{{\gvone{z}}\in{\gvmc{Z}}} \tE_{{\gvone{z}}\delta}^{{\gvtwo{r}{z}}}\right)
\cdot
\left(\loprod_{{\gvone{y}}\in{\gvmc{Y}}} \trE_{{\gvone{y}}\delta+\al_0}^{{\gvtwo{n}{y}}}\right),
\left(\oprod_{{\acone{x}}\in{\acmc{X}}} \trF_{{\acone{x}}\delta+\al_1}^{{\actwo{m}{x}}}\right)
\cdot
\left(\oprod_{{\acone{z}}\in{\acmc{Z}}} \tF_{{\acone{z}}\delta}^{{\actwo{r}{z}}}\right)
\cdot
\left(\loprod_{{\acone{y}}\in{\acmc{Y}}} \trF_{{\acone{y}}\delta+\al_0}^{{\actwo{n}{y}}}\right)\rangle \\
=\left(\delta_{{\gvmc{X}},{\acmc{X}}}\cdot
\prod\limits_{{\gvone{x}}\in{\gvmc{X}}}\left(
\delta_{{\gvtwo{m}{x}},{\acone{m}}_{\gvone{x}}}
\langle\trE_{{\gvone{x}}\delta+\al_1}^{\gvtwo{m}{x}},\trF_{{\gvone{x}}\delta+\al_1}^{\gvtwo{m}{x}}\rangle
\right)\right) \\
\quad\quad\cdot
\left(\delta_{{\gvmc{Z}},{\acmc{Z}}}\cdot
\prod\limits_{{\gvone{z}}\in{\gvmc{Z}}}\left(
\delta_{{\gvtwo{r}{z}},{\acone{r}}_{\gvone{z}}}
\langle\tE_{{\gvone{z}}\delta}^{\gvtwo{r}{z}},
\tF_{{\gvone{z}}\delta}^{\gvtwo{r}{z}}\rangle
\right)\right) \\
\quad\quad\cdot
\left(\delta_{{\gvmc{Y}},{\acmc{Y}}}\cdot
\prod\limits_{{\gvone{y}}\in{\gvmc{Y}}}\left(
\delta_{{\gvtwo{n}{y}},{\acone{n}}_{\gvone{y}}}
\langle\trE_{{\gvone{y}}\delta+\al_0}^{\gvtwo{n}{y}},\trF_{{\gvone{y}}\delta+\al_0}^{\gvtwo{n}{y}}\rangle
\right)\right). \label{equ:EEEFFF}
\end{array}
\end{equation}     
From \eqref{equ:EEEFFF}, we can get the first $``="$ of \eqref{eqn:valueDrinfeldpairEqn} by direct calculation. Then by \eqref{eqn:pairEFkdelali} and \eqref{eqn:pairnNODOTdelta}, we get the second $``="$ of \eqref{eqn:valueDrinfeldpairEqn}.
\end{proof}

Let ${\mathbb{B}}^+$
(resp. ${\mathbb{B}}^-$) be the 
elements of $\trUp$ (resp. $\trUm$) consisting of those assigned to the left (resp. right)
components of $\langle\,,\,\rangle$ of
the first parts of the equations in \eqref {eqn:valueDrinfeldpairEqn}.
Let ${\mathbb{B}}^0=
\{\trK_\lambda\trL_\mu\,|\,
\lambda,\mu\in\bZ\Pi\}$.
Then ${\mathbb{B}}^0$ is the 
$\bC$-basis of $\trUo$.

\begin{theorem} 
\label{theorem:trUthPBW}
 ${\mathbb{B}}^+$
 and ${\mathbb{B}}^-$
are the $\bC$-bases of $\trUp$
 and $\trUm$,
 respectively.
\end{theorem}
\begin{proof}
By Theorem~\ref{theorem:valueDrinfeldpair},
${\mathbb{B}}^+$ is linearly
independent over $\bC$.
Let $X=\oplus_{b\in{\mathbb{B}}^+}\bC b$
$(\subset\trUp)$.
By $\EQone$-$\EQeight$
and \eqref{eqn;DifEkdeltilEkdel},
we see that $\trE_i\cdot X\subset X$
for $i\in I$. Hence 
$X=\trUp$. Therefore
${\mathbb{B}}^+$ is the $\bC$-basis of $\trUp$.
Similarly we see that 
${\mathbb{B}}^-$ is the $\bC$-basis of $\trUm$.
\end{proof}

\subsection{Main theorem} \label{subsection:Aform}
In Subsection~\ref{subsection:Aform},
assume ${\mathrm{tr.deg}}(\bQ(q,a)/\bQ)=2$.
Define the bijection
$\eta:{\mathbb{B}}^+
\to{\mathbb{B}}^-$
by
$\langle B,\eta(B)\rangle\ne 0$. Define
the map $\Gamma:{\mathbb{B}}^+\to\bA$
by 
$\Gamma(B)=\langle B,\eta(B)\rangle$
$(B\in{\mathbb{B}}^+)$.

Let $\bA$ be the $\bQ$-subalgebra of 
$\bC$ generated by 
$a^{\pm 1}$, $q^{\pm 1}$ and
${\frac 1 {q^2-1}}$, i.e.  
\begin{equation*}
\bA=\bQ[a,a^{-1},q,q^{-1},{\frac 1 {q-q^{-1}}}].
\end{equation*}
Let $\trU_\bA$ 
be the $\bA$-subalgebra 
(with $1$) of
$\trU$
defined as
$\langle
\trK_i^{\pm 1}, \trL_i^{\pm 1}, 
\trE_i, \trF_i
\,|\,i\in I\rangle_\bA$.
We use a similar notation to 
${\rm{(Ba1)}}$ with $\trU_\bA$
instead of $\trU$.
Then we have similar facts
to ${\rm{(Ba2)}}$ (use \eqref{eqn:defEiEjKiLi}) and ${\rm{(Ba3)}}$
for $\trU_\bA$.
Note that
${\mathbb{B}}^0$ is the $\bA$-base of $\trUo_\bA$.
Note also that
${\mathbb{B}}^+$ and ${\mathbb{B}}^-$
are the $\bA$-bases of $\trUp_\bA$
and $\trUm_\bA$, respectively,
which can be proved in the same way as 
for the proof of Theorem~\ref{theorem:trUthPBW}.

Let $\zeta$, 
$\omega\in\bCt$.
Assume $\zeta^2\ne 1$.
Define the $\bQ$-algebra homomorphism $\theta_{\zeta,\omega}:\bA\to \bC$
by $\theta_{\zeta,\omega}(q)=\zeta$
and $\theta_{\zeta,\omega}(a)=\zeta^{-2}\omega$.
Let $\trmcU=\trmcU_{\zeta,\omega}$ be the Hopf algebra obtained as the specialization:
\begin{equation*}
\trmcU
=\trU_\bA\otimes\bC,\quad
(x\cdot 1_\bA)\otimes 1=1_\bA\otimes \theta_{\zeta,\omega}(x)\,\,(x\in\bA). 
\end{equation*} 
Define the
$\bA$-module homomorphism $g_{\zeta,\omega}:\trU_\bA\to\trmcU$
by $g_{\zeta,\omega}(X)=X\otimes 1$
$(X\in\trU_\bA)$.
We use a similar notation to 
${\rm{(Ba1)}}$
and ${\rm{(Ba5)}}$
with $\trmcU$
instead of $\trU$.
Then we have similar facts
to ${\rm{(Ba2)}}$ and ${\rm{(Ba3)}}$
for $\trmcU$.
Note that
$g_{\zeta,\omega}({\mathbb{B}}^0)$, 
$g_{\zeta,\omega}({\mathbb{B}}^+)$, 
$g_{\zeta,\omega}({\mathbb{B}}^-)$
are the $\bC$-bases
of 
$\trmcUo$, 
$\trmcUp$, $\trmcUm$,
respectively.

We also denote by $\langle\,,\,\rangle$
the bilinear map from 
$\trmcUpgeq\times\trmcUmleq$ to $\bC$ induced by the one of
\eqref{eqp:DEFpair}.
For linear subspaces $X$ and $Y$ of $\trmcUpgeq$ and $\trmcUmleq$
over $\bC$ respectively, 
let ${\widetilde{\trmcJpgeq}}(X,Y)=
\{x\in X\,|\,\forall y\in Y,\,
\langle x,y\rangle=0\}$
and ${\widetilde{\trmcJmleq}}(X,Y)=
\{y\in Y\,|\,\forall x\in X,\,
\langle x,y\rangle=0\}$.
Let 
\begin{equation*}
\begin{array}{l}
\trmcJpgeq=
{\widetilde{\trmcJpgeq}}(\trmcUpgeq,\trmcUmleq),\,\trmcJmleq=
{\widetilde{\trmcJmleq}}(\trmcUpgeq,\trmcUmleq),\\
\trmcJ^\prime=\rrmSpan{\bC}{\trmcJpgeq\trmcUmleq},\,
\trmcJ^{\prime\prime}=\rrmSpan{\bC}{\trmcUpgeq\trmcJmleq}, \\
\trmcJ=\trmcJ^\prime+\trmcJ^{\prime\prime},
\\
\trmcJp={\widetilde{\trmcJpgeq}}(\trmcUp,\trmcUm),\,\trmcJm={\widetilde{\trmcJmleq}}(\trmcUp,\trmcUm),
\\
\trmcJogeq=
{\widetilde{\trmcJpgeq}}(\trmcUogeq,\trmcUoleq),\,
\trmcJoleq={\widetilde{\trmcJmleq}}(\trmcUogeq,\trmcUoleq), \\
\trmcJo=\rrmSpan{\bC}{\trmcJogeq\trmcUoleq}+\rrmSpan{\bC}{\trmcUogeq\trmcJoleq}\,.
\end{array}
\end{equation*}
Then
$\trmcJpgeq$,
$\trmcJmleq$, 
$\trmcJ$, 
$\trmcJp$, $\trmcJm$, 
$\trmcJogeq$, $\trmcJoleq$,
and $\trmcJo$
are the two-sided ideals of
$\trmcUpgeq$,
$\trmcUmleq$, 
$\trmcU$, 
$\trmcUp$, $\trmcUm$, 
$\trmcUogeq$, $\trmcUoleq$,
and $\trmcUo$,
respectively,
while $\trmcJ^\prime$
and 
$\trmcJ^{\prime\prime}$
are also the 
two-sided ideals
of $\trmcU$.
Let 
$\trmfkU=
\trmfkU_{\zeta,\omega}=
\trmcU/\trmcJ$
(the quotient Hopf algebra),
and let ${\bar{g}}_{\zeta,\omega}:\trmcU\to\trmfkU$
mean the canonical map.
That is, ${\bar{g}}_{\zeta,\omega}$ is the Hopf algebra epimorphism
defined by ${\bar{g}}_{\zeta,\omega}(u)=u+\trmcJ$
$(u\in\trmcU)$.
We use a similar notation to 
${\rm{(Ba1)}}$ and ${\rm{(Ba5)}}$ with $\trmfkU$
instead of $\trU$,
and we have a similar fact
to ${\rm{(Ba2)}}$
for $\trmfkU$.
Then 
$\trmfkUpgeq$,
$\trmfkUmleq$,  
$\trmfkUp$, $\trmfkUm$, 
$\trmfkUogeq$, $\trmfkUoleq$
and $\trmfkUo$
can be naturally identified with
$\trmcUpgeq/\trmcJpgeq$,
$\trmcUmleq/\trmcJmleq$,
$\trmcUp/\trmcJp$,
$\trmcUm/\trmcJm$,
$\trmcUogeq/\trmcJogeq$,
$\trmcUoleq/\trmcJoleq$
and $\trmcUo/\trmcJo$,
respectively,
as $\bC$-algebras.
We also denote by $\langle\,,\,\rangle$
the bilinear map from 
$\trmfkUpgeq\times\trmfkUmleq$ to $\bC$ induced by the above one.
Then $\langle\,,\,\rangle$,
$\langle\,,\,\rangle_{|\trmfkUogeq\times\trmfkUoleq}$,
$\langle\,,\,\rangle_{|\trmfkUp\times\trmfkUm}$
and 
$\langle\,,\,\rangle_{|\trmfkUp_\lambda\times\trmfkUm_{-\lambda}}$
$(\lambda\in\bZgeqo\Pi)$
are non-degenerate.
For $\lambda\in\bZgeqo\Pi$, 
let $\trmfkC_\lambda$ $(\in \trmfkUp_\lambda
\otimes\trmfkUm_{-\lambda})$
be the canonical element of 
$\trmfkUp_\lambda\times\trmfkUm_{-\lambda}$
with respect to 
$\langle\,,\,\rangle_{|\trmfkUp_\lambda\times\trmfkUm_{-\lambda}}$.

For
$k\in\bZgeqo$, let 
$\trmfkU[\geq k]=
\oplus_{\lambda\in\bZgeqo\Pi,|\lambda|\geq k}
\rrmSpan{\bC}{\trmfkUm\trmfkUo\trmfkUp_\lambda}$.
Define the completion $\hattrmfkU$ of $\trmfkU$
by $\varprojlim\limits_k\trmfkU/\trmfkU[\geq k]$.
Similarly, for $n\in\bN$, 
define the completion $\hattrmfkU^{{\hat{\otimes}} n}=\hattrmfkU{\hat{\otimes}}\cdots {\hat{\otimes}}\hattrmfkU$ ($n$-times) of $\trmfkU^{\otimes n}$ 
(and also $\hattrmfkU^{\otimes n}$ ) by
\begin{equation} \label{eqqn;multicomp}
\hattrmfkU^{{\hat{\otimes}} n}
=\varprojlim\limits_k\trmfkU^{\otimes n}/{\mathfrak{F}}^{(n)}_k,
\quad\mbox{where we let
${\mathfrak{F}}^{(n)}_k=\sum_{r_1+\cdots+r_n=k}\trmfkU[\geq r_1 ]\otimes\cdots\otimes\trmfkU[\geq r_n]$.}
\end{equation}
We can view $\trmfkU^{\otimes n}\subset \hattrmfkU^{\otimes n}
\subset\hattrmfkU^{{\hat{\otimes}}n}$.
For $k\in\bZgeqo$, let
$\trp^{(n)}_k:\hattrmfkU^{{\hat{\otimes}}n}\to\trmfkU^{\otimes n}/{\mathfrak{F}}^{(n)}_k$
denote the natural projection onto the k-th component.
We can also regard 
$\hattrmfkU^{{\hat{\otimes}}n}$
as the complete metric space
with the metric $d(\,,\,)$
defined by
\begin{equation*}
d(x,y)
=\left\{
\begin{array}{ll}
0 & \quad (\mbox{if $x=y$}), \\
2^{-k} & 
\quad (\mbox{if 
$\trp^{(n)}_k(x-y)= 0$
and $\trp^{(n)}_{k+1}(x-y)\ne 0$ fo some $k\in\bZgeqo$})
\end{array}\right.
\end{equation*} for $x,y\in\hattrmfkU^{{\hat{\otimes}}n}$.
Let ${\hat{\Delta}}:\hattrmfkU
\to\hattrmfkU{\hat{\otimes}}\hattrmfkU$ and ${\hat{\varepsilon}}:\hattrmfkU
\to\bC$ be the
continuous $\bC$-algebra homomorphisms
induced from $\Delta$ and $\varepsilon$,
respectively. 
Define the continuous $\bC$-algebra isomorphism
${\hat{\tau}}:\hattrmfkU{\hat{\otimes}}\hattrmfkU\to\hattrmfkU{\hat{\otimes}}\hattrmfkU$ by
${\hat{\tau}}(X\otimes Y)=Y\otimes X$
$(X,Y\in\trmfkU)$.
Then we have the topological
bi-algebra $(\hattrmfkU,{\hat{\Delta}},{\hat{\varepsilon}})$.
Let
\begin{equation*}
\hattrmfkC=\sum_{\lambda\in\bZgeqo\Pi} \trmfkC_\lambda
\quad (\in\hattrmfkU{\hat{\otimes}}\hattrmfkU).
\end{equation*}
Define
the map ${\bar{\Gamma}}_{\zeta,\omega}:{\mathbb{B}}^+\to\bC$
by 
${\bar{\Gamma}}_{\zeta,\omega}(B)=\theta_{\zeta,\omega}(\Gamma(B))$ $(B\in{\mathbb{B}}^+)$.
Let ${\mathbb{B}}^+_{\zeta,\omega}=
\{B\in{\mathbb{B}}^+\,|\,{\bar{\Gamma}}_{\zeta,\omega}(B)\ne 0\}$.

Let ${\hat{g}}_{\zeta,\omega}={\bar{g}}_{\zeta,\omega}
\circ g_{\zeta,\omega}$.

\begin{theorem}\label{theorem:mainone}
We have
\begin{equation*}
\begin{array}{lcl}
\hattrmfkC & =& \sum\limits_{
\mbox{$B\in{\mathbb{B}}^+_{\zeta,\omega}$}}
{\frac 1 {{\bar{\Gamma}}_{\zeta,\omega}(B)}}{\hat{g}}_{\zeta,\omega}(B)\otimes{\hat{g}}_{\zeta,\omega}(\eta(B)) \\
 & =& 
 \left(\oprod_{
\mbox{
$x\in\bZgeqo$}
 } 
 \left(\sum\limits_{(m_x)_{\zeta^2}!\ne 0}{\frac 
 {\zeta^{4xm_x}} {(m_x)_{\zeta^2}!(\zeta^2-1)^{2xm_x}}}\left({\hat{g}}_{\zeta,\omega}(\trE_{x\delta+\al_1}^{m_x})
 \otimes {\hat{g}}_{\zeta,\omega}(\trF_{x\delta+\al_1}^{m_x})
 \right)\right)\right) \\
 &  & \cdot
 \left(\oprod_{
\mbox{
$z\in\bN,\zeta^{4z}\ne 1$}
 } 
 \left(\sum\limits_{r_z=0}^\infty
 {\frac 1 {r_z!}}\left({\frac 
 {z\zeta^{4z-1}} {[2r_z]_\zeta(\zeta^2-1)^{2z-1}}}\right)^{r_z}
 \left({\hat{g}}_{\zeta,\omega}(\tE_{z\delta}^{r_z})
 \otimes{\hat{g}}_{\zeta,\omega}(\tF_{z\delta}^{r_z})
 \right)\right)\right) \\
 & & 
 \cdot\left(\loprod_{
\mbox{
$y\in\bZgeqo$}
 } 
 \left(\sum\limits_{(n_y)_{\zeta^2}!\ne 0}{\frac 
 {\zeta^{4yn_y}} {(n_y)_{\zeta^2}!(\zeta^2-1)^{2yn_y}}}\left({\hat{g}}_{\zeta,\omega}(\trE_{y\delta+\al_0}^{n_y})
 \otimes{\hat{g}}_{\zeta,\omega}(\trF_{y\delta+\al_0}^{n_y})
 \right)\right)\right).
\end{array}
\end{equation*}
\end{theorem}

In the Drinfeld argument, we have the main theorem:
\begin{theorem} \label{theorem:maintwo}
Assume $\dim\trmfkUo<\infty$
{\rm{(i.e. $\zeta$ and $\omega$ are roots of unity}} 
{\rm{)}}.
Let $\trmfkH$ be the canonical element
of 
$\trmfkUogeq\times\trmfkUoleq$
with respect to
$\langle\,,\,\rangle_{|\trmfkUogeq\times\trmfkUoleq}$,
where see also Lemma~{\rm{\ref{lemma:calRzero}}} below. Let
\begin{equation*}
\trmfkR=\hattrmfkC\trmfkH \quad (\in\hattrmfkU{\hat{\otimes}}\hattrmfkU).
\end{equation*}
\newline
{\rm{(1)}}
$\trmfkR$ is invertible
and we have
\begin{equation*}
\trmfkR^{-1}=\sum_{\lambda\in\bZgeqo\Pi} (S\otimes\rmid)(\trmfkC_\lambda\trmfkH).
\end{equation*}
\newline
{\rm{(2)}} We have
\begin{equation} \label{eqn:maintwoEqThree}
\trmfkR{\hat{\Delta}}(X)\trmfkR^{-1}
=({\hat{\tau}}\circ{\hat{\Delta}})(X) \quad
(X\in\hattrmfkU).
\end{equation}
\newline
{\rm{(3)}} We have
\begin{equation} \label{eqn:maintwoEqFour}
({\hat{\Delta}}\otimes\rmid)(\trmfkR)=\trmfkR_{13}\trmfkR_{23},\quad
(\rmid\otimes{\hat{\Delta}})(\trmfkR)=\trmfkR_{13}\trmfkR_{12}.
\end{equation} In particular, we have
\begin{equation} \label{eqn:maintwoEqFive}
\trmfkR_{23}\trmfkR_{23}\trmfkR_{13}
=\trmfkR_{13}\trmfkR_{23}\trmfkR_{23}
\quad (\in\hattrmfkU{\hat{\otimes}}\hattrmfkU{\hat{\otimes}}\hattrmfkU).
\end{equation}
\end{theorem} 
\subsection{Structure of $\trmfkUo$ for some $\zeta$ and $\omega$}
Let $\udUo$ be 
the associative $\bC$-algebra (with $1$)
defined by the generators
$\udK{i}^{\pm 1}$, $\udL{i}^{\pm 1}$
$(i\in I)$
and the relations
$\udK{i}\udK{i}^{-1}=\udK{i}^{-1}\udK{i}=1$, 
$\udL{i}\udL{i}^{-1}=\udL{i}^{-1}\udL{i}=1$, 
$XY=YX$ $(X,Y\in\{\udK{i}^{\pm 1}, 
\udL{i}^{\pm 1}|i\in I\})$.
Then the elements $\udK{0}^{r_1}\udK{1}^{r_2}\udL{0}^{r_3}\udL{1}^{r_4}$ $(r_t\in\bZ\,(t\in\fkJ_{1,4}))$ form 
the $\bC$-basis of $\udUo$.
Regard $\udUo$ as the 
Hopf algebra
with $\Delta(\udK{i})=\udK{i}\otimes\udK{i}$,
$\Delta(\udL{i})=\udL{i}\otimes\udL{i}$.
Let $\udUogeq=\oplus_{r_1,r_2\in\bZ}\bC\udK{0}^{r_1}\udK{1}^{r_2}$ and 
$\udUoleq=\oplus_{r_3,r_4\in\bZ}\bC\udL{0}^{r_3}\udL{1}^{r_4}$.
Then $\udUogeq$ and $\udUoleq$ 
are the Hopf subalgebras of $\udUo$.

Let $x\in2\bN+1$, and
let $y\in\bN$ such that ${\mathrm{G.C.D.}}(x,y)=1$.
Let $\zeta$, $\omega\in\bCt$
be such that
$\zeta^x=1$, $\zeta^{x^\prime}\ne 1$
$(x^\prime\in\fkJ_{1,x-1})$, 
$\omega^y=1$, $\omega^{y^\prime}\ne 1$ $(y^\prime\in\fkJ_{1,y-1})$.
Namely 
$\zeta$ and $\omega$ are
$x$-th and $y$-th
roots of unity, respectively.
Let $\zeta=\xi^2$.
Then $\zeta^2$ and $\xi$ are also  primitive $x$-th
roots of unity. 
Let $\udzeta_{00}=\udzeta_{11}=\zeta^2$
and $\udzeta_{01}=\zeta^{-2}\omega$, 
$\udzeta_{10}=\zeta^{-2}\omega^{-1}$.
Define 
the $\bC$-bilinear homomorphism
$\langle\,,\,\rangle:\udUogeq\times\udUoleq\to\bC$
by 
\begin{equation*}
\langle\udK{0}^{r_1}\udK{1}^{r_2},
\udL{0}^{r_3}\udL{1}^{r_4}\rangle
=\udzeta_{00}^{r_1r_3}\udzeta_{01}^{r_1r_4}
\udzeta_{10}^{r_2r_3}\udzeta_{11}^{r_2r_4}
\quad (r_t\in\bZ\,(t\in\fkJ_{1,4})).
\end{equation*}
Then $\langle\,,\,\rangle$
satisfies the same equations
as those of $(DP1)$-$(DP2)$
with $\udUo$ in place of $\trU$.
Let $u$, $v\in\bZ$ be such that $ux+vy=1$.
Let $\uudK{1}=\udK{1}$,
$\uudK{2}=\udK{1}^{vy}\udK{0}$
and $\uudL{1}=\udL{0}$,
$\uudL{2}=\udL{0}^{vy}\udL{1}$. Then we have
\begin{equation*}
\langle\uudK{1}^{r_1}\uudK{2}^{r_2},
\uudL{1}^{r_3}\uudL{2}^{r_4}\rangle
=(\zeta^{-2}\omega^{-1})^{r_1r_3}\omega^{r_2r_4}
\quad (r_t\in\bZ\,(t\in\fkJ_{1,4})).
\end{equation*} \label{eqn:UnBaKxy}
Let $\udI$ be the two-sided ideal of 
$\udUo$ (as a $\bC$-algebra)
generated by $\uudK{1}^{xy}-1$, 
$\uudK{2}^y-1$ and 
$\uudL{1}^{xy}-1$, 
$\uudL{2}^y-1$.
\begin{equation}
\mbox{Then $\udI$ is also generated by $\udK{1}^{xy}-1$, 
$\udK{1}^y\udK{0}^y-1$ and 
$\udL{1}^{xy}-1$, 
$\udL{1}^y\udL{0}^y-1$.}  
\end{equation}
Notice that $\udI$ is a Hopf ideal of $\udUo$.
Let $\dudUo=\udUo/\udI$ (as a quotient Hopf algebra).
Let $d:\udUo\to\dudUo$ be the canonical map. 
Let $\dudK{i}=d(\udK{i})$, $\dudL{i}=d(\udL{i})$
$(i\in I)$, 
$\duudK{t}=d(\uudK{t})$, $\duudL{t}=d(\uudL{t})$ 
$(t\in\fkJ_{1,2})$
and 
$\dudUogeq=d(\udUogeq)$, $\dudUoleq=d(\udUoleq)$. 
We also denote by $\langle\,\,,\,\rangle$
the non-degenerate bilinear map 
$\dudUogeq\times\dudUoleq\to\bC$ induced from
the above one.
Let $\dudmcRo\in\dudUogeq\otimes\dudUoleq$
be the universal $R$-matrix with respect to $\langle\,\,,\,\rangle$.
The we have
\begin{equation*}
\dudmcRo
={\frac 1 {xy^2}}\sum_{s,t=0}^{xy-1}
\sum_{g,h=0}^{y-1}(\zeta^{-2}\omega^{-1})^{-st}
\omega^{-gh}\duudK{1}^s\duudK{2}^g
\otimes \duudL{1}^t\duudL{2}^h.
\end{equation*}

\begin{lemma} \label{lemma:calRzero}
{\rm{(1)}}
Let $a$, $b\in\bZ$ and $x_i$, $y_i\in\bZ$ $(i\in I)$.
Define the $\bC$-algebra homomorphism
$f_1:\dudUogeq\to\bC$ 
{\rm{(}}resp. $f_2:\dudUoleq\to\bC${\rm{)}}
by $f_1(\dudK{i})=\zeta^{a_i}\omega^{x_i}$
{\rm{(}}resp. $f_2(\dudL{i})=\zeta^{b_i}\omega^{y_i}${\rm{)}}
$(i\in I)$,
where let $a_1=a$, $a_0=-a$
and $b_1=b$, $b_0=-b$. Then we have
\begin{equation*}
(f_1\otimes f_2)(\dudmcRo)
=\xi^{ab}\omega^{-x_1y_0+x_0y_1},
\end{equation*} where $\bC\otimes\bC$ is identified with $\bC$
by $\lambda\otimes \mu=\lambda\mu$ $(\lambda,\mu\in\bC)$.
\newline 
{\rm{(2)}} Assume that $\zeta$ and 
$\omega$ are the same as those of 
Theorem~{\rm{\ref{theorem:maintwo}}}.
Then there exists a unique
Hopf algebra isomorphism
$r:\dudUo\to\trmfkUo$
such that $r(\dudK{i})={\hat{g}}_{\zeta,\omega}(\trK_i)$, $r(\dudL{i})={\hat{g}}_{\zeta,\omega}(\trL_i)$
$(i\in I)$.
In particular,
$(r\otimes r)(\dudmcRo)=\trmfkH$.
\end{lemma}

\section{Vector representation}
\label{section:GenVrep}

From now on until \eqref{eqn:DefcalRz}, we assume $q(|q|-1)\ne 0$.

Let $\trkappa\in\bN$.
Let $\omega\in\bCt$ be 
such that $\omega^\trkappa=1$
and $\omega^y\ne 1$ $(y\in\fkJ_{1,\trkappa-1})$.
In this section, we assume $a=q^{-2}\omega$.
Let $u$, $v\in\bCt$.
Then we have the representation (or the $\bC$-algebra homomorphism)
$\rho:\trU\to{\mathrm{M}}_{2\trkappa}(\bC)$
defined by the following equations
\eqref{eqn:deftrrhoKOne}-\eqref{eqn:deftrrhoFZero}.
For $k\in\bN$ and $i,j\in\fkJ_{1,k}$,
let ${\dot{E}}^{(k)}_{ij}\in{\mathrm{M}}_k(\bC)$ be the matrix unit such that
$(i,j)$-position is $1$, i.e., ${\dot{E}}^{(k)}_{ij}=(\delta_{si}\delta_{tj})_{1\leq s, t\leq k}$. 
Let ${\acute{E}}_{ij}={\dot{E}}^{(2\trkappa)}_{ij}$ $(i,j\in\fkJ_{1,2\kappa})$.
Let ${\acute{E}}=\sum_{i=1}^{2\trkappa}{\acute{E}}_{ii}$
(the identity matrix).

Let ${\grave{E}}_{s,t}=E^{(\trkappa)}_{s,t}$
$(\in{\mathrm{M}}_\trkappa(\bC))$.
Let ${\grave{C}}=(\sum_{t=1}^{\trkappa-1}{\grave{E}}_{t,t+1})
+{\grave{E}}_{\trkappa,1}$.
That is,  ${\grave{C}}$ is 
the cyclic permutation matrix of size $\trkappa$.
Let ${\grave{P}}=\sum_{s,t=1}^\trkappa
\tromega^{(s-1)(t-1)}{\grave{E}}_{s,t}$.
Then ${\grave{P}}^{-1}={\frac 1 \trkappa}\sum_{s,t=1}^\trkappa
\tromega^{-(s-1)(t-1)}{\grave{E}}_{s,t}$.
Let ${\grave{D}}=\sum_{s=1}^\trkappa
\tromega^{s-1}{\grave{E}}_{s,s}$.
Then ${\grave{P}}^{-1}{\grave{C}}{\grave{P}}={\grave{D}}$.
We also note $\tromega{\grave{D}}{\grave{C}}={\grave{C}}{\grave{D}}$.
Let ${\grave{E}}=\sum_{s=1}^{\trkappa}{\grave{E}}_{ss}$.

We also identify ${\mathrm{M}}_{2\trkappa}(\bC)$
with ${\mathrm{M}}_2(\bC)\otimes
{\mathrm{M}}_{\trkappa}(\bC)$
in a natural way. 

\begin{equation} \label{eqn:deftrrhoKOne}
\rho(\trK_1)=u \sum_{t=1}^\trkappa\tromega^{t-1}(q^2 {\acute{E}}_{2t-1,2t-1}+ {\acute{E}}_{2t,2t})
=u\cdot
\left(\begin{array}{cc}
q^2 & 0 \\ 0 & 1
\end{array}\right)
\otimes {\grave{D}}.
\end{equation}
\begin{equation} \label{eqn:deftrrhoLOne}
\rho(\trL_1)=u \sum_{t=1}^\trkappa\tromega^{t-1}({\acute{E}}_{2t-1,2t-1}+ q^2 {\acute{E}}_{2t,2t})
=u\cdot
\left(\begin{array}{cc}
1 & 0 \\ 0 & q^2
\end{array}\right)
\otimes {\grave{D}}.
\end{equation}
\begin{equation} \label{eqn:deftrrhoEOne}
\rho(\trE_1)=-u(q^2-1)\sum_{t=1}^\trkappa\tromega^{t-1} {\acute{E}}_{2t-1,2t}
=-u(q^2-1)\cdot
\left(\begin{array}{cc}
0 & 1 \\ 0 & 0
\end{array}\right)
\otimes {\grave{D}}.
\end{equation}

\begin{equation} \label{eqn:deftrrhoFOne}
\rho(\trF_1)=\sum_{t=1}^\trkappa{\acute{E}}_{2t,2t-1}
=
\left(\begin{array}{cc}
0 & 0 \\ 1 & 0
\end{array}\right)
\otimes {\grave{E}}.
\end{equation}
\begin{equation} \label{eqn:deftrrhoKZero}
\rho(\trK_0)=v\left(\tromega {\acute{E}}_{1,1}+
(\sum_{t=1}^{\trkappa-1}\tromega^{-(t-1)}(q^2 {\acute{E}}_{2t,2t}+ {\acute{E}}_{2t+1,2t+1}))
+\tromega q^2 {\acute{E}}_{2\trkappa,2\trkappa}
\right)
=v\cdot
\left(\begin{array}{cc}
\tromega & 0 \\ 0 & q^2
\end{array}\right)
\otimes {\grave{D}}^{-1}.
\end{equation}
\begin{equation} \label{eqn:deftrrhoLZero}
\rho(\trL_0)=v\left(q^2\tromega {\acute{E}}_{1,1}+
(\sum_{t=1}^{\trkappa-1}\tromega^{-(t-1)}({\acute{E}}_{2t,2t}+ q^2  {\acute{E}}_{2t+1,2t+1}))
+\tromega {\acute{E}}_{2\trkappa,2\trkappa}
\right)
=v\cdot
\left(\begin{array}{cc}
\tromega q^2 & 0 \\ 0 & 1
\end{array}\right)
\otimes {\grave{D}}^{-1}.
\end{equation}
\begin{equation} \label{eqn:deftrrhoEZero}
\rho(\trE_0)=-v(q^2-1)
\left(
(\sum_{t=1}^{\trkappa-1}\tromega^{-(t-1)} {\acute{E}}_{2t,2t+1})
+\tromega {\acute{E}}_{2\trkappa,1}\right)
=-v(q^2-1)\cdot
\left(\begin{array}{cc}
0 & 0 \\ 1 & 0
\end{array}\right)
\otimes {\grave{D}}^{-1}{\grave{C}}.
\end{equation}
\begin{equation} \label{eqn:deftrrhoFZero}
\rho(\trF_0)=
(\sum_{t=1}^{\trkappa-1}{\acute{E}}_{2t+1,2t})
+{\acute{E}}_{1,2\trkappa}
=
\left(\begin{array}{cc}
0 & 1 \\ 0 & 0
\end{array}\right)
\otimes {\grave{C}}^{-1}.
\end{equation}

\noindent
From now one until end of 
Section~\ref{section:GenVrep},
for $y\in\bZ$, let ${\bar y}\in\fkJ_{1,2\trkappa}$  
be such that $y-{\bar y}\in\trkappa\bZ$.
We have the following equations
\eqref{eqn:rhEondelone}-\eqref{eqn:rhtilFondel}.

\begin{equation} \label{eqn:rhEondelone}
\begin{array}{lcl}
\rho(\trE_{n\delta+\al_1})
& = & (-1)^{n+1}\tromega^{2n} 
\cdot {\frac {u^{n+1}v^n(q^2-1)^{2n+1}} {q^n}}
\cdot \sum_{s=1}^{\trkappa}
\tromega^{s-1}
{\acute{E}}_{2s-1,{\overline{2(n+s)}}}  \\
& = & (-1)^{n+1}\tromega^{2n} 
\cdot {\frac {u^{n+1}v^n(q^2-1)^{2n+1}} {q^n}}
\cdot \left(\begin{array}{cc}
0 & 1 \\ 0 & 0
\end{array}\right)
\otimes {\grave{D}}{\grave{C}}^n.
\end{array}
\end{equation}

\begin{equation} \label{eqn:rhEondelzero}
\begin{array}{lcl}
\rho(\trE_{n\delta+\al_0})
& = & (-1)^{n+1}\tromega^n
\cdot {\frac {u^nv^{n+1}(q^2-1)^{2n+1}} {q^n}}
\cdot \sum_{s=1}^{\trkappa}
\tromega^{1-s}
{\acute{E}}_{2s,{\overline{2(n+s)+1}}} \\
& = & (-1)^{n+1}\tromega^n
\cdot {\frac {u^nv^{n+1}(q^2-1)^{2n+1}} {q^n}}
\cdot \left(\begin{array}{cc}
0 & 0 \\ 1 & 0
\end{array}\right)
\otimes {\grave{D}}^{-1}{\grave{C}}^{n+1}.
\end{array}
\end{equation}

\begin{equation} \label{eqn:rhEondel}
\begin{array}{lcl}
\rho(\trE_{n\delta})
& =& (-1)^n\tromega^{2n-1} 
\cdot {\frac {u^nv^n(q^2-1)^{2n}} {q^{n+1}}}
\cdot \sum_{s=1}^{\trkappa}({\acute{E}}_{2s-1,{\overline{2(n+s)-3}}}-q^2{\acute{E}}_{2s,{\overline{2(n+s-1)}}}) \\
& =& (-1)^n\tromega^{2n-1} 
\cdot {\frac {u^nv^n(q^2-1)^{2n}} {q^{n+1}}}
\cdot \left(\begin{array}{cc}
1 & 0 \\ 0 & -q^2
\end{array}\right)
\otimes {\grave{C}}^n.
\end{array}
\end{equation}

\begin{equation} \label{eqn:rhtilEondel}
\begin{array}{lcl}
\rho(\tE_{n\delta})
& =& (-1)^n\tromega^n 
\cdot {\frac {u^nv^n(q^2-1)^{2n-1}(q^{2n}-1)} {n q^{n-1}}}\\
& & \cdot \sum_{s=1}^\trkappa({\frac 1 {q^{2n}}}{\acute{E}}_{2s-1,{\overline{2(n+s)-3}}}-
{\acute{E}}_{2s,{\overline{2(n+s-1)}}}) \\
& =& (-1)^n\tromega^n 
\cdot {\frac {u^nv^n(q^2-1)^{2n-1}(q^{2n}-1)} {n q^{n-1}}}
\left(\begin{array}{cc}
{\frac 1 {q^{2n}}} & 0 \\
0 & -1
\end{array}\right)\otimes{\grave{C}}^n.
\end{array} 
\end{equation}

\begin{equation} \label{eqn:rhFondelone}
\begin{array}{lcl}
\rho(\trF_{n\delta+\al_1})
& = & (-1)^n\tromega^{-n} 
\cdot {\frac 1 {q^n}}
\cdot \sum_{s=1}^{\trkappa}
{\acute{E}}_{{\overline{2(n+s)},2s-1}} \\
& = & (-1)^n\tromega^{-n} 
\cdot {\frac 1 {q^n}}
\cdot \left(\begin{array}{cc}
0 & 0 \\ 1 & 0
\end{array}\right)
\otimes {\grave{C}}^{-n}.
\end{array}
\end{equation}

\begin{equation} \label{eqn:rhFondelzero}
\begin{array}{lcl}
\rho(\trF_{n\delta+\al_0})
& = & (-1)^n\tromega^{-n} 
\cdot {\frac 1 {q^n}}
\cdot \sum_{s=1}^{\trkappa}
{\acute{E}}_{{\overline{2(n+s)+1},2s}} \\
& = & (-1)^n\tromega^{-n} 
\cdot {\frac 1 {q^n}}
\cdot \left(\begin{array}{cc}
0 & 1 \\ 0 & 0
\end{array}\right)
\otimes {\grave{C}}^{-(n+1)}.
\end{array}
\end{equation}

\begin{equation} \label{eqn:rhFondel}
\begin{array}{lcl}
\rho(\trF_{n\delta})
& =& (-1)^{n-1}\tromega^{-n+1}
\cdot {\frac {1} {q^{n+1}}}
\cdot \sum_{s=1}^{\trkappa}(q^2{\acute{E}}_{{\overline{2(n+s)-1}},2s-1}-\tromega^{-n} {\acute{E}}_{{\overline{2(n+s)}},2s}) \\
& =& (-1)^{n-1}\tromega^{-n+1}
\cdot {\frac {1} {q^{n+1}}}
\cdot \left(\begin{array}{cc}
q^2 & 0 \\ 0 & -\tromega^{-n}
\end{array}\right)
\otimes {\grave{C}}^{-n}.
\end{array}
\end{equation}

\begin{equation} \label{eqn:rhtilFondel}
\begin{array}{lcl}
\rho(\tF_{n\delta})
& =& (-1)^{n-1}
\cdot {\frac {q^{2n}-1} {n q^{n-1}(q^2-1)}}\\
& & \cdot \sum_{s=1}^\trkappa({\acute{E}}_{{\overline{2(n+s)-1}},2s-1}-
{\frac {\tromega^{-n} } {q^{2n}}}
{\acute{E}}_{{\overline{2(n+s)}},2s}) \\
& =& (-1)^{n-1}
\cdot {\frac {q^{2n}-1} {n q^{n-1}(q^2-1)}}
\left(\begin{array}{cc}
1 & 0 \\
0 & -{\frac {\tromega^{-n} } {q^{2n}}}
\end{array}\right)\otimes{\grave{C}}^{-n}.
\end{array}
\end{equation}

For $y\in\bCt$,
define the 
$\bC$-algebra homomorphism
$\rho_y:\trU\to{\mathrm{M}}_{2\trkappa}(\bC)$
by $\rho_y(X)=\rho(X)$
for $X\in\{\trK_i^{\pm1}, \trL_i^{\pm1}\,|\,i\in I\}
\cup\{\trE_1,\trF_1\}$
and $\rho_y(\trE_0)=y\cdot\rho(\trE_0)$,
$\rho_y(\trF_0)=y^{-1}\cdot\rho(\trF_0)$

Let $\lambda$, $\mu\in\bCt$ Let $z={\frac \lambda \mu}$
and assume $|z|<<1$.
For $i\in I$, let
\begin{equation}
\begin{array}{l}
\trRuvdotdelali  \\
\quad =  {\acute{E}}\otimes{\acute{E}}+\sum_{n=0}^\infty
\sum_{k=1}^\infty
\left(
{\frac 1 {{(k)}_{q^2}!
{\langle\trE_{n\delta+\al_i},\trF_{n\delta+\al_i}\rangle^k}}}
(\rho_\lambda(\trE_{n\delta+\al_i})^k\otimes
\rho_\mu(\trF_{n\delta+\al_i})^k)
\right) \\
\quad =  {\acute{E}}
\otimes{\acute{E}} +\sum_{n=0}^\infty
\left(
{\frac 1 {
{\langle\trE_{n\delta+\al_i},\trF_{n\delta+\al_i}\rangle}}}
(\rho_\lambda(\trE_{n\delta+\al_i})\otimes
\rho_\mu(\trF_{n\delta+\al_i}))
\right).
\end{array}
\end{equation} Then we have
\begin{equation}
\begin{array}{l}
\trRuvdotdelalzero \\
\quad ={\acute{E}}
\otimes{\acute{E}} -
{\frac {z(q^2-1)v} {1-(zuvq^2)^\trkappa}}
\left(\sum\limits_{n=0}^{\trkappa-1}(zuvq^2)^n
\left(\sum\limits_{s,t=1}^\trkappa
(\tromega^{1-s}
{\acute{E}}_{2s,{\overline{2(n+s)+1}}}
\otimes {\acute{E}}_{{\overline{2(n+t)+1},2t}})
\right)
\right) \\
\quad ={\acute{E}}
\otimes{\acute{E}} -
{\frac {z(q^2-1)v} {1-(zuvq^2)^\trkappa}}
\left(\sum\limits_{n=0}^{\trkappa-1}(zuvq^2)^n
\left(\left(\begin{array}{cc}
0 & 0 \\
1 & 0
\end{array}\right)\otimes{\grave{D}}^{-1}{\grave{C}}^{n+1}
\right)\otimes
\left(\left(\begin{array}{cc}
0 & 1 \\
0 & 0
\end{array}\right)\otimes{\grave{C}}^{-(n+1)}\right)
\right).
\end{array}
\end{equation}
and
\begin{equation}
\begin{array}{l}
\trRuvdotdelalone \\
\quad ={\acute{E}}
\otimes{\acute{E}} -
{\frac {(q^2-1)u} {1-(zuvq^2)^\trkappa}}
\left(\sum\limits_{n=0}^{\trkappa-1}(zuv\tromega q^2)^n
\left(\sum\limits_{s,t=1}^\trkappa
(\tromega^{s-1}
{\acute{E}}_{2s-1,{\overline{2(n+s)}}}
\otimes {\acute{E}}_{{\overline{2(n+t)},2t-1}})
\right)
\right) \\
\quad ={\acute{E}}
\otimes{\acute{E}} -
{\frac {(q^2-1)u} {1-(zuvq^2)^\trkappa}}
\left(\sum\limits_{n=0}^{\trkappa-1}(zuv\tromega q^2)^n
\left(\left(\begin{array}{cc}
0 & 1 \\
0 & 0
\end{array}\right)\otimes{\grave{D}}{\grave{C}}^n
\right)\otimes
\left(\left(\begin{array}{cc}
0 & 0 \\
1 & 0
\end{array}\right)\otimes{\grave{C}}^{-n}\right)
\right).
\end{array}
\end{equation}

Let
\begin{equation}
\begin{array}{l}
\trRuvdotdel  \\
\quad =  \prod\limits_{n=0}^\infty\left({\acute{E}}\otimes{\acute{E}}+
\sum\limits_{k=1}^\infty
\left(
{\frac 1 {k!
{\langle\tE_{n\delta},\tF_{n\delta}\rangle^k}}}
(\rho_\lambda(\tE_{n\delta})^k\otimes
\rho_\mu(\tF_{n\delta})^k)
\right)\right) \\
\quad =
\exp\left(
-\sum\limits_{n=1}^\infty{\frac {z^n\tromega^nu^nv^nq^{4n}} n}
\cdot{\frac {q^n-q^{-n}} {q^n+q^{-n}}}
\left(\left(\begin{array}{cc}
{\frac 1 {q^{2n}}} & 0 \\
0 & -1
\end{array}\right)\otimes{\grave{C}}^n
\right)\otimes
\left(\left(\begin{array}{cc}
1 & 0 \\
0 & -{\frac {\tromega^{-n} } {q^{2n}}}
\end{array}\right)\otimes{\grave{C}}^{-n}\right)
\right).
\end{array} 
\end{equation}

Let 
\begin{equation*}
{\mathcal{A}}(z)=\exp\left(-\sum_{k=1}^\infty
{\frac {z^k(q^k-q^{-k})} {k(q^k+q^{-k})}}\right), 
\end{equation*}
where ${\mathcal{A}}(z)$ 
can be defined since $|z|<1$
and $|q|\ne 1$.
Since $|z|<<1$, we have
\begin{equation*}
\begin{array}{l}
{\mathcal{A}}(z){\mathcal{A}}(q^2z)=
\exp\left(-\sum_{k=1}^\infty
{\frac {z^k(q^{3k}-q^{-k})} {k(q^k+q^{-k})}}\right) \\
\quad =\exp \left(-\sum_{k=1}^\infty 
{\frac {z^k(q^{2k}-1)} {k}}\right) \\
\quad =
{\frac {1-q^2z} {1-z}},
\end{array}
\end{equation*}
and
\begin{equation*}
{\frac {{\mathcal{A}}(q^4z)} {{\mathcal{A}}(z)}}
={\frac {{\mathcal{A}}(q^2z){\mathcal{A}}(q^4z)} {{\mathcal{A}}(z){\mathcal{A}}(q^2z)}}
={\frac {(1-z)(1-q^4z)} {(1-q^2z)^2}}.
\end{equation*}

Let ${\acute{P}}=\left(\begin{array}{cc}
1 & 0 \\
0 & 1
\end{array}\right)\otimes{\grave{P}}$.
Then we have
\begin{equation}
\begin{array}{l}
\trRuvdotdel  \\
\quad = ({\acute{P}}\otimes{\acute{P}}) \\
\quad\quad \cdot
\exp\left(
-\sum\limits_{n=1}^\infty{\frac {z^n\tromega^nu^nv^nq^{4n}} n}
\cdot{\frac {q^n-q^{-n}} {q^n+q^{-n}}}
\left(\left(\begin{array}{cc}
{\frac 1 {q^{2n}}} & 0 \\
0 & -1
\end{array}\right)\otimes{\grave{D}}^n
\right)\otimes
\left(\left(\begin{array}{cc}
1 & 0 \\
0 & -{\frac {\tromega^{-n} } {q^{2n}}}
\end{array}\right)\otimes{\grave{D}}^{-n}\right)
\right)
\\
\quad\quad\quad \cdot ({\acute{P}}^{-1}\otimes{\acute{P}}^{-1}) \\
\quad = ({\acute{P}}\otimes{\acute{P}}) \\
\quad\quad \cdot\Bigl(
\sum\limits_{s,t=1}^\trkappa
\Bigl\{
{\mathcal{A}}(z\tromega uv q^2\tromega^{s-t})
\left(\left(\left(\begin{array}{cc}
1 & 0 \\
0 & 0
\end{array}\right)\otimes{\grave{E}}_{ss}
\right)\otimes
\left(\left(\begin{array}{cc}
1 & 0 \\
0 & 0
\end{array}\right)\otimes{\grave{E}}_{tt}\right)\right) \\
\quad\quad 
+{\frac 1 {{\mathcal{A}}(zuv \tromega^{s-t})}}
\left(\left(\left(\begin{array}{cc}
1 & 0 \\
0 & 0
\end{array}\right)\otimes{\grave{E}}_{ss}
\right)\otimes
\left(\left(\begin{array}{cc}
0 & 0 \\
0 & 1
\end{array}\right)\otimes{\grave{E}}_{tt}\right)\right) \\
\quad\quad 
+{\frac 1 {{\mathcal{A}}(z\tromega uv q^4 \tromega^{s-t})}}
\left(\left(\left(\begin{array}{cc}
0 & 0 \\
0 & 1
\end{array}\right)\otimes{\grave{E}}_{ss}
\right)\otimes
\left(\left(\begin{array}{cc}
1 & 0 \\
0 & 0
\end{array}\right)\otimes{\grave{E}}_{tt}\right)\right) \\
\quad\quad +{\mathcal{A}}(z uv q^2 \tromega^{s-t})
\left(\left(\left(\begin{array}{cc}
0 & 0 \\
0 & 1
\end{array}\right)\otimes{\grave{E}}_{ss}
\right)\otimes
\left(\left(\begin{array}{cc}
0 & 0 \\
0 & 1
\end{array}\right)\otimes{\grave{E}}_{tt}\right)\right)
\Bigr\}\Big)
\\
\quad\quad\quad \cdot ({\acute{P}}^{-1}\otimes{\acute{P}}^{-1})
\end{array} 
\end{equation}

\begin{equation*}
\begin{array}{l}
\quad = ({\acute{P}}\otimes{\acute{P}}) \\
\quad\quad \cdot\Bigl(
\sum\limits_{s,t=1}^\trkappa
\Bigl\{
{\mathcal{A}}(z\tromega uv q^2\tromega^{s-t})
\left(\left(\left(\begin{array}{cc}
1 & 0 \\
0 & 0
\end{array}\right)\otimes{\grave{E}}_{ss}
\right)\otimes
\left(\left(\begin{array}{cc}
1 & 0 \\
0 & 0
\end{array}\right)\otimes{\grave{E}}_{tt}\right)\right) \\
\quad\quad 
+{\mathcal{A}}(zuv q^2 \tromega^{s-t})
{\frac {1-zuv \tromega^{s-t}} {1-zuv q^2 \tromega^{s-t}}}
\left(\left(\left(\begin{array}{cc}
1 & 0 \\
0 & 0
\end{array}\right)\otimes{\grave{E}}_{ss}
\right)\otimes
\left(\left(\begin{array}{cc}
0 & 0 \\
0 & 1
\end{array}\right)\otimes{\grave{E}}_{tt}\right)\right) \\
\quad\quad 
+{\mathcal{A}}(z\tromega uv q^2 \tromega^{s-t})
{\frac {1-z\tromega uv q^2 \tromega^{s-t}} {1-z\tromega uv q^4 \tromega^{s-t}}}
\left(\left(\left(\begin{array}{cc}
0 & 0 \\
0 & 1
\end{array}\right)\otimes{\grave{E}}_{ss}
\right)\otimes
\left(\left(\begin{array}{cc}
1 & 0 \\
0 & 0
\end{array}\right)\otimes{\grave{E}}_{tt}\right)\right) \\
\quad\quad +{\mathcal{A}}(z uv q^2\tromega^{s-t})
\left(\left(\left(\begin{array}{cc}
0 & 0 \\
0 & 1
\end{array}\right)\otimes{\grave{E}}_{ss}
\right)\otimes
\left(\left(\begin{array}{cc}
0 & 0 \\
0 & 1
\end{array}\right)\otimes{\grave{E}}_{tt}\right)\right)
\Bigr\}\Bigr)
\\
\quad\quad\quad \cdot ({\acute{P}}^{-1}\otimes{\acute{P}}^{-1})
\end{array} 
\end{equation*}

\begin{equation*}
\begin{array}{l}
\quad = {\frac 1 {\trkappa^2}}\cdot
\sum\limits_{x_1,y_1,x_2,y_2=1}^\trkappa 
\Bigl\{\sum\limits_{s,t=1}^\trkappa 
\tromega^{(x_1-y_1)(s-1)+(x_2-y_2)(t-1)} \\
\quad\quad \cdot
\Bigl\{
{\mathcal{A}}(z\tromega uv q^2\tromega^{s-t})
\left(\left(\left(\begin{array}{cc}
1 & 0 \\
0 & 0
\end{array}\right)\otimes{\grave{E}}_{x_1,y_1}
\right)\otimes
\left(\left(\begin{array}{cc}
1 & 0 \\
0 & 0
\end{array}\right)\otimes{\grave{E}}_{x_2,y_2}\right)\right) \\
\quad\quad 
+{\mathcal{A}}(zuv q^2 \tromega^{s-t})
{\frac {1-zuv \tromega^{s-t}} {1-zuv q^2 \tromega^{s-t}}}
\left(\left(\left(\begin{array}{cc}
1 & 0 \\
0 & 0
\end{array}\right)\otimes{\grave{E}}_{x_1,y_1}
\right)\otimes
\left(\left(\begin{array}{cc}
0 & 0 \\
0 & 1
\end{array}\right)\otimes{\grave{E}}_{x_2,y_2}\right)\right) \\
\quad\quad 
+{\mathcal{A}}(z\tromega uv q^2 \tromega^{s-t})
{\frac {1-z\tromega uv q^2 \tromega^{s-t}} {1-z\tromega uv q^4 \tromega^{s-t}}}
\left(\left(\left(\begin{array}{cc}
0 & 0 \\
0 & 1
\end{array}\right)\otimes{\grave{E}}_{x_1,y_1}
\right)\otimes
\left(\left(\begin{array}{cc}
1 & 0 \\
0 & 0
\end{array}\right)\otimes{\grave{E}}_{x_2,y_2}\right)\right) \\
\quad\quad +{\mathcal{A}}(z uv q^2\tromega^{s-t})
\left(\left(\left(\begin{array}{cc}
0 & 0 \\
0 & 1
\end{array}\right)\otimes{\grave{E}}_{x_1,y_1}
\right)\otimes
\left(\left(\begin{array}{cc}
0 & 0 \\
0 & 1
\end{array}\right)\otimes{\grave{E}}_{x_2,y_2}\right)\right)
\Bigr\}\Bigr\}.
\end{array} 
\end{equation*}

Assume $u=v=q^{-1}$.
Let 
\begin{equation}
\begin{array}{lcl}
{\mathcal{R}}_{{\mathfrak{h}}} & = & 
\sum\limits_{s,t=1}^\trkappa
\Bigl\{
q^{-{\frac 1 2}}\tromega^{-(s-1)(2-t)+(2-s)(t-1)}
\left(\left(\left(\begin{array}{cc}
1 & 0 \\
0 & 0
\end{array}\right)\otimes{\grave{E}}_{ss}
\right)\otimes
\left(\left(\begin{array}{cc}
1 & 0 \\
0 & 0
\end{array}\right)\otimes{\grave{E}}_{tt}\right)\right) \\
&  & 
+q^{{\frac 1 2}}\tromega^{-(s-1)(1-t)+(2-s)(t-1)}
\left(\left(\left(\begin{array}{cc}
1 & 0 \\
0 & 0
\end{array}\right)\otimes{\grave{E}}_{ss}
\right)\otimes
\left(\left(\begin{array}{cc}
0 & 0 \\
0 & 1
\end{array}\right)\otimes{\grave{E}}_{tt}\right)\right) \\
&  &  
+q^{{\frac 1 2}}\tromega^{-(s-1)(2-t)+(1-s)(t-1)}
\left(\left(\left(\begin{array}{cc}
0 & 0 \\
0 & 1
\end{array}\right)\otimes{\grave{E}}_{ss}
\right)\otimes
\left(\left(\begin{array}{cc}
1 & 0 \\
0 & 0
\end{array}\right)\otimes{\grave{E}}_{tt}\right)\right) \\
&  &  
+q^{-{\frac 1 2}}\tromega^{-(s-1)(1-t)+(1-s)(t-1)}
\left(\left(\left(\begin{array}{cc}
0 & 0 \\
0 & 1
\end{array}\right)\otimes{\grave{E}}_{ss}
\right)\otimes
\left(\left(\begin{array}{cc}
0 & 0 \\
0 & 1
\end{array}\right)\otimes{\grave{E}}_{tt}\right)\right)
\Bigr\}
\\
 & = & 
\sum\limits_{s,t=1}^\trkappa
\Bigl\{
q^{-{\frac 1 2}}\tromega^{-s+t}
\left(\left(\left(\begin{array}{cc}
1 & 0 \\
0 & 0
\end{array}\right)\otimes{\grave{E}}_{ss}
\right)\otimes
\left(\left(\begin{array}{cc}
1 & 0 \\
0 & 0
\end{array}\right)\otimes{\grave{E}}_{tt}\right)\right) \\
&  & 
+q^{{\frac 1 2}}\tromega^{t-1}
\left(\left(\left(\begin{array}{cc}
1 & 0 \\
0 & 0
\end{array}\right)\otimes{\grave{E}}_{ss}
\right)\otimes
\left(\left(\begin{array}{cc}
0 & 0 \\
0 & 1
\end{array}\right)\otimes{\grave{E}}_{tt}\right)\right) \\
&  &  
+q^{{\frac 1 2}}\tromega^{-s+1}
\left(\left(\left(\begin{array}{cc}
0 & 0 \\
0 & 1
\end{array}\right)\otimes{\grave{E}}_{ss}
\right)\otimes
\left(\left(\begin{array}{cc}
1 & 0 \\
0 & 0
\end{array}\right)\otimes{\grave{E}}_{tt}\right)\right) \\
&  &  
+q^{-{\frac 1 2}}
\left(\left(\left(\begin{array}{cc}
0 & 0 \\
0 & 1
\end{array}\right)\otimes{\grave{E}}_{ss}
\right)\otimes
\left(\left(\begin{array}{cc}
0 & 0 \\
0 & 1
\end{array}\right)\otimes{\grave{E}}_{tt}\right)\right)
\Bigr\}.
\end{array} 
\end{equation}

Let
\begin{equation}
\label{eqn:DefcalRz}
{\mathcal{R}}(z)
=
{\mathcal{R}}_{\bullet\delta+\al_1}(q^{-1},q^{-1})
{\mathcal{R}}_{\bullet\delta}(q^{-1},q^{-1})
{\mathcal{R}}_{\bullet\delta+\al_0}(q^{-1},q^{-1})
{\mathcal{R}}_{{\mathfrak{h}}}.
\end{equation}

Using a general argument from the theory of power series in a single complex variable and that of absolutely convergent double series, we derive Theorem~\ref{theorem:MainVrepRMat} below from
Theorem~\ref{theorem:maintwo}
and Lemma~\ref{lemma:calRzero} and from the fact that $\rho$
can also be regarded as the representation of $\trmfkU_{\zeta,\omega}$ 
if $q$ is considered as that of \eqref{eqn:DefNewq} below.

\begin{theorem}\label{theorem:MainVrepRMat}
Let $q\in\bCt$
be such that
\begin{equation}
\label{eqn:DefNewq}
\mbox{$|q|\ne 1$,
or $q^{2m+1}=1$, $q^y\ne 1$
$(y\in\fkJ_{1,2m})$
for some $m\in\bN$.}
\end{equation}
Let $z,w\in\bCt$
be such that $|u|<1$
for all
\begin{equation*}
u\in\{z,w,zw,zq^2,wq^2,zwq^{-2},zq^{-2},wq^{-2},zwq^{-2}\}.
\end{equation*}  
Then we have the Yang-Baxter equation{\rm{:}}
\begin{equation}
{\mathcal{R}}(z)_{12}{\mathcal{R}}(zw)_{13}{\mathcal{R}}(w)_{23}
={\mathcal{R}}(w)_{23}{\mathcal{R}}(zw)_{13}{\mathcal{R}}(z)_{12},
\end{equation} where ${\mathcal{R}}(z)$ is defined 
by placing $q$, $z$
in the same letters of that of 
\eqref{eqn:DefcalRz}.
\end{theorem}

\subsection{$a=-q^{-2}$ case} \label{section:minusonecase}

Let ${\mathcal{R}}(z)$ be of 
Theorem~\ref{theorem:MainVrepRMat}.

Assume $\trkappa=1$, i.e., $a=q^{-2}$
and $\omega=1$. Then $\trUqQ({\hat {sl_2}})$ is the usual $U_q({\hat {sl_2}})$, and 
${\mathcal{R}}(z)$
is
the $R$-matrix studied by \cite{LSS93}:
\begin{equation*}
\mbox{{\rm{(${\mathcal{R}}1$)}
}}
\quad\quad
\begin{array}{l}
{\mathcal{R}}(z)=q^{-{\frac 1 2}}{\mathcal{A}}(z)\cdot 
\Bigl({\hat{E}}_{1,1}
+{\frac {q(z-1)} {q^2z-1}}{\hat{E}}_{2,2}
+{\frac {q^2-1} {q^2z-1}}{\hat{E}}_{2,3} \\
\hspace{3cm} 
+{\frac {z(q^2-1)} {q^2z-1}}{\hat{E}}_{3,2}
+{\frac {q(z-1)} {q^2z-1}}{\hat{E}}_{3,3}
+{\hat{E}}_{4,4}
\Bigr),
\end{array}
\end{equation*} where ${\hat{E}}_{s,t}=E^{(4)}_{s,t}$
$(s,t\in\fkJ_{1,4})$.
Note that
$q^{{\frac 1 2}}{\mathcal{A}}(z)^{-1}{\mathcal{R}}(z)$
for ${\mathcal{R}}(z)$ of {\rm{(${\mathcal{R}}1$)}
}
is the standard $R$-matrix associated with the six-vertex model of physics.

From now on until Theorem~\ref
{theorem:YBERtwoz}, we assume $\trkappa=2$. So $a=-q^{-2}$ and $\omega=-1$.
Let ${\mathcal{B}}(z)={\frac  1 2}({\mathcal{A}}(z)+{\mathcal{A}}(-z))$
and ${\mathcal{C}}(z)={\frac  1 2}({\mathcal{A}}(z)-{\mathcal{A}}(-z))$.
Then ${\mathcal{R}}(z)$
is the following,
where ${\hat{E}}_{s,t}=E^{(16)}_{s,t}$
$(s,t\in\fkJ_{1,16})$.
\newline\newline
{\rm{(${\mathcal{R}}2$)}}
\newline
${\mathcal{R}}(z)=$ \newline
$q^{-{\frac 1 2}}\cdot\Bigl(\mathcal{B}(z) {\hat{E}}_{1,1}$
$-\mathcal{C}(z) {\hat{E}}_{1,6}$
$-\mathcal{B}(z) {\hat{E}}_{2,2}$
$+\mathcal{C}(z) {\hat{E}}_{2,5}$
$+\frac{q \left(q^2 z^2 \mathcal{B}(z)-q^2 z \mathcal{C}(z)+z
   \mathcal{C}(z)-\mathcal{B}(z)\right)}{\left(q^2 z-1\right) \left(q^2 z+1\right)} {\hat{E}}_{3,3}$
   \newline
$-\frac{q \left(q^2 z^2 \mathcal{C}(z)+q^2 (-z) \mathcal{B}(z)-\mathcal{C}(z)+z
   \mathcal{B}(z)\right)}{\left(q^2 z-1\right) \left(q^2 z+1\right)} {\hat{E}}_{3,8}$
$+\frac{(q-1) (q+1) \left(q^2 z \mathcal{C}(z)+\mathcal{B}(z)\right)}{\left(q^2 z-1\right)
   \left(q^2 z+1\right)} {\hat{E}}_{3,9}$
$+\frac{(q-1) (q+1) \left(q^2 z \mathcal{B}(z)+\mathcal{C}(z)\right)}{\left(q^2 z-1\right)
   \left(q^2 z+1\right)} {\hat{E}}_{3,14}$ \newline
$-\frac{q \left(q^2 z^2 \mathcal{B}(z)-q^2 z \mathcal{C}(z)+z
   \mathcal{C}(z)-\mathcal{B}(z)\right)}{\left(q^2 z-1\right) \left(q^2 z+1\right)} {\hat{E}}_{4,4}$ 
$+\frac{q \left(q^2 z^2 \mathcal{C}(z)+q^2 (-z) \mathcal{B}(z)-\mathcal{C}(z)+z
   \mathcal{B}(z)\right)}{\left(q^2 z-1\right) \left(q^2 z+1\right)} {\hat{E}}_{4,7}$
$+\frac{(q-1) (q+1) \left(q^2 z \mathcal{C}(z)+\mathcal{B}(z)\right)}{\left(q^2 z-1\right)
   \left(q^2 z+1\right)} {\hat{E}}_{4,10}$
$+\frac{(q-1) (q+1) \left(q^2 z \mathcal{B}(z)+\mathcal{C}(z)\right)}{\left(q^2 z-1\right)
   \left(q^2 z+1\right)} {\hat{E}}_{4,13}$
$+\mathcal{C}(z) {\hat{E}}_{5,2}$
$-\mathcal{B}(z) {\hat{E}}_{5,5}$
$-\mathcal{C}(z) {\hat{E}}_{6,1}$
$+\mathcal{B}(z) {\hat{E}}_{6,6}$ \newline
$-\frac{q \left(q^2 z^2 \mathcal{C}(z)+q^2 (-z) \mathcal{B}(z)-\mathcal{C}(z)+z
   \mathcal{B}(z)\right)}{\left(q^2 z-1\right) \left(q^2 z+1\right)} {\hat{E}}_{7,4}$
$+\frac{q \left(q^2 z^2 \mathcal{B}(z)-q^2 z \mathcal{C}(z)+z
   \mathcal{C}(z)-\mathcal{B}(z)\right)}{\left(q^2 z-1\right) \left(q^2 z+1\right)} {\hat{E}}_{7,7}$
$+\frac{(q-1) (q+1) \left(q^2 z \mathcal{B}(z)+\mathcal{C}(z)\right)}{\left(q^2 z-1\right)
   \left(q^2 z+1\right)} {\hat{E}}_{7,10}$
$+\frac{(q-1) (q+1) \left(q^2 z \mathcal{C}(z)+\mathcal{B}(z)\right)}{\left(q^2 z-1\right)
   \left(q^2 z+1\right)} {\hat{E}}_{7,13}$
$+\frac{q \left(q^2 z^2 \mathcal{C}(z)+q^2 (-z) \mathcal{B}(z)-\mathcal{C}(z)+z
   \mathcal{B}(z)\right)}{\left(q^2 z-1\right) \left(q^2 z+1\right)} {\hat{E}}_{8,3}$ 
$-\frac{q \left(q^2 z^2 \mathcal{B}(z)-q^2 z \mathcal{C}(z)+z
   \mathcal{C}(z)-\mathcal{B}(z)\right)}{\left(q^2 z-1\right) \left(q^2 z+1\right)} {\hat{E}}_{8,8}$
$+\frac{(q-1) (q+1) \left(q^2 z \mathcal{B}(z)+\mathcal{C}(z)\right)}{\left(q^2 z-1\right)
   \left(q^2 z+1\right)} {\hat{E}}_{8,9}$
$+\frac{(q-1) (q+1) \left(q^2 z \mathcal{C}(z)+\mathcal{B}(z)\right)}{\left(q^2 z-1\right)
   \left(q^2 z+1\right)} {\hat{E}}_{8,14}$ 
$+\frac{(q-1) (q+1) z \left(q^2 z \mathcal{B}(z)+\mathcal{C}(z)\right)}{\left(q^2 z-1\right)
   \left(q^2 z+1\right)} {\hat{E}}_{9,3}$ \newline
$-\frac{(q-1) (q+1) z \left(q^2 z \mathcal{C}(z)+\mathcal{B}(z)\right)}{\left(q^2 z-1\right)
\left(q^2 z+1\right)} {\hat{E}}_{9,8}$
$+\frac{q \left(q^2 z^2 \mathcal{B}(z)-q^2 z \mathcal{C}(z)+z
   \mathcal{C}(z)-\mathcal{B}(z)\right)}{\left(q^2 z-1\right) \left(q^2 z+1\right)} {\hat{E}}_{9,9}$
$+\frac{q \left(q^2 z^2 \mathcal{C}(z)+q^2 (-z) \mathcal{B}(z)-\mathcal{C}(z)+z
   \mathcal{B}(z)\right)}{\left(q^2 z-1\right) \left(q^2 z+1\right)} {\hat{E}}_{9,14}$
$-\frac{(q-1) (q+1) z \left(q^2 z \mathcal{B}(z)+\mathcal{C}(z)\right)}{\left(q^2 z-1\right)
   \left(q^2 z+1\right)} {\hat{E}}_{10,4}$
$+\frac{(q-1) (q+1) z \left(q^2 z \mathcal{C}(z)+\mathcal{B}(z)\right)}{\left(q^2 z-1\right)
   \left(q^2 z+1\right)} {\hat{E}}_{10,7}$
$+\frac{q \left(q^2 z^2 \mathcal{B}(z)-q^2 z \mathcal{C}(z)+z
   \mathcal{C}(z)-\mathcal{B}(z)\right)}{\left(q^2 z-1\right) \left(q^2 z+1\right)} {\hat{E}}_{10,10}$ \newline
$+\frac{q \left(q^2 z^2 \mathcal{C}(z)+q^2 (-z) \mathcal{B}(z)-\mathcal{C}(z)+z
   \mathcal{B}(z)\right)}{\left(q^2 z-1\right) \left(q^2 z+1\right)} {\hat{E}}_{10,13}$
$+\mathcal{B}(z) {\hat{E}}_{11,11}$
$+\mathcal{C}(z) {\hat{E}}_{11,16}$
$+\mathcal{B}(z) {\hat{E}}_{12,12}$
$+\mathcal{C}(z) {\hat{E}}_{12,15}$ \newline
$+\frac{(q-1) (q+1) z \left(q^2 z \mathcal{C}(z)+\mathcal{B}(z)\right)}{\left(q^2 z-1\right)
   \left(q^2 z+1\right)} {\hat{E}}_{13,4}$
$-\frac{(q-1) (q+1) z \left(q^2 z \mathcal{B}(z)+\mathcal{C}(z)\right)}{\left(q^2 z-1\right)
   \left(q^2 z+1\right)} {\hat{E}}_{13,7}$
$-\frac{q \left(q^2 z^2 \mathcal{C}(z)+q^2 (-z) \mathcal{B}(z)-\mathcal{C}(z)+z
   \mathcal{B}(z)\right)}{\left(q^2 z-1\right) \left(q^2 z+1\right)} {\hat{E}}_{13,10}$
$-\frac{q \left(q^2 z^2 \mathcal{B}(z)-q^2 z \mathcal{C}(z)+z
   \mathcal{C}(z)-\mathcal{B}(z)\right)}{\left(q^2 z-1\right) \left(q^2 z+1\right)} {\hat{E}}_{13,13}$
$-\frac{(q-1) (q+1) z \left(q^2 z \mathcal{C}(z)+\mathcal{B}(z)\right)}{\left(q^2 z-1\right)
   \left(q^2 z+1\right)} {\hat{E}}_{14,3}$
$+\frac{(q-1) (q+1) z \left(q^2 z \mathcal{B}(z)+\mathcal{C}(z)\right)}{\left(q^2 z-1\right)
   \left(q^2 z+1\right)} {\hat{E}}_{14,8}$ \newline
$-\frac{q \left(q^2 z^2 \mathcal{C}(z)+q^2 (-z) \mathcal{B}(z)-\mathcal{C}(z)+z
   \mathcal{B}(z)\right)}{\left(q^2 z-1\right) \left(q^2 z+1\right)} {\hat{E}}_{14,9}$
$-\frac{q \left(q^2 z^2 \mathcal{B}(z)-q^2 z \mathcal{C}(z)+z
   \mathcal{C}(z)-\mathcal{B}(z)\right)}{\left(q^2 z-1\right) \left(q^2 z+1\right)} {\hat{E}}_{14,14}$
$+\mathcal{C}(z) {\hat{E}}_{15,12}$
$+\mathcal{B}(z) {\hat{E}}_{15,15}$
$+\mathcal{C}(z) {\hat{E}}_{16,11}$
$+\mathcal{B}(z) {\hat{E}}_{16,16}\Bigr)$.

\vspace{1cm}

For $b, c\in\bC$,
let ${\mathcal{R}}(z;b,c)$ be the one
obtained from $q^{{\frac 1 2}}\left(q^2 z-1\right) \left(q^2 z+1\right)\cdot{\mathcal{R}}(z)$ 
by replacing $\mathcal{B}(z)$
and $\mathcal{C}(z)$ by $b$ and $c$, respectively.
We consider $q$ and $z$ in ${\mathcal{R}}(z;b,c)$ as arbitrary elements of $\bC$.
By a calculation by Mathematica~14.3~\cite{Wolf14-3}, we conclude:
\begin{theorem} 
\label{theorem:YBERtwoz}
We have{\rm{:}}
\begin{equation*}
\begin{array}{l}
{\mathcal{R}}(z;b_1,c_1)_{12}{\mathcal{R}}(zw;b_2,c_2)_{13} 
{\mathcal{R}}(w;b_3,c_3)_{23} \\
\quad ={\mathcal{R}}(w;b_3,c_3)_{23}{\mathcal{R}}(zw;b_2,c_2)_{13}
{\mathcal{R}}(z;b_1,c_1)_{12}
\end{array}
\end{equation*}
for all $b_t,c_t\in\bC$ $(t\in\fkJ_{1,3})$
and all $q$, $z$, $w\in\bC$.

\end{theorem}

\section{Drinfeld-Reshetikhin twisting} \label{section:DiReshTwist}
Let $M\in\bN$ and let ${\tilde{\omega}}\in\bCt$
be such that ${\tilde{\omega}}^M=1$
and ${\tilde{\omega}}^m\ne 1$ $(m\in\fkJ_{1,M-1})$, i.e.,
${\tilde{\omega}}$ is a primitive $m$-th root of unity.
Let ${\tilde{\omega}}_{00}={\tilde{\omega}}_{11}=1$
and ${\tilde{\omega}}_{01}={\tilde{\omega}}$, ${\tilde{\omega}}_{10}={\tilde{\omega}^{-1}}$. 
Let ${\mathcal{M}}$
be the associative unital $\bC$-algebra
defined by the generators 
$\sigma_i$, $\tau_i$ 
$(i\in I)$
and
the relations $\sigma_i^M=\tau_i^M=1$,
$\sigma_i\sigma_j=\sigma_j\sigma_i$,  
$\tau_i\tau_j=\tau_j\tau_i$, 
$\sigma_i\tau_j=\tau_j\sigma_i$
$(i, j\in I)$.
Then $\dim\mcM=M^4$
and $\{\sigma_0^{x_0}\sigma_1^{x_1}
\tau_0^{y_0}\tau_1^{y_1}
\,|\,x_0,x_1,y_0,y_1\in\fkJ_{1,M-1}\}$
form the $\bC$-basis of $\mcM$.
Regard $\mcM$ as the Hopf algebra by
\begin{equation*}
\Delta(\sigma_i)=\sigma_i\otimes\sigma_i,\,
\Delta(\tau_i)=\tau_i\otimes\tau_i,\,
\varepsilon(\sigma_i)=\varepsilon(\tau_i)=1,\,
S(\sigma_i)=\sigma_i^{-1},\,S(\tau_i)=\tau_i^{-1}
\, (i\in I).
\end{equation*}

Let $\mcMtrU=\trU\otimes\mcM$,
and regard $\mcMtrU$ as the 
Hopf algebra such that
$\trU$ and $\mcM$ and the Hopf subalgebra of $\mcMtrU$
and the following equations hold:
\begin{equation*}
\sigma_iX\sigma_i^{-1}={\tilde{\omega}}_{i0}^x{\tilde{\omega}}_{i1}^yX,
\,\, \tau_iX\tau_i^{-1}={\tilde{\omega}}_{0i}^{-x}{\tilde{\omega}}_{1i}^{-y}X
\quad (i\in I,\,x,y\in\bZ,\,X\in\trU_{x\al_0+y\al_1}).
\end{equation*}
Let 
\begin{equation*}
G={\frac 1 {M^2}}\sum_{k,r,x,y\in\fkJ_{1,M-1}}
{\tilde{\omega}}^{-ky+rx}\sigma_0^k\sigma_1^r
\otimes \tau_0^x\tau_1^y
\quad (\in\mcMtrU\otimes\mcMtrU).
\end{equation*}
Then $G$ is the universal $R$-matrix of $\mcM$.
We have  \begin{equation*}
G^{-1}={\frac 1 {M^2}}\sum_{k,r,x,y\in\fkJ_{1,M-1}}
{\tilde{\omega}}^{-ky+rx}\sigma_0^k\sigma_1^r
\otimes 
S(\tau_0^x\tau_1^y)={\frac 1 {M^2}}\sum_{k,r,x,y\in\fkJ_{1,M-1}}
{\tilde{\omega}}^{ky-rx}\sigma_0^k\sigma_1^r
\otimes 
\tau_0^x\tau_1^y.
\end{equation*}
Applying the Drinfeld-twist induced  by $G=\sum_ku_k\otimes v_k$, we obtain an alternative Hopf algebra structure $\mcMtrU^{(G)}=(\mcMtrU,\Delta^{(G)},
\Delta^{(G)},S^{(G)})$ on $\mcMtrU$
(see \cite{R90} (see also \cite{D90})).
It is defined by
\begin{equation*}
\Delta^{(G)}(X)=G\Delta(X)G^{-1},\,
\varepsilon^{(G)}(X)=\varepsilon(X),\,
S^{(G)}(X)=\sum_{k,r}u_kS(v_k)S(X)S^2(u_r)v_r
\quad (X\in\trU).
\end{equation*}
We call $\mcMtrU^{(G)}$ the {\it{twisting of 
$\mcMtrU$ by $G$}}.

\begin{lemma}
Let $Q^\prime$ be the one obtained from $Q$
of \eqref{eqn:defQ}
by replacing $\ha$ by $\ha{\tilde{\omega}}^{-2}$.
Then there exists a unique
Hopf algebra monomorphism
$\varphi:\trU_{q,Q^\prime}\to\mcMtrU^{(G)}$ such that
\begin{equation*}
\varphi(\trK_i)=\trK_i\sigma_i^{-1}\tau_i^{-1},\,
\varphi(\trL_i)=\trL_i\sigma_i^{-1}\tau_i^{-1},\,
\varphi(\trE_i)=\trE_i\tau_i^{-1},\,
\varphi(\trF_i)=\sigma_i^{-1}\trF_i
\,\,(i\in I).
\end{equation*}
\end{lemma}

\begin{remark}
Define ${\mathcal{M}}\trmfkU^{(G)}$ in the same way as $\mcMtrU^{(G)}$, with $\trmfkU$ in place of $\trU$.
Let $\trmfkR^{(G)}=G_{21}\cdot \trmfkR\cdot G^{-1}$.
By \cite{R90}, $\trmfkR^{(G)}$ satisfies equations (for $({\mathcal{M}}\trmfkU^{(G)}, {\hat{\Delta}}^{(G)})$) similar
to those of 
\eqref{eqn:maintwoEqThree}-\eqref{eqn:maintwoEqFive}.
\end{remark}

\begin{remark}
Assume $a=-q^2$ and ${\tilde{\omega}}={\sqrt{-1}}$.
Then $\rho$ for $N=2$ (see Section~\ref{section:minusonecase})
cannot be extended to the algebra homomorphism
from $\mcMtrU^{(G)}$.
In particular, ${\mathcal{R}}(z)$ of {\rm{(${\mathcal{R}}2$)}}
cannot be obtained from the universal $R$-matrix of $U_q(\hat{sl_2})$ using $\trmfkR^{(G)}$.
\end{remark}

\section{$\hbar$-version} \label{section:hadic}
Let $\bC[[\hbar]]$ be the
$\bC$-algebra of formal power series in 
$\hbar$.
Let $\bC((\hbar))$ be the field
of the formal Laurent series in 
$\hbar$, that is, $\bC((\hbar))$ can be considered as the field of fractions of $\bC[[\hbar]]$.

Let $z\in\bCt$ and 
\begin{equation} \label{eqn:Pmatrix}
P=(p_{ij})_{i,j\in I}=\left(\begin{array}{ccc} 
2 & -2+z  \\
-2-z & 2 
\end{array}\right)
\in{\mathrm{M}}_2(\bC).
\end{equation}
If $z\ne 0$,
Let ${\tilde{P}}=({\tilde{p}}_{ij})_{i,j\in
I}=P^{-1}={\frac 1 {z^2}}
\left(\begin{array}{ccc} 
2 & 2-z  \\
2+z & 2 
\end{array}\right)$.
Let ${\breve{q}}_{ij}=\exp(p_{ij}\hbar)$
$(i,j\in I)$.
We use the $\hbar$-adic topology for a $\bC[[\hbar]]$-algebra.
Let $U^{{\mathfrak{h}}}_\hbar=U^{{\mathfrak{h}}}_{\hbar,P}({\hat {sl_2}})$ be
the topological $\bC[[\hbar]]$-algebra 
topologically presented by the generators 
${\breve{H}}_i$, ${\breve{X}}_i$, 
${\breve{E}}_i$, ${\breve{F}}_i$ $(i\in I)$
and the relations (c.f. \cite{GG25}):
\begin{align}
& \left\{\begin{array}{ll}
-z{\breve{X}}_0=(4+z){\breve{H}}_0+4{\breve{H}}_1,\,\,
z{\breve{X}}_1=4{\breve{H}}_0+(4-z){\breve{H}}_1 
& \quad (\mbox{if $z\ne 0$}), \\
{\breve{H}}_0={\breve{X}}_0
=-{\breve{H}}_1=-{\breve{X}}_1
& \quad (\mbox{if $z=0$})),
\end{array}\right.
\\
& [{\breve{H}}_i,{\breve{H}}_j]=0,\,
[{\breve{H}}_i,{\breve{E}}_j]=p_{ij}{\breve{E}}_j,\,
[{\breve{H}}_i,{\breve{F}}_j]=-p_{ij}{\breve{F}}_j,\\
& [{\breve{E}}_i,{\breve{F}}_j]=
\delta_{ij}{\frac {\exp(\hbar{\breve{H}}_i)-\exp(-\hbar{\breve{X}}_i)} {\exp(\hbar)-\exp(-\hbar)}}, \\
&{\breve{E}}_i^3{\breve{E}}_j-{\breve{q}}_{ij}(3)_{{\breve{q}}_{ii}}{\breve{E}}_i^2{\breve{E}}_j{\breve{E}}_i+{\breve{q}}_{ii}{\breve{q}}_{ij}^2(3)_{{\breve{q}}_{ii}}{\breve{E}}_i{\breve{E}}_j{\breve{E}}_i^2-{\breve{q}}_{ii}^3{\breve{q}}_{ij}^3{\breve{E}}_j{\breve{E}}_i^3=0
\quad (i\ne j), \label{eqn:hdefEiEj} \\
&{\breve{F}}_i^3{\breve{F}}_j-{\breve{q}}_{ji}(3)_{{\breve{q}}_{ii}}{\breve{F}}_i^2{\breve{F}}_j{\breve{F}}_i+{\breve{q}}_{ii}{\breve{q}}_{ji}^2(3)_{{\breve{q}}_{ii}}{\breve{F}}_i{\breve{F}}_j{\breve{F}}_i^2-{\breve{q}}_{ii}^3{\breve{q}}_{ji}^3{\breve{F}}_j{\breve{F}}_i^3=0
\quad (i\ne j). \label{eqn:hdefFiFj} 
\end{align}
We regard  $U^{{\mathfrak{h}}}_\hbar$ as the topological
Hopf algebra with
$\Delta({\breve{H}}_i)={\breve{H}}_i\otimes 1+1\otimes {\breve{H}}_i$, 
$\Delta({\breve{E}}_i)={\breve{E}}_i\otimes 1
+\exp(\hbar{\breve{H}}_i)\otimes {\breve{E}}_i$, 
$\Delta({\breve{F}}_i)={\breve{F}}_i\otimes \exp(-\hbar{\breve{H}}_i)
+1\otimes {\breve{F}}_i$ $(i\in I)$.
Let ${\breve{q}}=\exp(\hbar)\in\bC((\hbar))$
and 
${\breve{Q}}
=({\breve{q}}_{ij})_{ij\in I}$
$(\in{\mathrm{M}}_2(\bC((\hbar))))$.
Let ${\breve{U}}={\breve{U}}_{{\breve{q}},{\breve{Q}}}({\hat {sl_2}})$ be the Hopf $\bC((\hbar))$-algebra defined in the same way as that for 
$\trUqQ({\hat {sl_2}})$
with $\bC((\hbar))$, 
${\breve{q}}$, 
${\breve{Q}}$
in place of 
$\bC$, 
$q$, 
$Q$.
Then we have the $\bC((\hbar))$-algebra
homomorphism 
${\breve{\varphi}}:{\breve{U}}\to U^{{\mathfrak{h}}}_\hbar$
with 
${\breve{\varphi}}(K_i)=
\exp(\hbar{\breve{H}}_i)$, 
${\breve{\varphi}}(L_i)=
\exp(-\hbar{\breve{X}}_i)$, 
${\breve{\varphi}}(E_i)=
\exp(\hbar{\breve{E}}_i)$, 
${\breve{\varphi}}(F_i)=
-(\exp(\hbar)-\exp(-\hbar))\exp(\hbar{\breve{F}}_i)$
$(i\in I)$.
Let $U^{{\mathfrak{h}}}_\hbar{\bar{\otimes}}U^{{\mathfrak{h}}}_\hbar$ be the topological completion of 
$U^{{\mathfrak{h}}}_\hbar\otimes_{\bC[[\hbar]]}U^{{\mathfrak{h}}}_\hbar$, 
and let 
$U^{{\mathfrak{h}}}_\hbar{\tilde{\otimes}}U^{{\mathfrak{h}}}_\hbar$ be the completion 
of $U^{{\mathfrak{h}}}_\hbar{\bar{\otimes}}U^{{\mathfrak{h}}}_\hbar$
similar to \eqref{eqqn;multicomp}.
Let $t_0$ be $\sum_{i,j\in I}{\tilde{p}}_{ij}{\breve{H_i}}\otimes{\breve{H_i}}$
(resp. ${\frac 1 2}{\breve{H}}_1\otimes {\breve{H}}_1$)
if $z\ne 0$
(resp. $z= 0$).
Let
\begin{equation*}
\trmfkR^{{\mathfrak{h}}}=
\left(\sum_{\lambda\in\bZgeqo\Pi} ({\breve{\varphi}}\otimes{\breve{\varphi}})(\trmfkC_\lambda)\right)
\cdot\exp(-\hbar t_0)
\quad (\in U^{{\mathfrak{h}}}_\hbar{\tilde{\otimes}}U^{{\mathfrak{h}}}_\hbar).
\end{equation*}
Then $\trmfkR^{{\mathfrak{h}}}$
is the universal $R$-matrix of $U^{{\mathfrak{h}}}$,
which means that it satisfies the equations similar to
those of Theorem~\ref{theorem:maintwo}.

\vspace{1cm}
\noindent
{\bf{Acknowledgments}}\,
Regarding this work,
H. Yamane was partially supported by JSPS Grant-in-Aid for Scientific Research
(C) 25K06931. We also used Mathematica 14.3 \cite{Wolf14-3} and its earlier versions.

\vspace{1cm}
\noindent
{\bf{Data Availability}}\,
No datasets were generated or analyzed during this work.

\vspace{1cm}
\noindent
{\bf{Declarations}}

\vspace{1cm}
\noindent
{\bf{Conflict of interest}}\,
The authors declare that they have no conflict of interest.

\end{document}